\documentclass[reqno]{amsart}
\usepackage{amsmath,amssymb}
\usepackage{mathptmx}      % use Times fonts if available on your TeX system
\usepackage{mathrsfs}
\usepackage{latexsym}
\usepackage{amssymb}
\usepackage{bm}
\usepackage[usenames,dvipsnames]{xcolor}
\usepackage{graphicx,pifont}
\usepackage{multicol,multirow}
\usepackage{ulem,cancel}
\usepackage{tikz}
\usetikzlibrary[patterns]
\usepackage{comment}
\usepackage{array}
\usepackage{subcaption}
\usepackage{txfonts}
\usepackage{mathtools}
\usepackage[shortlabels]{enumitem}
\usepackage{hyperref}
\usepackage{booktabs}
\usepackage{cleveref}
\crefname{equation}{}{}
\crefname{enumi}{}{}
\crefname{figure}{Figure}{Figure}
\crefname{subsection}{Subsection}{Subsections}
\crefname{lemma}{Lemma}{Lemma}
\crefname{theorem}{Theorem}{Theorem}
\crefname{proposition}{Proposition}{Proposition}
\crefname{section}{Section}{Section}

\newtheorem{theorem}{Theorem}[section]
\newtheorem{lemma}[theorem]{Lemma}

\theoremstyle{definition}

\theoremstyle{remark}
\newtheorem{remark}[theorem]{Remark}
\numberwithin{equation}{section}
\numberwithin{figure}{section}

\newcommand{\bs}[1]{\boldsymbol{#1}}

\newcommand{\bb}[1]{{\mbox{\sffamily \bfseries #1}}}
\newcommand{\oname}[1]{\textrm{#1}}
\newcommand{\abs}[1]{\left|#1\right|}
\newcommand{\nrm}[1]{\left|\left|#1\right|\right|}
\newcommand{\eqdef}{\stackrel{\mathrm{def}}{=\joinrel=}}

\newcommand{\pp}[2]{\frac{\partial {#1}}{\partial {#2}}}

\newcommand{\phf}[1]{{{#1}+1/2}}
\newcommand{\phfg}[2]{{{#1}+{#2}/2}}

\newcommand{\mhf}[1]{{{#1}-1/2}}

\newcommand{\mhfg}[2]{{{#1}-{#2}/2}}

\newcommand{\cfl}{{\textrm{cfl}}}

\newcommand{\ave}{\textrm{ave}}

\begin{document}

\title[HV Method for Euler Equations]{An Explicit Fourth-Order Hybrid-Variable Method for Euler Equations with A Residual-Consistent Viscosity}

\author[X.~Zeng]{Xianyi Zeng}
\address{Department of Mathematics, Lehigh University, Bethlehem, PA 18015, United States.}
\email[Corresponding author, X.~Zeng]{xyzeng@lehigh.edu}

\date{\today}

%\subjclass[2020]{65M06 \and 65M08 \and 65M22 \and 35Q31}


\begin{abstract}
%My abstract: 150 -- 250 words.
In this paper we present a formally fourth-order accurate hybrid-variable method for the Euler equations in the context of method of lines.
The hybrid-variable (HV) method seeks numerical approximations to both cell-averages and nodal solutions and evolves them in time simultaneously; and it is proved in previous work that these methods are inherent superconvergent.
Taking advantage of the superconvergence, the method is built on a third-order discrete differential operator, which approximates the first spatial derivative at each grid point, only using the information in the two neighboring cells.
Stability and accuracy analyses are conducted in the one-dimensional case for the linear advection equation; whereas extension to nonlinear systems including the Euler equations is achieved using characteristic decomposition and the incorporation of a residual-consistent viscosity to capture strong discontinuities.
Extensive numerical tests are presented to assess the numerical performance of the method for both 1D and 2D problems.
\end{abstract}

\maketitle

\section{Introduction}
\label{sec:intro}
Numerical methods for hyperbolic conservation laws in general and the Euler equations in particular have been a focal research area in the past decades. 
It has always been a challenging task, because discontinuity usually develops even with smooth initial data, and the correct speed for the shock wave can only be captured if the method is conservative.
To this end, methods that are second-order in region where solutions are smooth and first-order near discontinuities~\cite{JPBoris:1976a,BvanLeer:1979a,STZalesak:1979a,FShakib:1991a} have been extremely popular in both academic and industrial applications related to fluid mechanics.

In recent years, there has been growing interests in constructing higher-order schemes for Euler equations, at least when the solutions are smooth, while maintaining the numerical conservation and being robust near discontinuities.
Examples include, but are not limited to, the class of weighted essentially non-oscillatory (WENO) methods~\cite{GSJiang:1996a} and their variants, the discrete Galerkin (DG) schemes~\cite{BCockburn:1998a} and the more recent hybridizable discrete Galerkin (HDG) methods~\cite{NCNguyen:2012a}, as well as the high-order residual distribution schemes~\cite{RAbgrall:2019a}.
High-order schemes are especially preferred over lower-order ones due to their ability of resolving small or sub-grid scale features in the flow, such as those due to the interaction between shock waves and/or contact discontinuities.

High-order methods generally demand more computation costs than lower-order ones.
For example, WENO type methods make use of multiple large stencils to achieve high-order accuracy and numerical stability, which usually ends in large communication cost in parallel computations.
The DG type schemes, in contrast, employ local high-order polynomial reconstruction and thus communication is only required between adjacent mesh cells; however, the localization of discrete operators is usually computationally expensive in the sense that more degrees of freedom are required to achieve the same order of accuracy. %comparing to WENO or finite volume methods.
To get an idea, let us consider a one-dimensional (1D) grid with $N$ uniform cells for scalar problems, then a $p$-th order finite-difference or finite-volume scheme requires $N+1$ or $N$ unknowns, respectively, whereas the discretization of spatial derivative $\partial_x$ requires a stencil of at least $p$ or $p+1$ consecutive cells. 
For a DG type scheme, however, most computations are localized to stay inside each cell, but in order to achieve $p$-th order accuracy one requires $(p+1)$ unknowns per cell, which results in a total of $N(p+1)$ unknowns.

In this work, we attempt to seek a balance between using a larger stencil and using more degrees of freedom per mesh cell to achieve higher-order accuracy for Euler equations; and the methodology extends to other hyperbolic conservation laws naturally.
Particularly, the proposed method is based on the recently developed inherently superconvergent hybrid-variable (HV) discretization framework~\cite{XZeng:2019a}, which seeks numerical approximations to both nodal solutions and cell-averaged solutions and evolves them in time using the method of lines.
In this work we restrict the stencil of each discretization operator to two adjacent cells (in the 1D case), and utilize the superconvergence of HV schemes to achieve the highest possible order of accuracy.
Using the same 1D example as before, the proposed HV scheme will have $2N+1$ unknowns ($N+1$ nodal approximations and $N$ cell averages); and it approximates $\partial_x$ with third-order accuracy: in two neighboring cells, there are five unknowns ($3$ nodal values and $2$ cell-averaged ones), and four in the upwind-biased stencil will be used to construct a discrete differential operator (DDO) that approximates $\partial_x$.
Due to the superconvergence property, the third-order DDO leads to a HV scheme that is fourth-order accurate in space, see the general theory in~\cite{XZeng:2019a} and also the accuracy, stability, and convergence proof in~\cref{sec:anal}.

The DDO itself is a Hermite type operator and it leads to a linearly stable method due to the choice of upwind-biased stencil. 
Although it is reported in the literature that Hermite type methods seem to provide better total variation control comparing to methods that are based on Lagrange interpolation, its mechanism is never fully understood~\cite{JGoodrich:2006a,DAppelo:2012a}. 
To this end, we employ the classical idea of artificial viscosity as a discontinuity capturing operator to enhance the robustness of the method near shocks.
In particular, we develop a residual-consistent viscosity that is inspired by the ones utilizing an entropy residual~\cite{CJohnson:1990a,JLGuermond:2014a,JLGuermond:2016a,VStiernstrom:2021a}, so that the formal fourth-order of accuracy is maintained in smooth-solution region and it reduces to the classical von Neumann-Richtmyer viscosity~\cite{JVonNeumann:1950a} near strong discontinuities.

Lastly, we note that although preliminary numerical results showed without proof that the superconvergence property is valid on both structured and unstructured grids~\cite{XZeng:2014a}, this work will assume Cartesian grids for simplicity; extension of the methodology to simplicial grids and its analysis will be addressed in future work.

The remainder of this paper is organized as follows. 
The proposed fourth-order HV method is described for one-dimensional problems in~\cref{sec:hv}, and accuracy and stability analysis are provided in~\cref{sec:anal}.
In particular, we propose a concept called {\it essential order of accuracy} and utilize it to prove the convergence of the proposed HV scheme (\cref{thm:anal_conv}).
The one-dimensional method for Cauchy problems is completed with the construction of an artificial viscosity that is consistent with the entropy residual, which is employed to detect discontinuities without hurting the formal order of accuracy in smooth-solution regions.
In~\cref{sec:ext}, we generalize the methodology to boundary conditions other than the periodic ones as well as two-dimensional Cartesian grids.
Extensive 1D and 2D benchmark tests are offered in~\cref{sec:num} to assess the numerical performance of the current work.
Finally,~\cref{sec:concl} provides a summary of the proposed method.

\section{A fourth-order HV method for Euler equations, Part I}
\label{sec:hv}
We consider the Euler equations:
\begin{equation}\label{eq:hv_euler}
  \bs{W}_t + \nabla\cdot\bs{F}(\bs{W}) = \bs{0}\;,\quad(\bs{x},\;t)\in\Omega\times[0,\;T]\;,
\end{equation}
where $\Omega\subset\mathbb{R}^d$ is a bounded Lipschitz domain with piecewise smooth boundaries.
The conservative variable $\bs{W}$ and the flux function $\bs{F}$ are given by:
\begin{equation}\label{eq:hv_vars}
  \bs{W} = \begin{bmatrix}
    \rho \\ \rho \bs{v} \\ \rho E
  \end{bmatrix}\quad\textrm{ and }\quad
  \bs{F}(\bs{W}) = \begin{bmatrix}
    \rho \bs{v}^T \\ \rho \bs{v}\bs{v}^T + p\bs{I}_d \\ (\rho E+p)\bs{v}^T
  \end{bmatrix}\;,
\end{equation}
where $\rho$, $\bs{v}$, and $E$ are density, velocity, and specific total energy, respectively; and $\bs{I}_d\in\mathbb{R}^{d\times d}$ is the identity matrix.
The pressure $p$ is determined from other thermodynamic variables including the density and the specific internal energy $e=E-\bs{v}\cdot\bs{v}/2$ by the equation of state:
\begin{equation}\label{eq:hv_eos}
  p = p(\rho,\;e)\;.
\end{equation}
In this work, we assume polytropic gas where $p=(\gamma-1)\rho e$ with $\gamma$ being the specific heat capacity ratio.
The speed of sound $c_s>0$ is defined as $c_s^2 = \frac{\partial p}{\partial \rho}\Big|_s$, where $s$ is the specific entropy that is defined here as:
\begin{equation}\label{eq:hv_etp}
  s = p/\rho^\gamma\;,
  %s = -\log(p/\rho^\gamma)\;,
\end{equation}
and thus the speed of sound is given by:
\begin{equation}\label{eq:hv_sps}
  c_s = \sqrt{\frac{\gamma p}{\rho}}\;.
\end{equation}
Later we shall also use the conservative entropy $S=\rho s=p/\rho^{\gamma-1}$ in the construction of a residual-consistent artificial viscosity.
%Note that we adopt the specific entropy definition of~\cite{AHarten:1983a}, which corresponds to the entropy $S=\rho s=p/\rho^{\gamma-1}$ that symmetrizes the hyperbolic equation.
%Note that we adopt the specific entropy definition of~\cite{AHarten:1983a}, which corresponds to the entropy $S=\rho s=-\rho\log(p/\rho^\gamma)$ that symmetrizes the hyperbolic equation.
%The methodology, however, can be easily extended to other equations of state such as the stiffened gases. 
In this section, we present the core components of the HV method in the context of the Cauchy problem for the one-dimensional (1D) Euler equation.
Then the analysis of the method is provided in~\cref{sec:anal} and other technical issues including non-periodic boundary conditions and extension to two-dimensional (2D) problems are presented in~\cref{sec:ext}.

\subsection{A fourth-order HV method in 1D}
\label{sec:hv_1d}
In the 1D case, let us consider $\Omega=[0,\;L]$ and a grid with $N$ uniform cells $\Omega_j^{j+1}=[x_j\;,x_{j+1}]$, where $0\le j\le N-1$.
Particularly, the grid points are given by $x_j=jh$, $0\le j\le N$, where $h=L/N$; and the cell centers are denoted $x_{\phf{j}}=(x_j+x_{j+1})/2$, $0\le j\le N-1$.
At the semi-discretization level, the HV methods seek approximations to both nodal solutions $\bs{W}_j(t)\approx\bs{W}(x_j,\;t)$ and cell-averaged solutions $\overline{\bs{W}}_{\phf{j}}(t)\approx\frac{1}{h}\int_{x_j}^{x_{j+1}}\bs{W}(x\;,t)dx$.

The ODE for updating $\overline{\bs{W}}_{\phf{j}}$ can be derived naturally by integrating the governing equation on $\Omega_j^{j+1}$:
\begin{equation}\label{eq:hv_1d_cell}
  \overline{\bs{W}}_{\phf{j}}' + \frac{1}{h}\left(\bs{F}(\bs{W}_{j+1})-\bs{F}(\bs{W}_j)\right) = \bs{0}\;;
\end{equation}
the ODE for updating $\bs{W}_j$, however, is obtained by linearizing the flux term and a characteristic decomposition:
\begin{equation}\label{eq:hv_1d_node}
  \bs{W}_j' + \bs{R}_j\bs{\Lambda}_j[\mathcal{D}_x(\bs{R}_j^{-1}\bs{W})]_j = \bs{0}\;,
\end{equation}
where $\bs{J}_j = \bs{R}_j\bs{\Lambda}_j\bs{R}_j^{-1}$ is the eigenvalue decomposition of the Jacobian matrix $\bs{J}_j\eqdef\pp{\bs{F}(\bs{W}_j)}{\bs{W}}$.
In the case of Euler equations, they can be derived in closed form thus no numerical eigenvalue decomposition is required.
The discrete differential operator (DDO) $[\mathcal{D}_x\bullet]_j$ designates an HV approximation to the spatial derivative at $x_j$ that utilizes nearby cell-averaged and nodal solutions; it is to be specified later.

In~\cref{eq:hv_1d_node}, the chosen DDO is applied to each characteristics as some sort of upwinding in its stencil is required for stability.
In this work, we consider the shortest stencil in the construction of $[\mathcal{D}_x]_j$, i.e., using information located in the interval $[x_{j-1},\;x_{j+1}]$.
For illustration, let $w(x)$ be a sufficiently smooth function with $w_j\approx w(x_j)$ and $\overline{w}_{\phf{j}}\approx\frac{1}{h}\int_{x_j}^{x_{j+1}}w(x)dx$, then we construct $[\mathcal{D}_xw]_j$ using the three nodal values $w_{j-1}$, $w_j$, and $w_{j+1}$, and two cell-averaged values $\overline{w}_{\mhf{j}}$ and $\overline{w}_{\phf{j}}$.
In particular, we define the leftward biased operator $[\mathcal{D}_x^-]$ and the rightward biased operator $[\mathcal{D}_x^+]$ respectively as:
\begin{align}
  \label{eq:hv_1d_ddo_l}
  [\mathcal{D}_x^-w]_j &= \frac{1}{h}\left(w_{j-1}-\frac{7}{2}\overline{w}_{\mhf{j}}+2w_j+\frac{1}{2}\overline{w}_{\phf{j}}\right)\;, \\
  \label{eq:hv_1d_ddo_r}
  [\mathcal{D}_x^+w]_j &= \frac{1}{h}\left(-\frac{1}{2}\overline{w}_{\mhf{j}}-2w_j+\frac{7}{2}\overline{w}_{\phf{j}}-w_{j+1}\right)\;.
\end{align}
Then by comparing the leading terms in the Taylor series expansions one sees that both operators are {\it third-order} accurate in space in the sense that if the exact nodal and cell-averaged values are denoted by $w_j^\star$ and $\overline{w}_{\phf{j}}^\star$, respectively, then:
\begin{displaymath}
  [\mathcal{D}_x^-w^\star]_j = w'(x_j) + O(h^3)\;,\quad
  [\mathcal{D}_x^+w^\star]_j = w'(x_j) + O(h^3)\;.
\end{displaymath}

To this end, the proposed semi-discretization method (without shock capturing operator for now) is given by~\cref{eq:hv_1d_cell} and~\cref{eq:hv_1d_node} with $[\mathcal{D}_x]$ being applied to each component of $\bs{R}_j^{-1}\bs{W}$ specified as the one of the two operators~\cref{eq:hv_1d_ddo_l} and~\cref{eq:hv_1d_ddo_r}, depending on the sign of the corresponding eigenvalue.
More specifically, let us consider the construction of $[\mathcal{D}_x(\bs{R}_{j_0}^{-1}\bs{W})]_{j_0}$ for some fixed $j_0$.
Then we define the local characteristics $\bs{Y}=\bs{R}_{j_0}^{-1}\bs{W}$ and thus compute $\bs{Y}_j=\bs{R}_{j_0}^{-1}\bs{W}_j$ and $\overline{\bs{Y}}_{\phf{j}}=\bs{R}_{j_0}^{-}\overline{\bs{W}}_{\phf{j}}$ for $j$ nearby $j_0$.
Denoting the components of $\bs{Y}$ by $Y_{,k}$, $1\le k\le d+2$, and those of $\bs{Y}_j$ and $\overline{\bs{Y}}_{\phf{j}}$ by $Y_{j,k}$ and $\overline{Y}_{\phf{j},k}$, respectively, the $k$-th component of $[\mathcal{D}_x(\bs{R}_{j_0}^{-1}\bs{W})]_{j_0}$ is then defined as:
\begin{equation}\label{eq:hv_1d_ddo}
  [\mathcal{D}_x(\bs{R}_{j_0}^{-1}\bs{W})]_{j_0,k} = 
  [\mathcal{D}_xY_{,k}]_{j_0} \eqdef \left\{
  \begin{array}{lcl}
    \left[\mathcal{D}^-_xY_{,k}\right]_{j_0} & & \textrm{if}\quad \lambda_{j_0,k}>0\;, \\ \vspace*{-.05in} \\
    \left[\mathcal{D}^+_xY_{,k}\right]_{j_0} & & \textrm{if}\quad \lambda_{j_0,k}\le0\;.
  \end{array}\right.
\end{equation}
where $\lambda_{j_0,k}$ is the $k$-th diagonal entry of $\bs{\Lambda}_{j_0}$.

Although both $[\mathcal{D}^-_x]$ and $[\mathcal{D}^+_x]$ are third-order DDOs, they result in a fourth-order method.
More details will be provided in the next section for model advection equations.

\subsection{Artificial viscosity}
\label{sec:hv_av}
To capture the strong discontinuities that usually appear in nonlinear hyperbolic conservation laws, we adopt the strategy of artificial viscosity and solve the regularized equation:
\begin{equation}\label{eq:hv_av}
  \bs{W}_t + \bs{F}(\bs{W})_x - (\nu_h\bs{A}\bs{W}_x)_x = 0\;,
\end{equation}
where $\bs{A}$ is a scaling matrix and $\nu_h>0$ is the artificial viscosity that shrinks to zero as $h\to0$.
In the case of Euler equations, $\bs{A}$ is computed as:
\begin{equation}\label{eq:hv_av_scale}
  \bs{A} = \begin{bmatrix}
    1 & 0 & 0 \\ -v & 1 & 0 \\ -E & 0 & 1
  \end{bmatrix}\;,
\end{equation}
so that the total energy conservation is ensured.
While the details of computing $\nu_h$ will be left until~\cref{sec:ev}, we briefly discuss the discretization of the diffusion term here.
Integrating~\cref{eq:hv_av} over $\Omega_j^{j+1}$, one obtains an ODE for $\overline{\bs{W}}_{\phf{j}}$:
\begin{equation}\label{eq:hv_av_semi_cell}
  \overline{\bs{W}}_{\phf{j}}' + \frac{1}{h}\left(\bs{F}(\bs{W}_{j+1})-\bs{F}(\bs{W}_j)\right) - \frac{1}{h}\left[\nu_{j+1}\bs{A}_{j+1}[\mathcal{D}_x^c\bs{W}]_{j+1}-\nu_j\bs{A}_j[\mathcal{D}_x^c\bs{w}]_j\right] = \bs{0}\;.
\end{equation}
Here $\nu_j$ is the local artificial viscosity at $x_j$ and $\bs{A}_j = \bs{A}(\bs{W}_j)$; and $[\mathcal{D}_x^c\bs{W}]_j$ is a central difference operator defined as:
\begin{equation}\label{eq:hv_av_dx_c}
  [\mathcal{D}_x^c\bs{W}]_j = \frac{1}{h}\left(\overline{\bs{W}}_{\phf{j}}-\overline{\bs{W}}_{\mhf{j}}\right)\;.
\end{equation}
The ODE for updating $\bs{W}_j$ is obtained by linearizing the flux term and the diffusion term:
\begin{equation}\label{eq:hv_av_semi_node}
  \bs{W}_j' + \bs{R}_j\bs{\Lambda}_j[\mathcal{D}_x(\bs{R}_j^{-1}\bs{W})]_j - \nu_j\left[\bs{A}_j[\mathcal{D}_{xx}^c\bs{W}]_j+\left(\pp{\bs{A}(\bs{W}_j)}{\bs{W}}:[\mathcal{D}_x^c\bs{W}]_j\right)[\mathcal{D}_x^c\bs{W}]_j\right] = \bs{0}\;.
\end{equation}
Here we freeze the diffusion coefficient to $\nu_j$ for simplicity; the operator $[\mathcal{D}_{xx}^c\bs{W}]_j$ approximates $\bs{W}_{xx}$ at $x_j$ using hybrid data with a centered stencil:
\begin{equation}\label{eq:hv_av_dxx_c}
  [\mathcal{D}_{xx}^c\bs{W}]_j = \frac{1}{h^2}\left[3\overline{\bs{W}}_{\mhf{j}}-6\bs{W}_j+3\overline{\bs{W}}_{\phf{j}}\right]\;;
\end{equation}
Lastly, $\bb{T} = \partial\bs{A}(\bs{W})/\partial\bs{W}$ is a third-order tensor; in particular denoting the components of $\bs{A}$ by $A_{ab}$ and that of $\bs{W}$ by $W_c$, the components of $\bb{T}$ are $T_{abc}=\partial A_{ab}/\partial W_c$.
The operator $\bullet:\bullet$ designates the contraction between the last index of a third-order tensor and a vector.

\subsection{Time marching scheme}
\label{sec:hv_time}
To match the spatial order of accuracy, we use the classical fourth-order explicit Runge-Kutta method (RK4)~\cite{EHairer:1993a} to march the solution in time.
Specifically, let us denote the entire solution vector by $\bs{U}$ (including both nodal and cell-averaged solutions) and denote ODE system~\cref{eq:hv_av_semi_cell} and~\cref{eq:hv_av_semi_node} by $\bs{U}'+\mathcal{R}(\bs{U})=\bs{0}$.
Let $t_n$ and $t_{n+1}=t_n+\Delta t_n$ be two consecutive time steps, we update the numerical solution from $\bs{U}^n\approx\bs{U}(t_n)$ to $\bs{U}^{n+1}\approx\bs{U}(t_{n+1})$ by:
\begin{align}
  \notag
  \bs{U}^{(1)} &= \bs{U}^n - \frac{1}{2}\Delta t_n\mathcal{R}(\bs{U}^n)\;, \\
  \notag
  \bs{U}^{(2)} &= \bs{U}^n - \frac{1}{2}\Delta t_n\mathcal{R}(\bs{U}^{(1)})\;, \\
  \notag
  \bs{U}^{(3)} &= \bs{U}^n - \Delta t_n\mathcal{R}(\bs{U}^{(2)})\;, \\
  \label{eq:hv_time_rk4}
  \bs{U}^{n+1} &= \bs{U}^n - \frac{1}{6}\Delta t_n\mathcal{R}(\bs{U}^n) - \frac{1}{3}\Delta t_n\mathcal{R}(\bs{U}^{(1)}) - \frac{1}{3}\Delta t_n\mathcal{R}(\bs{U}^{(2)}) - \frac{1}{6}\Delta t_n\mathcal{R}(\bs{U}^{(3)})\;.
\end{align}
The last missing piece, the time step size, is computed following the Courant condition:
\begin{equation}\label{eq:hv_time_dt}
  \Delta t_n = \frac{\alpha_{\cfl} h}{\max\left(\max_jv_{\max}(\bs{W}_j^n),\;\max_jv_{\max}(\overline{\bs{W}}^n_{\phf{j}})\right)}\;.
\end{equation}
Here $v_{\max}(\bs{W})=\abs{v(\bs{W})} + c_s(\bs{W})$ is the maximum characteristic speed given the fluid state $\bs{W}$, with $v(\cdot)$ and $c_s(\cdot)$ being the velocity and speed of sound, respectively; and $\alpha_{\cfl}\le0.808$ is the threshold derived from linear stability analysis using the model advection equation~\cite{XZeng:2019a}.
In all computations of this work, we adopt a smaller Courant number $\alpha_{\cfl}=0.6$ to account for the extra viscosity term.

As a last remark, on the one hand, the analysis of the HV method especially the superconvergence property in the next section assumes exact initial data for both the nodal values and cell-averaged values.
On the other hand, the initial cell-averaged values are approximated using a fourth-order Gauss-Legendre quadrature rule in all numerical tests in this work.
That is, let the initial data be given analytically by $\bs{W}(x,\;0) = \bs{W}_0(x)$, and we assume $\bs{W}_0(x)$ can be computed at any $x$ using some black-box procedure, then the initial nodal values and cell-averaged value are computed by:
\begin{subequations}\label{eq:hv_time_ic}
  \begin{align}
    \label{eq:hv_time_ic_node}
    \bs{W}^0_j &= \bs{W}_0(x_j)\;, \\
    \label{eq:hv_time_ic_cell}
    \overline{\bs{W}}^0_{\phf{j}} &= \frac{1}{2}\bs{W}_0\left(x_{\phf{j}}-\frac{1}{2\sqrt{3}}h\right)+\frac{1}{2}\bs{W}_0\left(x_{\phf{j}}+\frac{1}{2\sqrt{3}}h\right)\;,
  \end{align}
\end{subequations}
whether a closed-form formula for computing $\overline{\bs{W}}$ exists or not.
%Here we use the Gauss-Legendre quadrature rule to approximate the cell integral up to fourth-order of accuracy.

\section{Stability, accuracy, and convergence analysis}
\label{sec:anal}
In this section we show the superconvergence property of the proposed semi-discretized HV method for model Cauchy problem of advection equations, as well as a convergence proof in the case of smooth solutions.
To this end, let us consider the advection equation with a constant advection velocity $\lambda>0$ and periodic boundary condition:
\begin{equation}\label{eq:anal_adv}
  \left\{\begin{array}{lcl}
    w_t + \lambda w_x = 0\;, & & (x,t)\in [0,\;L]\times[0,\;T]\;, \\ \vspace*{-.1in} \\
    w(0,t) = w(L,t)\;, & & t\in[0,\;T]\;, \\ \vspace*{-.1in} \\
    w(x,0) = w_0(x)\;, & & x\in[0,\;L]\;.
  \end{array}\right.
\end{equation}
Adopting the convention that $w_j=w_{j+N}$ and $\overline{w}_{\phf{j}}=\overline{w}_{\phf{j+N}}$, etc., the semi-discretized HV method gives rise to the ODE system of $\overline{w}_{\phf{j}}(t)$ and $w_j(t)$:
\begin{equation}\label{eq:anal_adv_ode}
  \left\{\begin{array}{lcl}
    \overline{w}'_{\phf{j}}+\frac{\lambda}{h}(w_{j+1}-w_j) = 0\;, & & \forall j\;,\quad t\in[0,\;T]\;, \\ \vspace*{-.1in} \\
    w_j'+\frac{\lambda}{h}\left(w_{j-1}-\frac{7}{2}\overline{w}_{\mhf{j}}+2w_j+\frac{1}{2}\overline{w}_{\phf{j}}\right) = 0\;, & & \forall j\;,\quad t\in[0,\;T]\;.
  \end{array}\right.
\end{equation}
First, we show the superconvergence of the method in the sense of von Neumann analysis, i.e., when the initial condition is given by a simple wave.
\begin{theorem}\label{thm:anal_sc_sw}
  Suppose $w^\star(x,t)$ is the exact solution to the problem~\cref{eq:anal_adv} with the simple wave initial condition:
  \begin{equation}\label{eq:anal_adv_sw_ic}
    w_0(x) = e^{i\kappa x}\;,
  \end{equation}
  where $\kappa=2m\pi/L$ for some interger $m$.
  Let $w_j(t)$ and $\overline{w}_{\phf{j}}(t)$ solve the ODE system~\cref{eq:anal_adv_ode} with exact initial data:
  \begin{equation}\label{eq:anal_adv_ode_ic}
    w_j(0) = w^\star(x_j,0)\;,\quad
    \overline{w}_{\phf{j}} = \frac{1}{h}\int_{x_j}^{x_{j+1}}w^\star(x,0)dx\;.
  \end{equation}
  Then there is the estimates:
  \begin{equation}\label{eq:anal_adv_est}
    \abs{w_j(T)-w^\star(x_j,T)} \le Ch^4\ \textrm{ and }\ 
    \abs{\overline{w}_{\phf{j}}(T)-\frac{1}{h}\int_{x_j}^{x_{j+1}}w^\star(x,T)dx} \le Ch^4\;,\quad
    \forall j\;,
  \end{equation}
  where $C$ is a constant that only depends on $L$, $T$, $\lambda$, and $\kappa$.
\end{theorem}
\begin{proof}
  We attempt to find a solution to~\cref{eq:anal_adv_ode} and~\cref{eq:anal_adv_ode_ic} in the form:
  \begin{equation}\label{eq:anal_adv_ode_sol}
    w_j(t) = A(t)e^{i\kappa x_j}\;,\quad
    \overline{w}_{\phf{j}}(t) = \frac{\overline{A}(t)}{i\kappa h}\left(e^{i\kappa x_{j+1}}-e^{i\kappa x_j}\right)\;,
  \end{equation}
  for some continuous functions $A(t)$ and $\overline{A}(t)$ with initial values being unity.
  Plugging~\cref{eq:anal_adv_ode_sol} into~\cref{eq:anal_adv_ode} and writing $\theta=\kappa h$, one obtains a system of two ODEs that are independent of $j$:
  \begin{equation}\label{eq:anal_adv_amp_mat}
    \frac{d}{dt}\begin{bmatrix}
      \overline{A}(t) \\ A(t)
    \end{bmatrix} = 
    -\frac{\lambda}{h}\bs{C}_{\theta}
    \begin{bmatrix}
      \overline{A}(t) \\ A(t)
    \end{bmatrix}\;,\quad
    \bs{C}_{\theta} = \begin{bmatrix}
      0 & i\theta \\
      \frac{-8+7e^{-i\theta}+e^{i\theta}}{i2\theta} & 2+e^{-i\theta}
    \end{bmatrix}\;,\quad
    \begin{bmatrix}
      \overline{A}(0) \\ A(0)
    \end{bmatrix} = 
    \begin{bmatrix}
      1 \\ 1
    \end{bmatrix}\;,
  \end{equation}
  whose solution is given by:
  \begin{equation}\label{eq:anal_adv_amp_sol}
    \begin{array}{l}
      \overline{A}(t) = \frac{\mu_2-i\theta}{\mu_2-\mu_1}e^{-\frac{\lambda}{h}\mu_1t} + \frac{i\theta-\mu_1}{\mu_2-\mu_1}e^{-\frac{\lambda}{h}\mu_2t}\;, \\ \vspace*{-.1in} \\
      A(t) = \frac{\mu_1(\mu_2-i\theta)}{i\theta(\mu_2-\mu_1)}e^{-\frac{\lambda}{h}\mu_1t} + \frac{\mu_2(i\theta-\mu_1)}{i\theta(\mu_2-\mu_1)}e^{-\frac{\lambda}{h}\mu_2t}\;.
    \end{array}
  \end{equation}
  Here $\mu_1$ and $\mu_2$ are the two roots of the equation:
  \begin{equation}\label{eq:anal_adv_char}
    \mu^2-(2+e^{-i\theta})\mu + 4-\frac{7}{2}e^{-i\theta}-\frac{1}{2}e^{i\theta} = 0\;.
  \end{equation}
  Let $b(\theta) = 2+e^{-i\theta}$ and $a(\theta) = \frac{-8+7e^{-i\theta}+e^{i\theta}}{i2\theta}$, by the Taylor theorem one gets:
  \begin{displaymath}
    a(\theta)+b(\theta) = i\theta + c_1(\theta)\theta^4\;,
  \end{displaymath}
  where $c_1(\theta)$ is an analytic and bounded function of $\theta$; particularly that exists a constant $C_1>0$ such that $\abs{c_1(\theta)}<C_1$ for all $\theta\in\mathbb{R}$.
  Using this notation, one gets:
  \begin{displaymath}
    \mu_{1,2} = \frac{1}{2}\left[b(\theta)\pm\sqrt{b(\theta)^2-4i\theta a(\theta)}\right] = \frac{1}{2}\left[b(\theta)\pm\sqrt{(b(\theta)-i2\theta)^2+i4c_1(\theta)\theta^5}\right]\;.
  \end{displaymath}
  Next, we use a fact (proved later) that there exists a continuous and bounded function $c_2(\theta)$, such that $\abs{c_2(\theta)}<C_2$ for some constant $C_2>0$ and:
  \begin{equation}\label{eq:anal_adv_sqrt}
    (b(\theta)-i2\theta+c_2(\theta)\theta^5)^2 = (b(\theta)-i2\theta)^2+i4c_1(\theta)\theta^5\;.
  \end{equation}
  Then, the two eigenvalues $\mu_{1,2}$ can be computed as:
  \begin{equation}\label{eq:anal_adv_eigs}
    \mu_1 = i\theta-\frac{1}{2}c_2(\theta)\theta^5\;,\quad
    \mu_2 = b(\theta)-i\theta+\frac{1}{2}c_2(\theta)\theta^5\;.
  \end{equation}
  
  To show the existence of such a $c_2(\theta)$, one notices that~\cref{eq:anal_adv_sqrt} is equivalent to:
  \begin{equation}\label{eq:anal_adv_c2}
    2(b(\theta)-i2\theta)c_2(\theta) + c_2(\theta)^2\theta^5 = i4c_1(\theta)\;.
  \end{equation}
  Letting $\theta\to0$, one obtains $c_2(0) = \frac{i4c_1(0)}{2b(0)} = \frac{i2c_1(0)}{3}$ is finite; hence~\cref{eq:anal_adv_c2} admits a continuous $c_2(\theta)$ for $\theta\in\mathbb{R}$ by the implicit value theorem.
  To show its boundedness, one just needs to use the fact that $c_2(\theta)\to0$ as $\theta\to\infty$ due to the boundedness of $b(\theta)$ and $c_1(\theta)$.

  Defining $\Delta\mu=\mu_2-\mu_1$ and plugging~\cref{eq:anal_adv_eigs} into~\cref{eq:anal_adv_amp_sol}, one gets:
  \begin{align*}
    \overline{A}(t) =& \frac{\Delta\mu-\frac{c_2}{2}\theta^5}{\Delta\mu}e^{-\frac{\lambda}{h}\left(i\theta-\frac{c_2}{2}\theta^5\right)t} + \frac{\frac{c_2}{2}\theta^5}{\Delta\mu}e^{-\frac{\lambda}{h}\mu_2t}\;, \\
    \textrm{and}\quad A(t) =& \frac{(1+i\frac{c_2}{2}\theta^4)(\Delta\mu-\frac{c_2}{2}\theta^5)}{\Delta\mu}e^{-\frac{\lambda}{h}\left(i\theta-\frac{c_2}{2}\theta^5\right)t} - \frac{i\frac{c_2}{2}\theta^4(\Delta\mu-\frac{c_2}{2}\theta^5)}{\Delta\mu}e^{-\frac{\lambda}{h}\mu_2t}\;.
  \end{align*}
  By \cref{lm:anal_sc_sw_diff} below, there exists a constant $C_3>0$ such that $\abs{\Delta\mu}>C_3$ for all $\theta\in\mathbb{R}$; furthermore by \cref{lm:anal_sc_sw_stab}, $\oname{Re}\,\mu_{1,2}\ge0$ for all $\theta$.
  First let us derive an estimate on $\overline{A}(t)$ and to this end obtain:
  \begin{align*}
    \abs{\overline{A}(t)-e^{-i\lambda\kappa t}} \le & \abs{e^{-\frac{\lambda}{h}\left(i\theta-\frac{c_2}{2}\theta^5\right)t}-e^{-i\lambda\kappa t}} + \abs{\frac{\frac{c_2}{2}\theta^5}{\Delta\mu}e^{-\frac{\lambda}{h}\mu_1t}} + \abs{\frac{\frac{c_2}{2}\theta^5}{\Delta\mu}e^{-\frac{\lambda}{h}\mu_2t}} \le \abs{e^{\frac{\lambda c_2}{2h}\theta^5t}-1} + \frac{C_2}{C_3}\abs{\theta}^5\;,
  \end{align*}
  and next we focus on estimating the first term in the right hand side.
  On the one hand: 
  \begin{displaymath}
    \oname{Re}\frac{\lambda c_2}{2h}\theta^5t = \oname{Re}\frac{\lambda}{h}(i\theta-\mu_1)=-\oname{Re}\frac{\lambda}{h}\mu_1\le0\quad\Rightarrow\quad 
    \abs{e^{\frac{\lambda c_2}{2h}\theta^5t}-1} \le 2,\ \forall\theta\;;
  \end{displaymath}
  and on the other hand, fixing any $\theta_0>0$ we obtain for all $-\theta_0<\theta<\theta_0$, by the mean value theorem:
  \begin{displaymath}
    \abs{e^{\frac{\lambda c_2}{2h}\theta^5t}-1} = \abs{e^{\frac{1}{2}\lambda c_2\kappa\theta^4t}-1} = \abs{\frac{1}{2}\lambda c_2\kappa\theta^4te^{\frac{1}{2}\xi\lambda c_2\kappa\theta^4t}} \le \frac{1}{2}\lambda C_2\abs{\kappa}T\abs{\theta}^4e^{\frac{1}{2}\lambda C_2\abs{\kappa}\theta_0^4T}\;,
  \end{displaymath}
  where $0<\xi<1$.
  Combining the two, one obtains that for all $\theta\in\mathbb{R}$:
  \begin{displaymath}
    \abs{e^{\frac{\lambda c_2}{2h}\theta^5t}-1} \le C_4h^4\;,
  \end{displaymath}
  where $C_4>0$ is a constant that is given by:
  \begin{displaymath}
    C_4 = \max\left(\frac{2\kappa^4}{\theta_0^4},\;\frac{\lambda C_2\abs{\kappa}^5T}{2}e^{\frac{1}{2}\lambda C_2\abs{\kappa}\theta_0^4T}\right)\;.
  \end{displaymath}
  Finally, we obtain the desired estimate on $\overline{A}(t)$ as:
  \begin{displaymath}
    \abs{\overline{A}(t)-e^{-i\lambda\kappa t}} \le C_4h^4+\frac{C_2\abs{\kappa}^5}{C_3}h^5 \le C_5h^4\;,
  \end{displaymath}
  with $C_5=C_4+C_2\abs{\kappa}^5L/C_3$, which gives rise to:
  \begin{displaymath}
    \abs{\overline{w}_{\phf{j}}(T)-\frac{1}{h}\int_{x_j}^{x_{j+1}}w^\star(x,T)dx} = \abs{\overline{A}(t)-e^{-i\lambda\kappa t}}\abs{\frac{e^{i\kappa x_{j+1}}-e^{i\kappa x_j}}{i\kappa h}} \le C_5h^4\;.
  \end{displaymath}
  Here we used the mean value theorem again to see that $\abs{(e^{i\kappa x_{j+1}}-e^{i\kappa x_j})/(i\kappa h)}\le1$.

  Following the same procedure, we can derive a similar bound on $\abs{A(t)-e^{-i\lambda\kappa t}}$ that leads to the first inequality of~\cref{eq:anal_adv_est}; the details are omitted here.
\end{proof}

Next we prove the two lemmas in the previous proof.
\begin{lemma}\label{lm:anal_sc_sw_diff}
  Let $\mu_{1,2}$ be the two roots of~\cref{eq:anal_adv_char}, then there exists a constant $C>0$ such that $\abs{\mu_1-\mu_2}>C$ for all $\theta\in\mathbb{R}$.
\end{lemma}
\begin{proof}
  Since $\mu_1 + \mu_2 = 2+e^{-i\theta}$ and $\mu_1\mu_2 = 4-\frac{7}{2}e^{-i\theta}-\frac{1}{2}e^{i\theta}$, one has
  \begin{align*}
    &(\mu_1-\mu_2)^2 = e^{-2i\theta} + 18e^{-i\theta} - 12 + 2e^{i\theta} \\
    \Rightarrow\quad&\abs{\mu_1-\mu_2}^4 = P(\cos\theta)\;,\quad P(x) \eqdef 425 - 456x + 96x^2 + 16x^3\;.
  \end{align*}
  It is trivial to verify that $P(x)$ is decreasing on $[-1,\;1]$ hence $P(\cos\theta)\ge P(1) = 81$; it follows that $\abs{\mu_1-\mu_2}\ge3$ for all $\theta$.
\end{proof}
\begin{lemma}\label{lm:anal_sc_sw_stab}
  Let $\mu_{1,2}$ be the two roots of~\cref{eq:anal_adv_char}, then $\oname{Re}\,\mu_{1,2}\ge0$ for all $\theta\in\mathbb{R}$.
\end{lemma}
\begin{proof}
  If $\cos\theta=1$, it is easy to verify that the two roots are $0$ and $3$; hence we assume $\cos\theta<1$ below.
  Let $\theta$ be arbitrary such that $\cos\theta<1$, we define the polynomials:
  \begin{align*}
    P(z) =& z^2-(2+e^{-i\theta})z+4-\frac{7}{2}e^{-i\theta}-\frac{1}{2}e^{i\theta}\;, \\
    P^\ast(z) =& (-z)^2-\overline{(2+e^{-i\theta})}(-z)+\overline{4-\frac{7}{2}e^{-i\theta}-\frac{1}{2}e^{i\theta}} = z^2+(2+e^{i\theta})z+4-\frac{7}{2}e^{i\theta}-\frac{1}{2}e^{-i\theta}\;.
  \end{align*}
  Then using a classical result by E. Frank~\cite[Theorem 2.1]{EFrank:1947a}, if for some complex number $\xi$ with $\oname{Re}\,\xi>0$ there is $|P^\ast(\xi)|>|P(\xi)|$, then $P(z)$ has one more zero in the open right complex plane than the polynomial $P_1(z)$ defined as:
  \begin{displaymath}
    P_1(z) = \frac{P^\ast(\xi)P(z)-P(\xi)P^\ast(z)}{z-\xi} = (4+2\cos\theta)\xi z-6i\sin\theta(z+\xi)-10+8\cos\theta+2\cos^2\theta\;.
  \end{displaymath}
  Hence we just need to find $\xi$, such that (i) $\oname{Re}\,\xi>0$, (ii) $|P^\ast(\xi)|>|P(\xi)|$, and (iii) $z=\frac{6\xi i\sin\theta+10-8\cos\theta-2\cos^2\theta}{(4+2\cos\theta)\xi-6i\sin\theta}$ has positive real part.
  To this end, it is elementary to verify that $\xi=1$ satisfies all three requirements; particularly for the last one there is:
  \begin{displaymath}
    \oname{Re}\,z = \frac{4(1-\cos\theta)^3}{(4+2\cos\theta)^2+36\sin^2\theta} > 0\;,
  \end{displaymath}
  due to the assumption that $\cos\theta<1$.
\end{proof}

Regarding arbitrarily initial data with sufficient smoothness, we can show that the method is {\it essentially} fourth-order in the sense given by the next theorem.
\begin{theorem}\label{thm:anal_esc}
  Suppose $w^\star(x,t)$ is the exact solution to the problem~\cref{eq:anal_adv} with $C^p$ initial condition in the sense that $w^{\textrm{ext}}_0(x,0)\in C^p(\mathbb{R})$, where $w^{\textrm{ext}}_0$ is the $L$-periodic extension of $w_0$; and we assume $p\ge3$.
  Let $w_j(t)$ and $\overline{w}_{\phf{j}}(t)$ be the solutions to~\cref{eq:anal_adv_ode} with exact initial data computed according to~\cref{eq:anal_adv_ode_ic}, then for any $\varepsilon>0$ there exists a constant $C>0$ that only depends on $L$, $T$, $\lambda$, $w_0$, and $\varepsilon$, such that for all $j$: 
  \begin{subequations}\label{eq:anal_esc_est}
    \begin{align}
      \label{eq:anal_esc_est_n}
      &\abs{w_j(T)-w^\star(x_j,T)} \le Ch^4+\varepsilon \\
      \label{eq:anal_esc_est_c}
      \textrm{and}\quad
      &\abs{\overline{w}_{\phf{j}}(T)-\frac{1}{h}\int_{x_j}^{x_{j+1}}w^\star(x,T)dx} \le Ch^4+\varepsilon\;.
    \end{align}
  \end{subequations}
\end{theorem}
\begin{proof}
  As $w^\star(x,0)=w_0(x)\in C^p([0,\;L]) \subset L^2([0,\;L])$, it has the Fourier series expansion:
  \begin{displaymath}
    w^\star(x,0) = \sum_ma_me^{i\kappa_m x}\;,\quad \kappa_m \eqdef \frac{2m\pi}{L}\;,
  \end{displaymath}
  where the summation runs over all integers.
  Thus the exact solutions $w^\star(x,T)$ and the solutions to the ODEs~\cref{eq:anal_adv_ode} are given by:
  \begin{align*}
    &w^\star(x_j,T) = \sum_ma_me^{-i\lambda\kappa_mT}e^{i\kappa_m x_j}\,,\ 
    &&\frac{1}{h}\int_{x_j}^{x_{j+1}}w^\star(x,T)dx = \sum_m\frac{a_me^{-i\lambda\kappa_mT}}{i\kappa_m h}(e^{i\kappa_mx_{j+1}}-e^{i\kappa_mx_j})\;; \\
    &w_j(T) = \sum_ma_mA_m(T)e^{i\kappa_mx_j}\,,\ 
    &&\overline{w}_{\phf{j}}(T) = \sum_m\frac{a_m\overline{A}_m(T)}{i\kappa_mh}(e^{i\kappa_mx_{j+1}}-e^{i\kappa_mx_j})\;.
  \end{align*}
  Here $A_m(t)$ and $\overline{A}_m(t)$ are given by~\cref{eq:anal_adv_amp_sol} corresponding to $\kappa=\kappa_m$.

  As the $L$-periodic extension of $w_0(x,0)$ is $C^3$ (since $p\ge3$), a classical estimate on the Fourier coefficients gives for all $m\ne0$ and $m_0>0$:
  \begin{displaymath}
    \abs{a_m} \le \frac{L^3}{8m^3\pi^3}\nrm{w'''_0(x)}_{\infty}\;\Rightarrow\;
    \sum_{\abs{m}>m_0}\abs{a_m} \le \frac{L^3\nrm{w'''_0(x)}_{\infty}}{8m_0^2\pi^3}\ \textrm{ and }\ 
    \sum_{\abs{m}>m_0}\abs{ma_m} \le \frac{L^3\nrm{w'''_0(x)}_{\infty}}{4m_0\pi^3}\;.
  \end{displaymath}
  Because the quadratic function~\cref{eq:anal_adv_char} has coefficients that are periodic in $\theta$, there exists a $C_1>0$ such that $\abs{\mu_{1,2}}\le C_1$ for all $\theta\in\mathbb{R}$.
  Then by the first equation of~\cref{eq:anal_adv_amp_sol}, \cref{lm:anal_sc_sw_diff}, and \cref{lm:anal_sc_sw_stab}, we have:
  \begin{displaymath}
    \abs{\overline{A}_m(t)} \le \frac{2(C_1+\abs{\kappa_m}h)}{3}\;. 
  \end{displaymath}
  To get a similar bound on $A_m(t)$, we note that $\lim_{\theta\to0}\mu_1(\theta)/\theta = i$, hence there exists a constant $C_2>0$ such that $\abs{\mu_1}\le C_2\abs{\theta}$. 
  Thus using the second equation of~\cref{eq:anal_adv_amp_sol} and the same two lemmas we obtain:
  \begin{displaymath}
    \abs{A_m(t)} \le \frac{C_2(C_1+\abs{\kappa_m}h)}{3} + \frac{C_1(1+C_2)}{3}\;.
  \end{displaymath}
  Hence there exists $C_3, C_4>0$ such that for all $m$ and $t$:
  \begin{equation}\label{eq:anal_esc_a}
    \abs{\overline{A}_m(t)},\;\abs{A_m(t)} \le C_3 + C_4\abs{m}h \le C_3 + C_4\abs{m}L\;.
  \end{equation}

  At the end of the proof for~\cref{thm:anal_sc_sw} we showed that there exists a $C_{5,m}>0$ that depends on $L$, $T$, $\lambda$, and $\kappa_m$, such that $\abs{\overline{A}_m(t)-e^{-i\lambda\kappa_mT}}, \abs{A_m(t)-e^{-i\lambda\kappa_mT}} \le C_{5,m}h^4$ for all $h>0$.
  These estimates lead to:
  \begin{align*}
    \abs{w_j(T)-w^\star(x_j,T)} &\le \sum_{\abs{m}\le m_0}\abs{a_m}\abs{A_m(T)-e^{-i\lambda\kappa_mT}} + \sum_{\abs{m}>m_0}\abs{a_m}\left(\abs{A_m(T)}+1\right) \\
    &\le \sum_{\abs{m}\le m_0}\abs{a_m}C_{5,m}h^4+\sum_{\abs{m}>m_0}\abs{a_m}(1+C_3+C_4\abs{m}h) \\
    &\le \sum_{\abs{m}\le m_0}\abs{a_m}C_{5,m}h^4 + \frac{L^3(1+C_3)\nrm{w'''_0(x)}_{\infty}}{8m_0^2\pi^3} + \frac{L^4C_4\nrm{w'''_0(x)}_{\infty}}{4m_0\pi^3}\;.
  \end{align*}
  To this end, given any $\varepsilon>0$ we may fix an $m_0>0$ such that:
  \begin{displaymath}
    m_0 > \frac{L^4C_4\nrm{w'''_0(x)}_\infty}{2\varepsilon\pi^3}\quad\textrm{ and }\quad
    m_0^2 > \frac{L^3(1+C_3)\nrm{w'''_0(x)}_\infty}{4\varepsilon\pi^3}\;,
  \end{displaymath}
  and define $C=\sum_{\abs{m}\le m_0}\abs{a_m}C_{5,m}>0$.
  Then both $C$ and $m_0$ only depend on $L$, $T$, $\lambda$, $w_0$, and $\varepsilon$, and for all $j$ there is:
  \begin{displaymath}
    \abs{w_j(T)-w^\star(x_j,T)} < Ch^4 + \frac{\varepsilon}{2} + \frac{\varepsilon}{2} = Ch^4 + \varepsilon\;.
  \end{displaymath}

  Lastly regarding an estimate on $\overline{w}_{\phf{j}}(t)$, we use the fact that $\abs{e^{i\kappa_mx_{j+1}}-e^{i\kappa_mx_j}}\le\abs{i\kappa_mh}$ for all $m$, $h$, and $j$, and obtain the estimate:
  \begin{align*}
    \abs{\overline{w}_{\phf{j}}(t)-\frac{1}{h}\int_{x_j}^{x_{j+1}}w^\star(x,T)dx} &\le \sum_{\abs{m}\le m_0}\abs{a_m}\abs{\overline{A}_m(t)-e^{-i\lambda\kappa_mT}} + \sum_{\abs{m}>m_0}\abs{a_m}\left(\abs{\overline{A}_m(t)}+1\right) \\
    &\le Ch^4+\varepsilon\;.
  \end{align*}
  Here $\varepsilon>0$ is arbitrary and $m_0$ and $C$ are defined the same way as before.
\end{proof}

A direct consequence of~\cref{thm:anal_esc} is the convergence of the HV solution to the exact one given $C^3$ initial data as $N\to\infty$, where $N$ is the number of cells:
\begin{theorem}\label{thm:anal_conv}
  Let $w^\star(x,t)$, $w_j(t)$, and $\overline{w}_{\phf{j}}(t)$ be defined the same way as in previous theorem, then for any $\varepsilon_0>0$, there exists an $N_0>0$ that only depends on $L$, $T$, $\lambda$, $w^\star$, and $\varepsilon_0$,  such that for all $N>N_0$ the corresponding semi-discretized HV solution satisfies:
  \begin{equation}\label{eq:anal_conv_est}
    \abs{w_j(T)-w^\star(x_j,T)} \le \varepsilon_0\quad\textrm{and}\quad 
    \abs{\overline{w}_{\phf{j}}(T)-\frac{1}{h}\int_{x_j}^{x_{j+1}}w^\star(x,T)dx} \le \varepsilon_0\;,\quad 
    \forall j\;.
  \end{equation}
\end{theorem}
\begin{proof}
  Choose $\varepsilon=\varepsilon_0/2$ in~\cref{thm:anal_esc} and let the correponding constant in~\cref{eq:anal_esc_est} be denoted $C$.
  Then for all $N_0>(2CL^4/\varepsilon_0)^{1/4}$, one has~\cref{eq:anal_conv_est} following~\cref{eq:anal_esc_est}.
  Because $C$ depends on $L$, $T$, $\lambda$, $w^\star$, and $\varepsilon_0$, so is $N_0$.
\end{proof}

\section{Residual-consistent viscosity in the HV framework}
\label{sec:ev}
The last piece in the 1D HV method of~\cref{sec:hv} is the calculation of the artificial viscosity $\nu_h$ in~\cref{eq:hv_av}, which is described here.
There are two major considerations while designing $\nu_h$: first it should provide sufficient diffusion near discontinuities so that spurious oscillations are suppressed, and second the magnitude of $\nu_h$ should be sufficiently small when the solution is smooth so that the former order of accuracy is not affected.
To this end we construct a residual-based viscosity that is motivated by the entropy viscosity by Guermond et al~\cite{JLGuermond:2014a,JLGuermond:2016a} for finite element methods and a recent residual-based viscosity by Stiernstr\"{o}m~\cite{VStiernstrom:2021a} in the context of finite difference methods.

\subsection{Cellular entropy residual}
\label{sec:rcv_res}
The entropy $S=S(\bs{W})=\rho s$ satisfies the inequality $S(\bs{W})_t+\nabla\cdot(\bs{v}(\bs{W})S(\bs{W}))\le0$
%The entropy $S=S(\bs{W})$ satisfies the inequality:
%\begin{equation}\label{eq:rcv_res_ineq}
%  S(\bs{W})_t + \nabla\cdot\bs{\eta}(\bs{W}) \le 0\;,
%  %S(\bs{W})_t + \nabla\cdot\bs{\eta}(\bs{W}) \ge 0\;,
%\end{equation}
with the equality hold whenever the solution is smooth.
Hence it has been widely used as an indicator for discontinuities in compressible flows.

Note that when the flow is smooth, one has the equivalent advective equation for the specific entropy: $s_t+\bs{v}\cdot\nabla s=0$.
To this end, we consider the entropy residual for $s$, which in the one-dimensional case and at the semi-discretized level is given by:
\begin{equation}\label{eq:rcv_res_semi}
  \oname{Res}_{\phf{j}}s = \frac{d\overline{s}_{\phf{j}}}{dt} + \overline{v}_{\phf{j}}\frac{s(\bs{W}_{j+1}) - s(\bs{W}_j)}{h}\;,
\end{equation}
where $\overline{s}_{\phf{j}}=s(\overline{\bs{W}}_{\phf{j}})$ and $\overline{v}_{\phf{j}}=v(\bs{W}_{\phf{j}})$ are the specific entropy and velocity computed from the cell-averaged solution, respectively.
%The entropy flux is given by $\bs{\eta}(\bs{W}) = S(\bs{W})\bs{v}(\bs{W})$.

%In the one-dimensional case, the entropy flux will be denoted $\eta(\bs{W})$ as it is a scalar. 
%For its construction we define for each cell $\Omega_j^{j+1}$ a semi-discretized residual of the entropy $S$:
%\begin{equation}\label{eq:rcv_res_semi}
%  \oname{Res}_{\phf{j}}S = \frac{d\overline{S}_{\phf{j}}}{dt} + \frac{1}{h}\left[\eta(\bs{W}_{j+1}) - \eta(\bs{W}_j)\right]\;,
%\end{equation}
%where $\overline{S}_{\phf{j}} = S(\overline{\bs{W}}_{\phf{j}})$.
To construct a fully-discretized residual $\oname{Res}_{\phf{j}}^ns$, the spatial part is replaced by $\overline{v}(\bs{W}_{\phf{j}}^n)\times (s(\bs{W}_{j+1}^n)-s(\bs{W}_j^n))/h$ and the time-derivative is approximated by the standard Backward Difference Formula (BDF) with the BDF order being chosen as $\min(n,2)$.

Because the entropy function $s(\cdot)$ and the velocity $v(\cdot)$ do not commute with the integral operator, thus~(\ref{eq:rcv_res_semi}) is only second-order accurate in space (and this is why we only use a BDF with up to second-order accuracy in time).
As we will see later, however, this choice will not affect for formal order of accuracy of the resulting method.

\subsection{Construction of the artificial viscosity}
\label{sec:rcv_vh}
At the beginning of the time step $t^n$, we compute an artificial viscosity $\nu_{\phf{j}}^n$ for each cell, average them to obtain a nodal viscosity (on the domain boundaries, the nodal value is taken as the nearest cell value), and fix the nodal viscosities throughout the Runge-Kutta stages.
In particular, let the cellular residual $\oname{Res}_{\phf{j}}^ns$ be computed as before, we compute $\nu_{\phf{j}}^n$ by:
\begin{equation}\label{eq:rcv_vh}
  \nu_{\phf{j}}^n = \frac{h v_{\max}(\overline{\bs{W}}_{\phf{j}}^n)}{4}(\mathcal{M}\!\circ\!\mathcal{S}\!\circ\!\mathcal{N})(Z_{\phf{j}}^n),\ Z_{\phf{j}}^n = \frac{h\abs{\oname{Res}_{\phf{j}}^ns}/v_{\max}(\overline{\bs{W}}_{\phf{j}}^n)}{\delta s^n + \delta S^n/\overline{\rho}^n + \epsilon\overline{s}^n}.
\end{equation}
%\begin{equation}\label{eq:rcv_vh}
%  \nu_j^n = \min\left(\frac{1}{2}h\left(v(\bs{W}_j^n)+c_s(\bs{W}_j^n)\right),\;\frac{\alpha_\nu h^2\abs{\oname{Res}_j^nS}}{\delta S^n+\epsilon\overline{S}^n}\cdot\frac{\abs{\Delta S_j^n}}{\delta S^n+\epsilon\overline{S}^n}\right)\;.
%\end{equation}
%{\color{red}
%  \begin{equation}\label{eq:rcv_vh_ver2}
%    \nu_j^n = \frac{1}{2}h\left(v(\bs{W}_j^n)+c_s(\bs{W}_j^n)\right)\mathcal{F}\left(\frac{\alpha_\nu\Delta t_n\abs{\oname{Res}_j^nS}/\alpha_{\cfl}}{\delta S^n+\epsilon\overline{S}^n}\right)\;,\quad\mathcal{F}(z) = \frac{z}{1+z}\;.
%  \end{equation}
%}
Here the first term ($hv_{\max}/4$) in the right hand side of~\cref{eq:rcv_vh} is the standard von Neumann-Richtmyer viscosity~\cite{JVonNeumann:1950a} corresponding to the average mesh size per unknown variable (i.e., $h/2$), the indicator $Z_{\phf{j}}^n$ is a dimensionless quantity computed from the entropy residual, and the three operations $\mathcal{M}$, $\mathcal{S}$, and $\mathcal{N}$ in the sequence of composition are a smoother, a sharpener, and a normalizer, respectively, as will be explained later.
The quantities involved in the computation of $Z_{\phf{j}}^n$ are $\delta s^n\approx\nrm{s-\abs{\Omega}^{-1}\nrm{s}_1}_{\infty}$, $\overline{s}^n\approx\abs{\Omega}^{-1}\nrm{s}_1$, $\delta S^n\approx\nrm{S-\abs{\Omega}^{-1}\nrm{S}_1}_{\infty}$, $\overline{S}^n\approx\abs{\Omega}^{-1}\nrm{S}_1$, and $\overline{\rho}^n\approx\abs{\Omega}^{-1}\nrm{\rho}_1$; and they are computed by:
\begin{displaymath}
  \overline{z}^n = \frac{1}{hN}\sum_{j=0}^{N-1}h\abs{z_{\phf{j}}^n}\;,
  \quad
  \delta z^n = \max\left(\max_j\abs{z_j^n-\overline{z}^n},\;\max_j\abs{z_{\phf{j}}^n-\overline{z}^n}\right)\;,
\end{displaymath}
where $z_{\phf{j}}^n=z(\overline{\bs{W}}_{\phf{j}}^n)$ and $z_j^n=z(\bs{W}_j^n)$ with $z$ stands for $\rho$, $s$, or $S$;
and $\epsilon=10^{-16}$ is a small positive number that prevents division-by-zero.

The three operations mentioned earlier are explained below:
\begin{itemize}
  \item $\mathcal{N}$ rescales all indicators according to:
    \begin{equation}\label{eq:rcv_vh_norm}
      Z_{\phf{j}}^n \mapsto Z_{\phf{j}}^n/\max(1,\max_kZ_{\phf{k}}^n)\;.
    \end{equation}
  \item $\mathcal{S}$ is a function and $\mathcal{S}(Z)$ computes a quantity that fast transits from $0$ to $1$ as $Z$ increases from $0$ to $1$:
    \begin{equation}\label{eq:rcv_vh_sharp}
      \mathcal{S}(Z) = \frac{1}{2} - \frac{1}{2}\cos\left(\pi\min\left(\frac{Z}{z_0},1\right)\right)\;,\quad 0\le Z\le 1\;.
    \end{equation}
  \item $\mathcal{M}$ smooths the sharpened scaling coefficient to avoid abrupt change in the artificial viscosity.
    In this work, we choose a smoother that first spread the largest artificial viscosity to a stencil slightly larger than the stencil of the numerical method, and then apply a linear low-pass filter to molify the resulting field that tends to be composed of step functions:
    \begin{equation}\label{eq:rcv_vh_smth}
      \mathcal{M}  = \mathcal{M}_{\ave}^3\!\circ\!\mathcal{M}_{\max}^2\;,
    \end{equation}
    where the operator $\mathcal{M}_{\max}$ computes $Z_{\phf{j}}\mapsto\max(Z_{\mhf{j}},Z_{\phf{j}},Z_{\phfg{j}{3}})$ and $\mathcal{M}_{\ave}$ computes $Z_{\phf{j}}\mapsto (Z_{\mhf{j}}+4Z_{\phf{j}}+Z_{\phfg{j}{3}})/6$, with obvious modifications near the boundaries such as $Z_{1/2}\mapsto\max(Z_{1/2},Z_{3/2})$ and $Z_{1/2}\mapsto(2Z_{1/2}+Z_{3/2})/3$ in the case of $\mathcal{M}_{\max}$ and $\mathcal{M}_{\ave}$, respectively.
    To see the impact of the smoother, we plot the scaling coefficient without any smoother, after applying $\mathcal{M}_{\max}^2$, and after applying the entire $\mathcal{M}_{\ave}$ in~\cref{fg:rcv_vh_coef}.
    These coefficients are computed based on solving the Sod shock tube problem, see~\cref{sec:num_1d_sod}.
    \begin{figure}\centering
  \includegraphics[width=.45\textwidth]{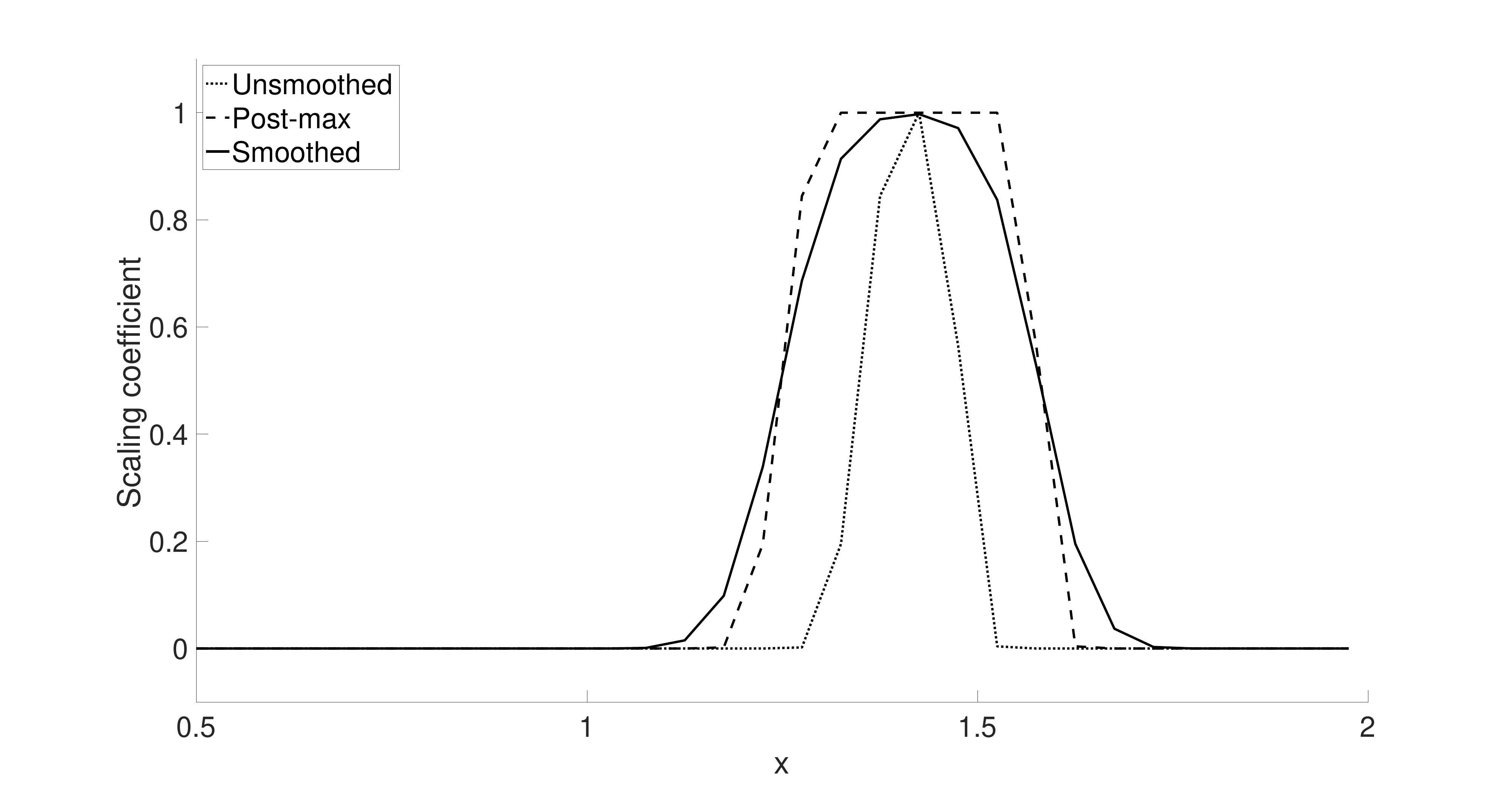}~~
  \includegraphics[width=.45\textwidth]{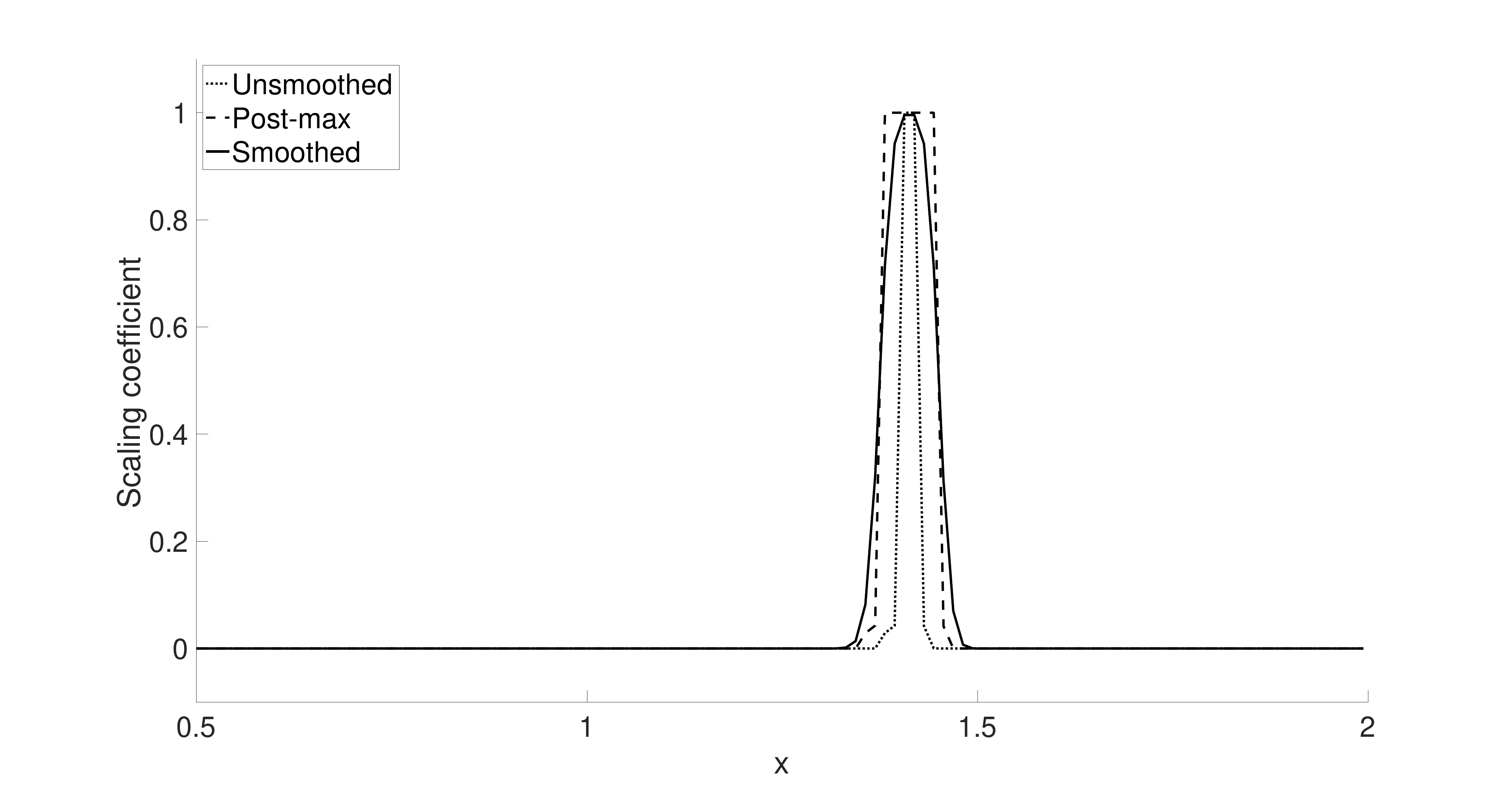}~~
      \caption{The scaling coefficients computed in the Sod shock tube test with $h=1/20$ (left) and $h=1/80$ (right).
      The coefficients $\mathcal{S}\circ\mathcal{N}(Z_{\phf{j}}^n)$, $\mathcal{M}_{\max}^2\circ\mathcal{S}\circ\mathcal{N}(Z_{\phf{j}}^n)$, and $\mathcal{M}\circ\mathcal{S}\circ\mathcal{N}(Z_{\phf{j}}^n)$ are denoted {\it unsmoothed}, {\it post-max}, and {\it smoothed}, respectively.}
      \label{fg:rcv_vh_coef}
    \end{figure}
\end{itemize}

Lastly we briefly estimate the scaling coefficient in various flow conditions and we suppress the script $^n$ for simplicity.
First, if the solution is sufficiently smooth, one has $\oname{Res}_{\phf{j}}=O(h^2)$ and thus $Z_{\phf{j}}=O(h^3)$ in~\cref{eq:rcv_vh}.
Now introducing the notation:
\begin{displaymath}
  Z_{\phf{j}}\xrightarrow{\mathcal{N}}\tilde{Z}_{\phf{j}}\xrightarrow{\mathcal{S}}\hat{Z}_{\phf{j}}\xrightarrow{\mathcal{M}}\check{Z}_{\phf{j}}\;,
\end{displaymath}
one see that:
\begin{displaymath}
  \tilde{Z}_{\phf{j}} = \frac{Z_{\phf{j}}}{\max(1,\max_kZ_{\phf{k}})} \le Z_{\phf{j}}\ \Rightarrow\ \tilde{Z}_{\phf{j}} = O(h^3)\;;
\end{displaymath}
then for sufficiently small $h>0$, there is:
\begin{displaymath}
  \hat{Z}_{\phf{j}} = \frac{1}{2} - \frac{1}{2}\cos\frac{\pi \tilde{Z}_{\phf{j}}}{z_0} = O(h^6)\;;
\end{displaymath}
and finally one easily obtains $\check{Z}_{\phf{j}}=O(h^6)$, which will not affect the formal fourth-order accuracy of the method.
At the beginning of the computation, due to a lower order BDF one has $\check{Z}_{\phf{j}}^0 = O(h^2)$ and $\check{Z}_{\phf{j}}^1 = O(h^4)$ for the first two time steps, respectively, which gives $\nu_{\phf{j}}^0=O(h^3)$ and $\nu_{\phf{j}}^1=O(h^5)$.
Hence we do not expect a reduction in the global order of accuracy, which is indeed confirmed by numerical tests.

Next, in the vicinity of a shock wave, with a small choice of $z_0$ one generally has $Z_{\phf{j}}>z_0$ and consequently $\hat{Z}_{\phf{j}}=1$, indicating the artificial viscosity reduces to the classical one by von Neumann and Richtmyer. 
Lastly, it is clear that $z_0$ plays a key role in the design of the entropy viscosity: on the one hand it decides the strength of shock wave at which the von Neumann-Richtmyer viscosity will be activated, and on the other hand the viscosity does not impact the formal order of accuracy only if $h^3\lesssim z_0$ or $h\lesssim\sqrt[3]{z_0}$.
In this work, a fixed value $z_0=0.04$ is selected for all numerical tests and it gives very good numerical performance.

%and lastly $\Delta S_j^n$ measures local entropy variation:
%\begin{displaymath}
%  \Delta S_j^n = \max\left(S_{j-1}^n,S_{\mhf{j}}^n,S_j^n,S_{\phf{j}}^n,S_{j+1}^n\right)-\min\left(S_{j-1}^n,S_{\mhf{j}}^n,S_j^n,S_{\phf{j}}^n,S_{j+1}^n\right)\;.
%\end{displaymath}
%Here if the node is on the boundary, the range on the right hand side excludes obvious non-existent variables.

\section{A fourth-order HV method for Euler equations, Part II}
\label{sec:ext}

Here we discuss extending the HV method to problems in a more general setting, namely the non-periodic boundary conditions in 1D in~\cref{sec:ext_bc} and the 2D problems on Cartesian grids in~\cref{sec:ext_2d} and~\cref{sec:ext_ev}, respectively.

\subsection{Non-periodic boundary conditions in 1D}
\label{sec:ext_bc}
Constructing $[\mathcal{D}_x]$ at a non-periodic boundary uses a stencil that stays inside the computational domain.
Without loss of generality, we demonstrate this procedure and construct $[\mathcal{D}_xw]_0$ for a scalar function as:
\begin{equation}\label{eq:ext_bc_left}
  [\mathcal{D}_xw]_0 = \frac{1}{h}\left(-4w_0+6\overline{w}_{1/2}-2w_1\right)\;.
\end{equation}
The formula~\cref{eq:ext_bc_left} is a second-order DDO that is one-order lower than the interior operators~\cref{eq:hv_1d_ddo_l} and~\cref{eq:hv_1d_ddo_r}; hence it will not affect the overall order of accuracy of the method.

For general hyperbolic conservation laws, \cref{eq:ext_bc_left} is used to compute the spatial derivative of an out-going characteristics.
To obtain an idea how the boundary condition affect the stability of the method, we consider the advection problem $w_t-w_x=0$ on the domain $[0,\;1]$ with either (1) periodic boundary conditions (Cauchy problem) or (2) homogeneous Dirichlet boundary condition at $x=1$ (initial boundary value problem or IBVP).
Denoting the entire solution vector by $\bs{U}=[\overline{w}_{1/2},\,\cdots,\,\overline{w}_{N-1/2},w_0,\,\cdots,w_{N-1}]$, in both cases the semi-discretized method is a linear ODE system:
\begin{equation}\label{eq:ext_bc_ode}
  \bs{U}' = \frac{1}{h}\bs{D}\bs{U}\;.
\end{equation}
In~\cref{fg:ext_bc_eigs}, we plot and compare the eigenvalue distribution of $\bs{D}$ for the Cauchy problem and the IBVP with two selections of the grid size $N$.
\begin{figure}\centering
  \includegraphics[width=.45\textwidth]{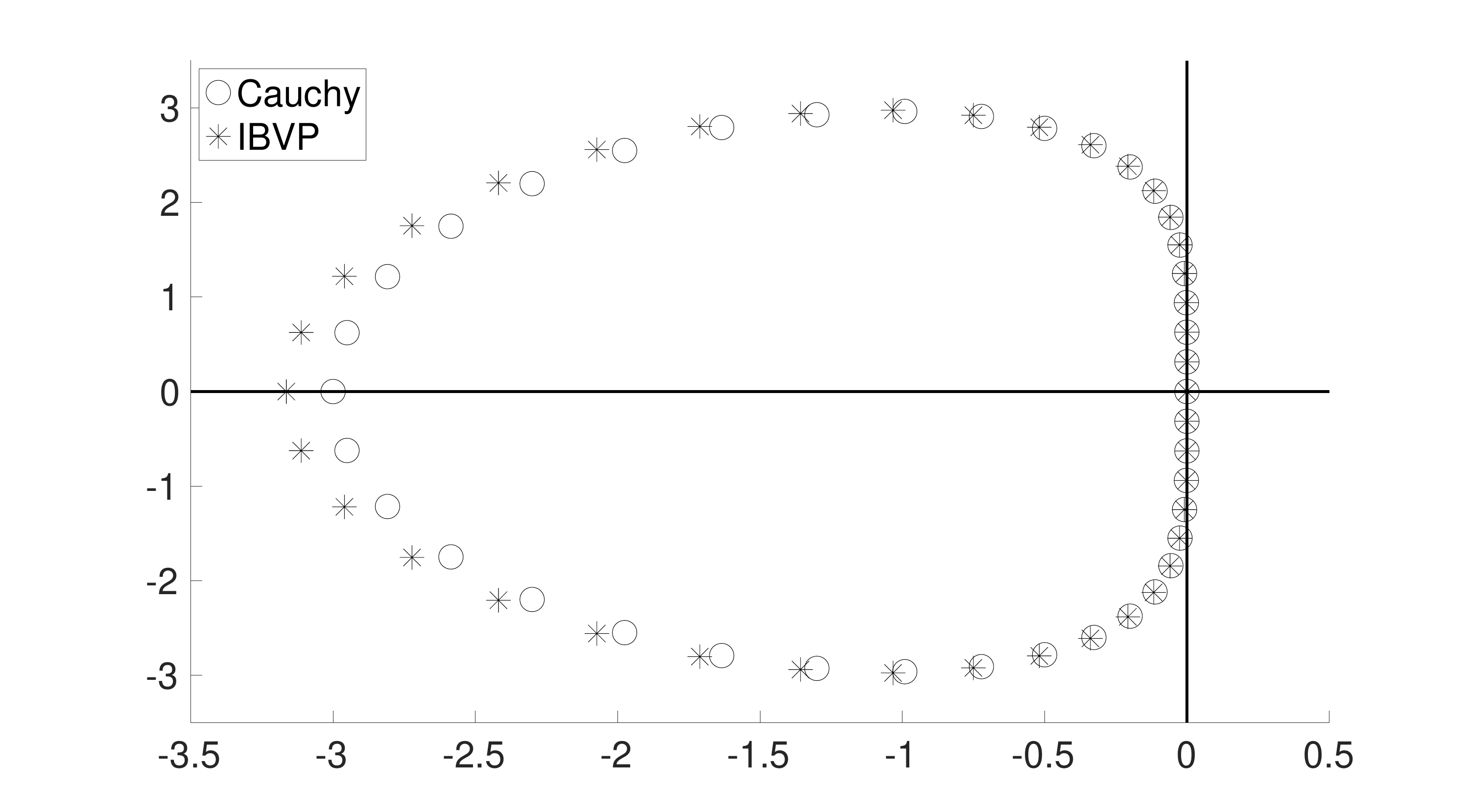}~~
  \includegraphics[width=.45\textwidth]{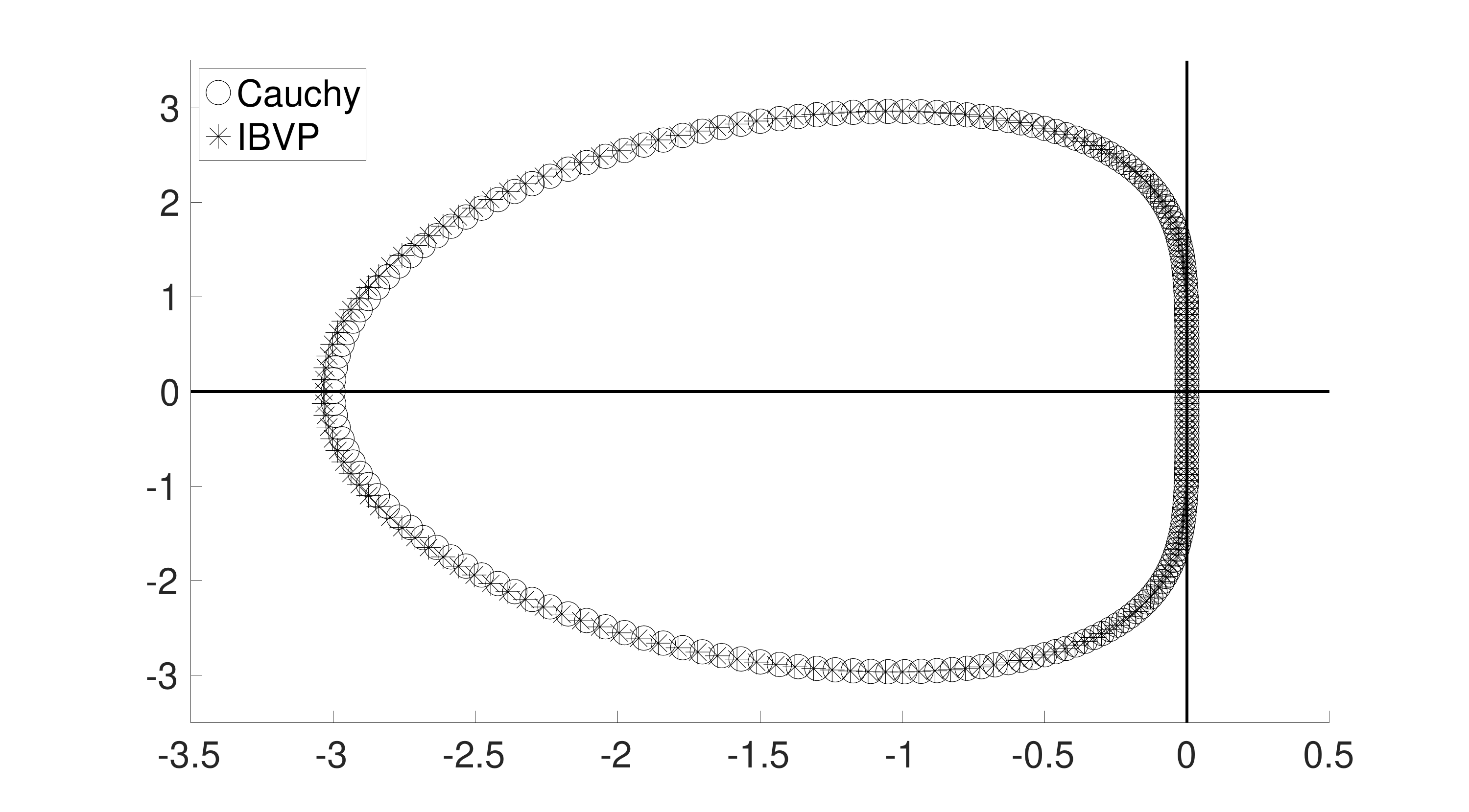}~~
  \caption{Eigenvalue distributions of $\bs{D}$ of~\cref{eq:ext_bc_ode} in the case of the Cauchy problem ($\circ$) and the IBVP ($\ast$) with the grid size equals either $N=20$ (left panel) or $N=100$ (right panel).}
  \label{fg:ext_bc_eigs}
\end{figure}
It appears that the boundary condition~\cref{eq:ext_bc_left} stabilizes the underlying method, and its impact diminishes as $N$ increases.

Lastly in the case of the wall boundary condition (or the slippery wall condition in multiple space dimensions) for the 1D Euler equations, say at $x_0$, we first strongly enforce the zero-velocity condition for $\bs{W}_0$ by setting the momentum to zero and removing its contribution from the total energy:
\begin{equation}\label{eq:ext_bc_wall_node}
  \bs{W}_0 = \begin{bmatrix}
    \rho_0 \\ \rho_0v_0 \\ \rho_0e_0+\frac{1}{2}\rho_0v_0^2
  \end{bmatrix}\quad\longrightarrow\quad
  \bs{W}_0 = \begin{bmatrix}
    \rho_0 \\ 0 \\ \rho_0e_0
  \end{bmatrix}\;,
\end{equation}
then introduce ghost cells that mirror the fluid states near the wall:
\begin{equation}\label{eq:ext_bc_wall_ghst}
  \overline{\bs{W}}_{\mhf{}} = \begin{bmatrix}
    \overline{\rho}_{1/2} \\ -\overline{\rho v}_{1/2} \\ \overline{\rho E}_{1/2}
  \end{bmatrix}\;,\quad
  \bs{W}_{-1} = \begin{bmatrix}
    \rho_1 \\ -\rho_1v_1 \\ \rho_1E_1
  \end{bmatrix}
\end{equation}
 and apply the same interior operator to compute $[\mathcal{D}_x\bs{W}]_0$.

\subsection{A fourth-order HV method for 2D Euler equations on Cartesian grids}
\label{sec:ext_2d}
In extending the methodology to Cartesian grids in multiple space dimensions, we focus on two-dimensional problems in this work although a similar strategy can be used for three-dimensional ones.
To this end, let us denote the grid points by $\bs{x}_{i,j}=(x_i,y_j)=(ih_x,jh_y)$ and the cell centers by $\bs{x}_{\phf{i},\phf{j}}$, respectively; here $h_x$ and $h_y$ are the uniform grid sizes in the $x$-direction and $y$-direction.
We assume that there exists $C_0>1$ such that $h_x/C_0 \le h_y \le C_0h_x$, i.e., the mesh is not severely stretched.  
Denoting the 2D Euler equation in compact form as:
\begin{equation}\label{eq:2d_euler}
  \bs{W}_t + \bs{F}(\bs{W})_x + \bs{G}(\bs{W})_y = \bs{0}\;, 
\end{equation}
the semi-discretized HV method approximates the nodal solutions at grid points and cell averages of each cell:
\begin{equation}\label{eq:2d_var}
  \bs{W}_{i,j}(t) \approx \bs{W}(\bs{x}_{i,j},t)\;,\quad
  \overline{\bs{W}}_{\phf{i},\phf{j}}(t) \approx \frac{1}{h_xh_y}\int_{x_i}^{x_{i+1}}\int_{y_j}^{y_{j+1}}\bs{W}(\bs{x},t)dxdy\;.
\end{equation}
Integrating~\cref{eq:2d_euler} over the cell $[x_i,x_{i+1}]\times[y_j,y_{j+1}]$ one derives the ordinary differential equation for $\overline{\bs{W}}_{\phf{i},\phf{j}}$ as:
\begin{equation}\label{eq:2d_semi_cell}
  \overline{\bs{W}}'_{\phf{i},\phf{j}} + \frac{1}{h_x}\left(\mathcal{J}_j^{j+1}\big|_{i+1}\bs{F}-\mathcal{J}_j^{j+1}\big|_i\bs{F}\right) + \frac{1}{h_y}\left(\mathcal{I}_i^{i+1}\big|_{j+1}\bs{G}-\mathcal{I}_i^{i+1}\big|_j\bs{G}\right) = \bs{0}\;,
\end{equation}
where $\mathcal{J}_j^{j+1}\big|_k\bullet\approx\frac{1}{h_y}\int_{y_j}^{y_{j+1}}\bullet(x_k,y)\,dy$ and $\mathcal{I}_i^{i+1}\big|_k\bullet\approx\frac{1}{h_x}\int_{x_i}^{x_{i+1}}\bullet(x,y_k)\,dx$ approximate the averages of the integral of fluxes along edges.
In this work, we focus on fourth-order methods hence the numerical fluxes for interior edges are computed by:
\begin{subequations}\label{eq:2d_flux_int}
  \begin{align}
    \label{eq:2d_flux_int_x}
    \mathcal{I}_i^{i+1}\big|_k\bs{G} &\eqdef \frac{13}{24}\left[\bs{G}(\bs{W}_{i,k})+\bs{G}(\bs{W}_{i+1,k})\right]-\frac{1}{24}\left[\bs{G}(\bs{W}_{i-1,k}) + \bs{G}(\bs{W}_{i+2,k})\right]\;, \\
    \label{eq:2d_flux_int_y}
    \mathcal{J}_j^{j+1}\big|_k\bs{F} &\eqdef \frac{13}{24}\left[\bs{F}(\bs{W}_{k,j})+\bs{F}(\bs{W}_{k,j+1})\right]-\frac{1}{24}\left[\bs{F}(\bs{W}_{k,j-1}) + \bs{F}(\bs{W}_{k,j+2})\right]\,;
  \end{align}
\end{subequations}
whereas third-order methods are used near the boundaries, for example, at the left and bottom boundaries:
\begin{subequations}\label{eq:2d_flux_bc}
  \begin{align}
    \label{eq:2d_flux_bc_x}
    \mathcal{I}_0^1\big|_k\bs{G} &\eqdef \frac{5}{12}\bs{G}(\bs{W}_{0,k})+\frac{2}{3}\bs{G}(\bs{W}_{1,k})-\frac{1}{12}\bs{G}(\bs{W}_{2,k})\;, \\
    \label{eq:2d_flux_bc_y}
    \mathcal{J}_0^1\big|_k\bs{F} &\eqdef \frac{5}{12}\bs{F}(\bs{W}_{k,0})+\frac{2}{3}\bs{F}(\bs{W}_{k,1})-\frac{1}{12}\bs{F}(\bs{W}_{k,2})\;,
  \end{align}
\end{subequations}
and similar ones are used at the other two boundaries.

The ordinary differential equation for $\bs{W}_{i,j}$ is constructed by:
\begin{equation}\label{eq:2d_semi_node}
  \bs{W}'_{i,j}+\bs{R}_{i,j}\bs{\Lambda}_{i,j}[\mathcal{D}_x\bs{R}_{i,j}^{-1}\bs{W}]_{i,j}+\bs{Q}_{i,j}\bs{\Upsilon}_{i,j}[\mathcal{D}_y\bs{Q}_{i,j}^{-1}\bs{W}]_{i,j} = \bs{0}\;.
\end{equation}
Here $\pp{\bs{F}(\bs{W}_{i,j})}{\bs{W}}=\bs{R}_{i,j}\bs{\Lambda}_{i,j}\bs{R}_{i,j}^{-1}$ and $\pp{\bs{G}(\bs{W}_{i,j})}{\bs{W}}=\bs{Q}_{i,j}\bs{\Upsilon}_{i,j}\bs{Q}_{i,j}^{-1}$ are the eigenvalue decompositions of the $x$-flux function and $y$-flux function, respectively; and the operators $[\mathcal{D}_x]_{i,j}$ and $[\mathcal{D}_y]_{i,j}$ are computed the same way as the DDO $[\mathcal{D}_x]$ described in~\cref{sec:hv} and next.

Let $w$ be a generic scalar function and we focus on constructing $[\mathcal{D}_xw]_{i,j}$ for a right-going wave using a left-biased stencil, then:
\begin{displaymath}
  [\mathcal{D}_xw]_{i,j} = \frac{1}{h_x}\left(w_{i-1,j}-\frac{7}{2}\overline{w}_{\mhf{i},j}+2w_{i,j}+\frac{1}{2}\overline{w}_{\phf{i},j}\right)\;.
\end{displaymath}
All variables in the right hand side except the edge averages $\overline{w}_{\mhf{i},j}$ and $\overline{w}_{\phf{i},j}$ are known; and these missing terms are approximated using nearby cell averages.
In particular for horizontal edges, we have the fourth-order approximation
\begin{equation}\label{eq:2d_edge_int}
  \overline{w}_{\phf{i},j} = \frac{7}{12}\left[\overline{w}_{\phf{i},\mhf{j}}+\overline{w}_{\phf{i},\phf{j}}\right] - \frac{1}{12}\left[\overline{w}_{\phf{i},\mhfg{j}{3}}+\overline{w}_{\phf{i},\phfg{j}{3}}\right]
\end{equation}
in the interior of the domain and the third-order approximations:
\begin{align}
  \label{eq:2d_edge_bc_0}
  \overline{w}_{\phf{i},0} &= \frac{11}{6}\overline{w}_{\phf{i},1/2} - \frac{7}{6}\overline{w}_{\phf{i},3/2} + \frac{1}{3}\overline{w}_{\phf{i},5/2}\;, \\
  \label{eq:2d_edge_bc_1}
  \overline{w}_{\phf{i},1} &= \frac{1}{3}\overline{w}_{\phf{i},1/2} + \frac{5}{6}\overline{w}_{\phf{i},3/2} - \frac{1}{6}\overline{w}_{\phf{i},5/2}\;.
\end{align}
near the bottom boundary.
The edge average of horizontal edges near the upper boundary as well as the edge averages of vertical edges can be computed similarly.

The wall boundary condition is handled similarly as in the 1D case (see the end of~\cref{sec:ext_bc}); particularly we first enforce the zero normal velocity condition strongly at nodes on the wall, then introduce one layer of ghost cells and define ghost nodal and cell-averaged values that mirror the nearby fluid states, and use these ghost values to compute the edge averages along the wall as well as the discrete differential operators.

\subsection{Entropy viscosity for 2D Euler equations on Cartesian grids}
\label{sec:ext_ev}
In this section we briefly discuss the entropy viscosity for 2D problems.
Let the artificial viscosity be denoted $\nu_h$ as before, the regularized equation is given by:
\begin{equation}\label{eq:ext_ev_eqn}
  \bs{W}_t + \bs{F}(\bs{W})_x + \bs{G}(\bs{W})_y - (\nu_h(\bs{A}_1\bs{W}_x+\bs{A}_2\bs{W}_y))_x - (\nu_h(\bs{B}_1\bs{W}_x+\bs{B}_2\bs{W}_y))_x = 0\;, 
\end{equation}
where for Euler equations the scaling matrices are:
\begin{equation}\label{eq:ext_ev_scale}
  \begin{array}{ll}
  \bs{A}_1 = \begin{bmatrix}
    1 & 0 & 0 & 0 \\
    -v_1 & 1 & 0 & 0 \\
    -v_2/2 & 0 & 1/2 & 0 \\
    -E+v_2^2/2 & 0 & -v_2/2 & 1
  \end{bmatrix}\,,\quad &
  \bs{A}_2 = \begin{bmatrix}
    0 & 0 & 0 & 0 \\
    0 & 0 & 0 & 0 \\
    -v_1/2 & 1/2 & 0 & 0 \\
    -v_1v_2/2 & v_2/2 & 0 & 0 
  \end{bmatrix}\,, \\
  \bs{B}_1 = \begin{bmatrix}
    0 & 0 & 0 & 0 \\
    -v_2/2 & 0 & 1/2 & 0 \\
    0 & 0 & 0 & 0 \\
    -v_1v_2/2 & 0 & v_1/2 & 0
  \end{bmatrix}\,,\quad & 
  \bs{B}_2 = \begin{bmatrix}
    1 & 0 & 0 & 0 \\
    -v_1/2 & 1/2 & 0 & 0 \\
    -v_2 & 0 & 1 & 0 \\
    -E+v_1^2/2 & -v_1/2 & 0 & 1
  \end{bmatrix}\,.
  \end{array}
\end{equation}
Here $v_1$ and $v_2$ are the two components of the velocity vector $\bs{v}$.

As in the 1D case, the diffusion terms are discretized using second-order central difference schemes.
Particularly in the cell-averaged equation for $\overline{\bs{W}}_{\phf{i},\phf{j}}$, there is:
\begin{align*}
  (\nu_h\bs{A}_1\bs{W}_x)_x &\approx
  \frac{
    \nu_{i+1,\phf{j}}\bs{A}_{1;i+1,\phf{j}}[\mathcal{D}_x^c\bs{W}]_{i+1,\phf{j}}-
    \nu_{i,\phf{j}}\bs{A}_{1;i,\phf{j}}[\mathcal{D}_x^c\bs{W}]_{i,\phf{j}}
  }{h_x}\;, \\
  (\nu_h\bs{A}_2\bs{W}_y)_x &\approx
  \frac{
    \nu_{i+1,\phf{j}}\bs{A}_{2;i+1,\phf{j}}[\mathcal{D}_y^c\bs{W}]_{i+1,\phf{j}}-
    \nu_{i,\phf{j}}\bs{A}_{2;i,\phf{j}}[\mathcal{D}_y^c\bs{W}]_{i,\phf{j}}
  }{h_x}\;, \\
  (\nu_h\bs{B}_1\bs{W}_x)_y &\approx
  \frac{
    \nu_{\phf{i},j+1}\bs{B}_{1;\phf{i},j+1}[\mathcal{D}_x^c\bs{W}]_{\phf{i},j+1}-
    \nu_{\phf{i},j}\bs{B}_{1;\phf{i},j}[\mathcal{D}_x^c\bs{W}]_{\phf{i},j}
  }{h_y}\;, \\
  (\nu_h\bs{B}_2\bs{W}_y)_y &\approx
  \frac{
    \nu_{\phf{i},j+1}\bs{B}_{2;\phf{i},j+1}[\mathcal{D}_y^c\bs{W}]_{\phf{i},j+1}-
    \nu_{\phf{i},j}\bs{B}_{2;\phf{i},j}[\mathcal{D}_y^c\bs{W}]_{\phf{i},j}
  }{h_y}\;.
\end{align*}
Here $\nu$ at half indices are computed as the average of nodal values, like $\nu_{i,\phf{j}}=(\nu_{i,j}+\nu_{i,j+1})/2$, and similarly the scaling matrices are evaluated using the average of nodal solutions, such as $\bs{A}_{1;i,\phf{j}}=\bs{A}_1((\bs{W}_{i,j+1}+\bs{W}_{i,j})/2)$.
The central HV-DDOs are computed similarly as in 1D case:
\begin{align*}
  [\mathcal{D}_x^c\bs{W}]_{i,\phf{j}} = \frac{1}{h_x}\left[\overline{\bs{W}}_{\phf{i},\phf{j}}-\overline{\bs{W}}_{\mhf{i},\phf{j}}\right]\;,\quad 
  [\mathcal{D}_y^c\bs{W}]_{i,\phf{j}} = \frac{1}{h_y}\left[\bs{W}_{i,j+1}-\bs{W}_{i,j}\right]\;;
\end{align*}
$[\mathcal{D}_x^c\bs{W}]_{\phf{i},j}$ and $[\mathcal{D}_y^c\bs{W}]_{\phf{i},j}$ are computed similarly.

To discretization the diffusion term in the nodal equation of $\bs{W}_{i,j}$, we again freeze the viscosity coefficient and compute:
\begin{align*}
  (\nu_h\bs{A}_1\bs{W}_x)_x &\approx \nu_{i,j}\left[\bs{A}_{1}(\bs{W}_{i,j})[\mathcal{D}_{xx}^c\bs{W}]_{i,j}+\left(\pp{\bs{A}_1(\bs{W}_{i,j})}{\bs{W}}:[\mathcal{D}_x^c\bs{W}]_{i,j}\right)[\mathcal{D}_x^c\bs{W}]_{i,j}\right]\;, \\
  (\nu_h\bs{A}_2\bs{W}_y)_x &\approx \nu_{i,j}\left[\bs{A}_{2}(\bs{W}_{i,j})[\mathcal{D}_{yx}^c\bs{W}]_{i,j}+\left(\pp{\bs{A}_2(\bs{W}_{i,j})}{\bs{W}}:[\mathcal{D}_x^c\bs{W}]_{i,j}\right)[\mathcal{D}_y^c\bs{W}]_{i,j}\right]\;, \\
  (\nu_h\bs{B}_1\bs{W}_x)_y &\approx \nu_{i,j}\left[\bs{B}_{1}(\bs{W}_{i,j})[\mathcal{D}_{xy}^c\bs{W}]_{i,j}+\left(\pp{\bs{B}_1(\bs{W}_{i,j})}{\bs{W}}:[\mathcal{D}_y^c\bs{W}]_{i,j}\right)[\mathcal{D}_x^c\bs{W}]_{i,j}\right]\;, \\
  (\nu_h\bs{B}_2\bs{W}_y)_y &\approx \nu_{i,j}\left[\bs{B}_{2}(\bs{W}_{i,j})[\mathcal{D}_{yy}^c\bs{W}]_{i,j}+\left(\pp{\bs{B}_2(\bs{W}_{i,j})}{\bs{W}}:[\mathcal{D}_y^c\bs{W}]_{i,j}\right)[\mathcal{D}_y^c\bs{W}]_{i,j}\right]\;, 
\end{align*}
where the computation of the central HV-DDOs are illustrated as below:
\begin{align*}
  &[\mathcal{D}_x^c\bs{W}]_{i,j} = \frac{1}{h_x}\left[\overline{\bs{W}}_{\phf{i},j}-\overline{\bs{W}}_{\mhf{i},j}\right]\;,\quad
  [\mathcal{D}_{xx}^c\bs{W}]_{i,j} = \frac{1}{h_x^2}\left[3\overline{\bs{W}}_{\mhf{i},j}-6\bs{W}_{i,j}+3\overline{\bs{W}}_{\phf{i},j}\right]\;, \\
  &[\mathcal{D}_{xy}^c\bs{W}]_{i,j} = [\mathcal{D}_{yx}^c\bs{W}]_{i,j} = \frac{1}{h_xh_y}\left[\overline{\bs{W}}_{\mhf{i},\mhf{j}}-\overline{\bs{W}}_{\mhf{i},\phf{j}}-\overline{\bs{W}}_{\phf{i},\mhf{j}}+\overline{\bs{W}}_{\phf{i},\phf{j}}\right]\;.
\end{align*}

Lastly, the artificial viscosity $\nu_h$ is constructed following essentially the same recipe in~\cref{sec:hv}; and we only describe here the construction of cellular entropy residual and the benchmark von Neumann-Richtmyer viscosity.
Following the discussion at the end of~\cref{sec:rcv_vh}, a second-order cellular entropy residual is needed to preserve the formal fourth-order accuracy of the method when the solutions are smooth.
To this end, at semi-discrete level we compute:
%\begin{align}
%  \notag
%  \oname{Res}_{\phf{i},\phf{j}}S =&\ \frac{d\overline{S}_{\phf{i},\phf{j}}}{dt} + \frac{\eta_1(\bs{W}_{i+1,j})+\eta_1(\bs{W}_{i+1,j+1})-\eta_1(\bs{W}_{i,j})-\eta_1(\bs{W}_{i,j+1})}{2h_x} \\
%  \label{eq:ext_ev_res}
%  &\ + \frac{\eta_2(\bs{W}_{i,j+1})+\eta_2(\bs{W}_{i+1,j+1})-\eta_2(\bs{W}_{i,j})-\eta_2(\bs{W}_{i+1,j})}{2h_y}\;,
%\end{align}
\begin{align}
  \notag
  \oname{Res}_{\phf{i},\phf{j}}s =&\ \frac{d\overline{s}_{\phf{i},\phf{j}}}{dt} + \overline{v}_{1,\phf{i},\phf{j}}
  \frac{s(\bs{W}_{i+1,j})+s(\bs{W}_{i+1,j+1})-s(\bs{W}_{i,j})-s(\bs{W}_{i,j+1})}{2h_x} \\
  \label{eq:ext_ev_res}
  &\ + \overline{v}_{2,\phf{i},\phf{j}}\frac{s(\bs{W}_{i,j+1})+s(\bs{W}_{i+1,j+1})-s(\bs{W}_{i,j})-s(\bs{W}_{i+1,j})}{2h_y}\;,
\end{align}
where as before $\overline{s}_{\phf{i},\phf{j}}$, $\overline{v}_{1,\phf{i},\phf{j}}$, and $\overline{v}_{2,\phf{i},\phf{j}}$ are $s$, $v_1$, and $v_2$ computed from $\overline{\bs{W}}_{\phf{i},\phf{j}}$.
Following the same strategy as described in~\cref{sec:rcv_vh}, one can then compute the fully discretized residual $\oname{Res}_{\phf{i},\phf{j}}^ns$ and process it to obtain $Z_{\phf{i},\phf{j}}^n$ and $\mathcal{M}\!\circ\!\mathcal{S}\!\circ\!\mathcal{N}(Z_{\phf{i},\phf{j}}^n)$; and the cellular artificial viscosity computes as:
\begin{equation}\label{eq:ext_ev_vh}
  \nu_{\phf{i},\phf{j}}^n = \frac{h v_{\max}}{4}(\mathcal{M}\!\circ\!\mathcal{S}\!\circ\!\mathcal{N})(Z_{\phf{i},\phf{j}}^n)\;,
\end{equation}
where $h=(h_x\abs{v_1}+h_y\abs{v_2})/\abs{\bs{v}}$ is the reference mesh length in the direction of flow velocity and $v_{\max} = \abs{\bs{v}}+c_s$ is the largest characteristic velocity; both terms are evaluated using the cell-average $\bs{W}_{\phf{i},\phf{j}}^n$.

\section{Numerical examples}
\label{sec:num}
In this section we assess the numerical performance of the proposed method with various one-dimensional and two-dimensional problems.
Unless otherwise noted, the Courant number is fixed to $\alpha_{\cfl}=0.6$ and the parameter $z_0$ that determines the artificial viscosity is selected $z_0=0.04$.
Most of the tests are governed by Euler equations; in other cases, the choice of the entropy in the artificial viscosity computation will be discussed therein.

\subsection{One-dimensional tests}
\label{sec:num_1d}
The first three 1D tests are governed by the linear advection equation $w_t+\lambda w_x=0$ with the entropy variables chosen as $s(w) = S(w) = w^2/2$; and for the purpose of computing the artificial viscosity as shown in~\cref{eq:rcv_res_semi} and~\cref{eq:rcv_vh}, we take $\overline{v}\equiv\lambda$ and $\overline{\rho}\equiv1$.
Particularly,~\cref{sec:num_1d_cauchy} and~\cref{sec:num_1d_ibvp} evaluate the accuracy of the method with periodic and Dirichlet type boundary conditions, respectively; and~\cref{sec:num_1d_gste} assesses the performance advecting various profiles including discontinuous ones.
Then the standard Euler equations are solved, with~\cref{sec:num_1d_col} and~\cref{sec:num_1d_wall} focusing on the numerical accuracy with periodic and wall boundary conditions whereas~\cref{sec:num_1d_sod} concerns a benchmark shock tube test.

\subsubsection{Linear advection: A Cauchy problem}
\label{sec:num_1d_cauchy}
We solve the equation $w_t+w_x=0$ on the domain $(x,t)\in[0,\,1]\times[0,\,1]$ with periodic boundary condition and the initial condition:
\begin{equation}\label{eq:num_1d_cauchy_ic}
  w(x,0) = \sin(2\pi x) + \cos(4\pi x)\;.
\end{equation}
At $T=1$ the exact solution is the same as the initial data, which is used to compute the numerical errors.
In~\cref{fg:num_1d_cauchy} we plot the $L_1$-norms of numerical errors measured in logarithmic scale on a sequence of 7 uniform grids with number of cells ranging from $40$ to $2560$, using different values of $z_0$.
Recall that {\it higher} $z_0$ indicates activation of the von Neumann viscosity at {\it stronger} shocks, in which case the method is {\it less} impacted by the artificial viscosity; in the figure, the convergence curves corresponding to $z_0=0.01$, $z_0=0.02$, $z_0=0.04$ (default choice of this paper), $z_0=0.08$, $z_0=1.0$ (the artificial viscosity is completely determined by the entropy residual), and no artificial viscosity at all.
\begin{figure}\centering
  \includegraphics[trim=.6in .0in 1.6in .6in, clip, width=.48\textwidth]{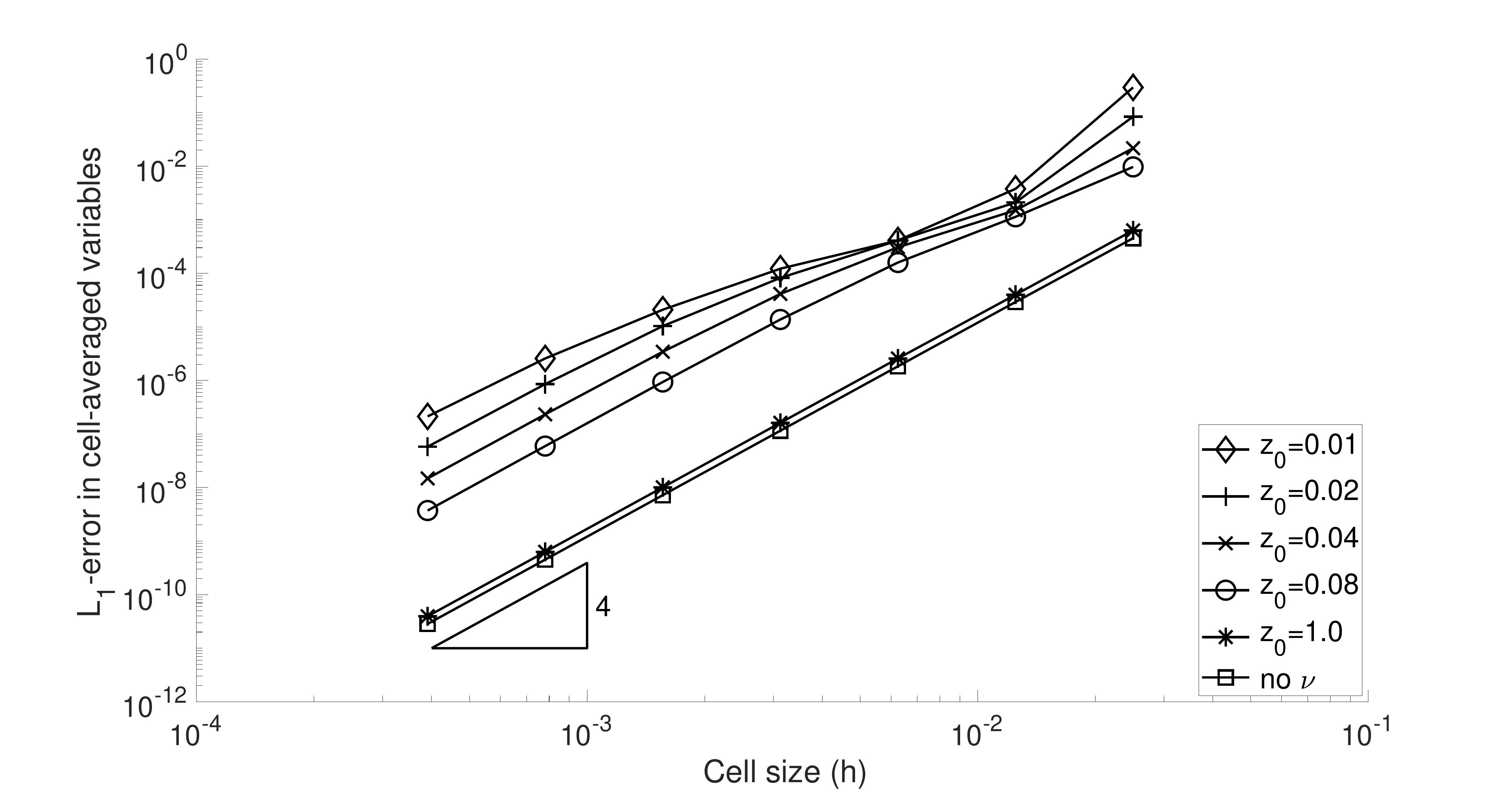} ~~~~
  \includegraphics[trim=.6in .0in 1.6in .6in, clip, width=.48\textwidth]{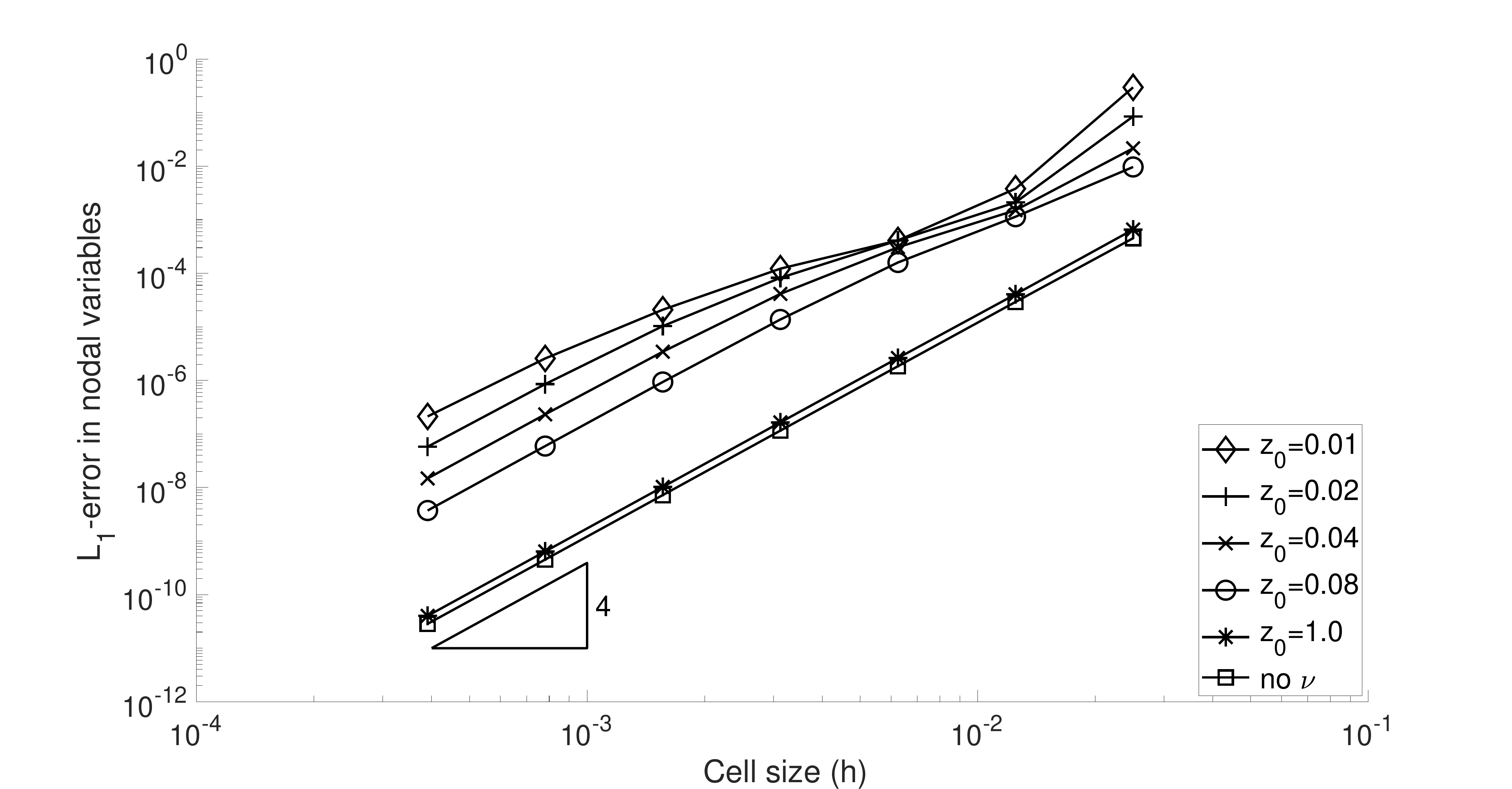}
  \caption{Convergence curves of numerical errors (logarithmic scale) computed in solving the Cauchy problem of linear advection equation (\cref{sec:num_1d_cauchy}) with various $z_0$ values.}
  \label{fg:num_1d_cauchy}
\end{figure}
Note that the numerical errors in cell-averaged values (left panel) and nodal values (right panel) are very close.

\subsubsection{Linear advection: An initial boundary value problem}
\label{sec:num_1d_ibvp}
Next essentially the same problem as before is solved, except that instead of a periodic boundary condition we enforce the following Dirichlet condition at $x=0$: $w(0,t) = -\sin(2\pi t) + \cos(4\pi t)$.
The same sequence of grids are used to compute the numerical solutions, whose errors are measured in $L_1$-norm and plotted in logarithmic scale in~\cref{fg:num_1d_ibvp}.
\begin{figure}\centering
  \includegraphics[trim=.6in .0in 1.6in .6in, clip, width=.48\textwidth]{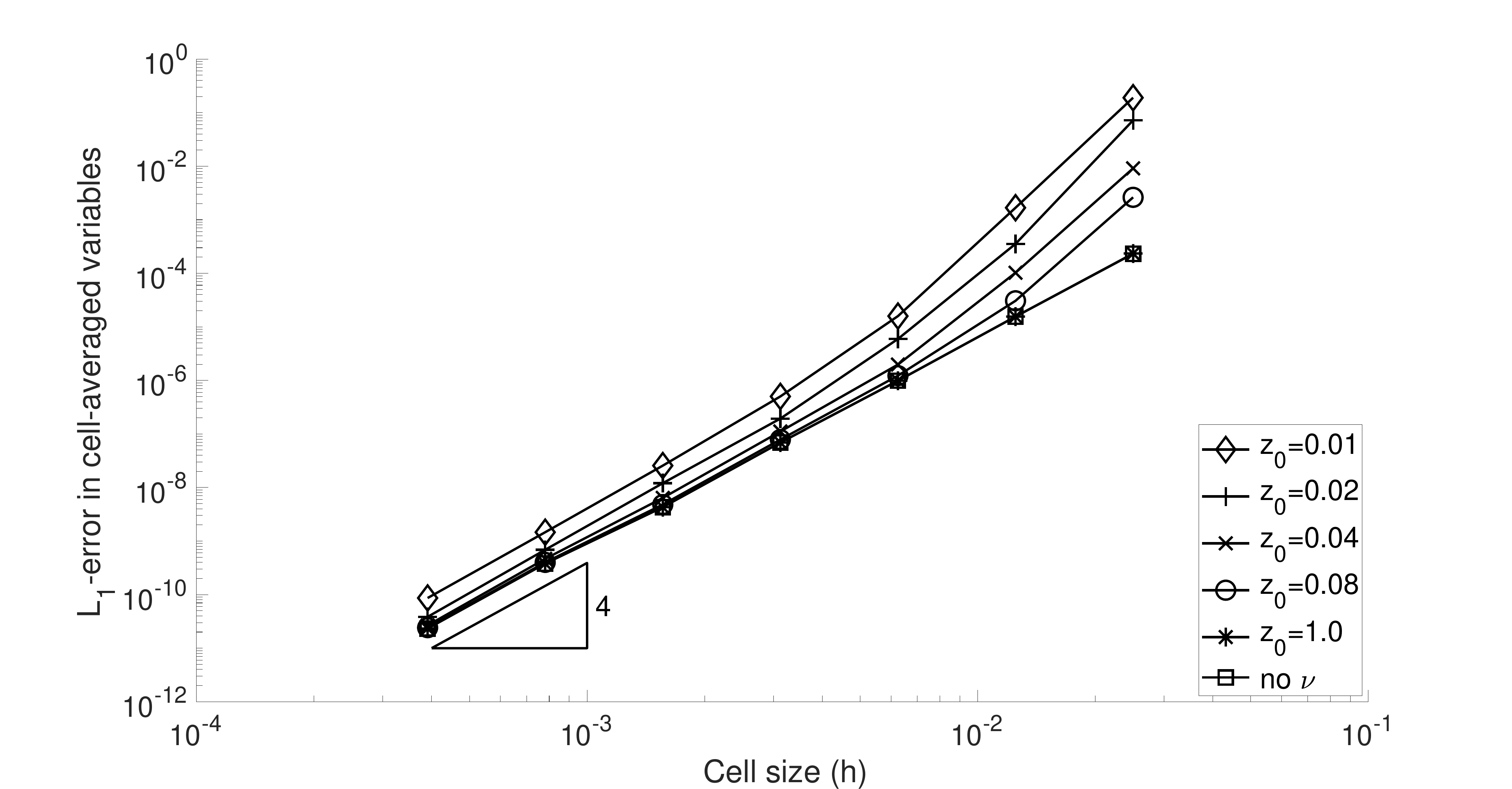} ~~~~
  \includegraphics[trim=.6in .0in 1.6in .6in, clip, width=.48\textwidth]{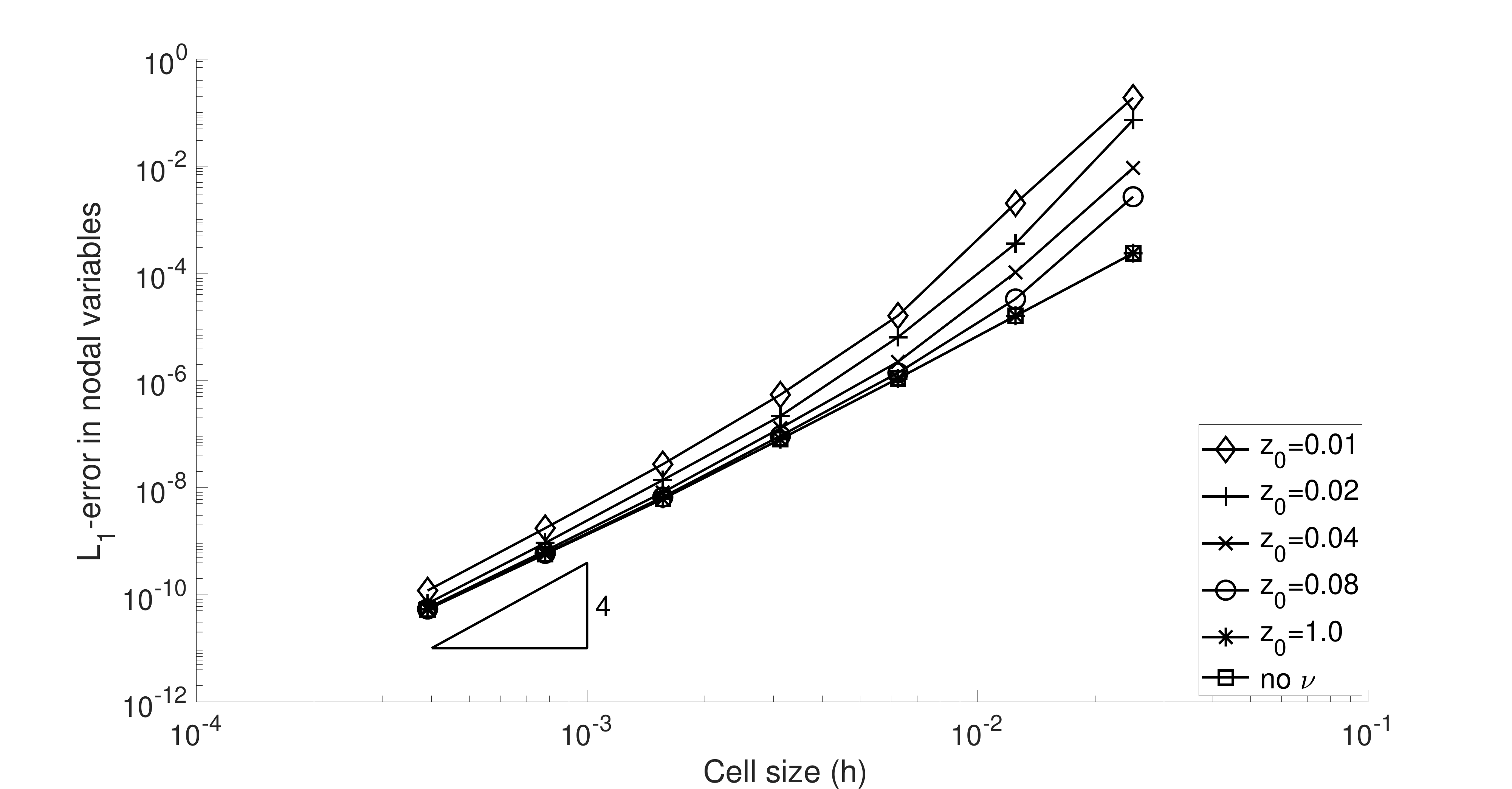}
  \caption{Convergence curves of numerical errors (logarithmic scale) computed in solving the initial boundary value problem of linear advection equation (\cref{sec:num_1d_ibvp}) with various $z_0$ values.}
  \label{fg:num_1d_ibvp}
\end{figure}
The convergence rate for nodal variables is slightly less than fourth-order whereas that for cell-averaged variables are fourth-order as predicted.

%The convergence curves in this and previous sub-sections confirm the formal fourth-order of accuracy of the proposed method.

\subsubsection{Linear advection of Gaussian, square, sharp triangle, and half ellipse waves}
\label{sec:num_1d_gste}
In the last 1D linear advection problem we assess the performance of the HV method in handling various types of discontinuities.
Particularly we solve $w_t+w_x=0$ on $(x,t)\in[-1,\,1]\times[0,\,8]$ with periodic boundary condition, and an initial wave that is composed of smooth but narrow Gaussians, a square wave, a sharp triangle wave, and a half ellipse~\cite{GSJiang:1996a}:
\begin{equation}\label{eq:num_1d_gste_ic}
  w(x,0) = \left\{\begin{array}{lcl}
    \frac{1}{6}\left[G_{\beta,\,z-\delta}(x)+4G_{\beta,\,z}(x)+G_{\beta,\,z+\delta}(x)\right]\,, & & -0.8\le x\le-0.6\;; \\ \vspace*{-.15in} \\
    1\,, & & -0.4\le x\le-0.2\;; \\ \vspace*{-.15in} \\
    1-\abs{10(x-0.1)}\,, & & 0\le x\le0.2\;; \\ \vspace*{-.15in} \\
    \frac{1}{6}\left[L_{\alpha,\,a-\delta}(x)+4L_{\alpha,\,a}(x)+L_{\alpha,\,a+\delta}(x)\right]\,, & & 0.4\le x\le0.6\;; \\ \vspace*{-.15in} \\
    0\,, & & \textrm{otherwise}\;.
  \end{array}\right.\,,
\end{equation}
where $G_{\beta,\,z}(x)\eqdef e^{-\beta(x-z)^2}$ and $L_{\alpha,\,a}(x)\eqdef\sqrt{\max(1-\alpha^2(x-a)^2,0)}$, and the constants are taken as $a=0.5$, $z=-0.7$, $\delta=0.005$, $\alpha=10$, and $\beta=\log(2)/(36\delta^2)$.
In~\cref{fg:num_1d_gste} we plot the HV solution obtained without or with the artificial viscosity on a uniform grid with $320$ cells and compare them to those obtained by a finite volume method (MUSCL~\cite{BvanLeer:1979a} with van Albada slope limiter~\cite{GDvanAlbada:1982a} and Rusanov numerical flux~\cite{VVRusanov:1962a}) using either the same grid ($320$ cells) or the same number of unknowns ($640$ cells).
\begin{figure}\centering
  \includegraphics[width=\textwidth]{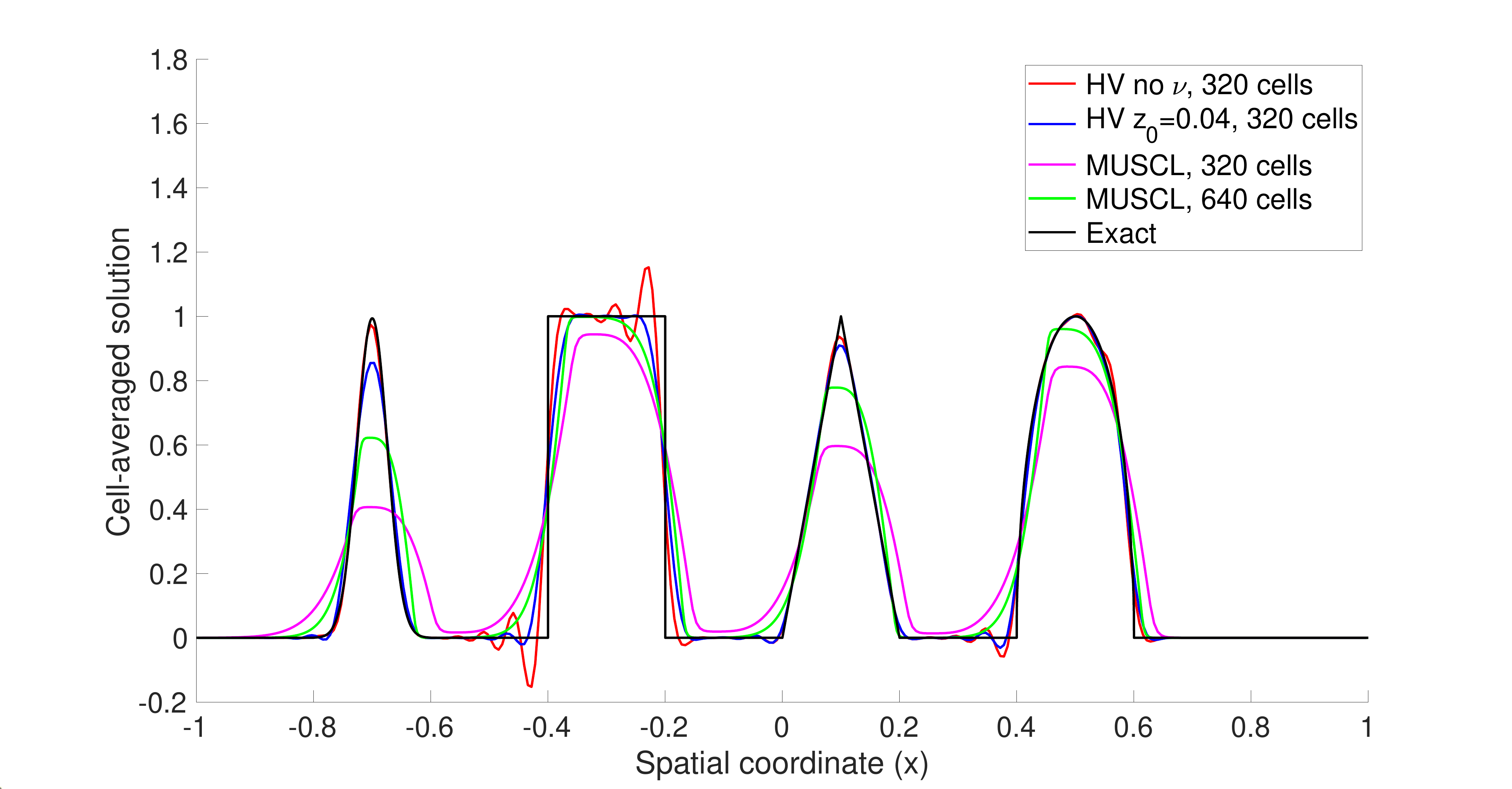}
  \caption{Cell-averaged solutions obtained by the HV method without (red curve) or with artificial viscosity (blue curve), MUSCL with the same grid (magenta curve), MUSCL with the same number of unknowns (green curve), and the exact solution (black curve).}
  \label{fg:num_1d_gste}
\end{figure}
One observes that the HV method provides much better accuracy than MUSCL whether using the same mesh resolution or the same number of unknowns; and comparing the two HV computations, the artificial viscosity effectively suppresses the spurious oscillations in the upstream direction of discontinuities.

\subsubsection{Euler equations: A Cauchy problem of wave collision}
\label{sec:num_1d_col}
Moving on to gas dynamics, we consider the collision between two symmetric bumps~\cite{MMHasan:2023a} by solving the Euler equations with specific heat ratio $\gamma=1.4$ on a periodic domain $\Omega=[-2,\,2]$ and the initial condition:
\begin{equation}\label{eq:num_1d_col_ic}
  \rho(x,0) = 1.4+1.4\epsilon B(x)\,,\quad
  v(x,0) = 0.0\,,\quad
  p(x,0) = 1+\epsilon B(x)\;,
\end{equation}
where $\epsilon=0.1$ and the symmetric bump $B(x)$ is given by:
\begin{equation}\label{eq:num_1d_col_bump}
  B(x) = \left\{\begin{array}{lcl}
    \left[\frac{1}{2}\left(1-\cos(2\pi(x+0.5))\right)\right]^4 & , & -1.5\le x\le-0.5 \\
    \left[\frac{1}{2}\left(1-\cos(2\pi(x-0.5))\right)\right]^4 & , & 0.5\le x\le1.5 \\
    0.0 & , & \textrm{otherwise}.
  \end{array}\right.
\end{equation}
Two symmetric waves moving in opposite directions are originated from the bump centered at $x=-0.5$ as well as the bump centered at $x=0.5$, respectively.
The two outward moving waves run into each other at $x=-2$ (or $x=2$) when $T\approx0.64$ and the two inward moving waves collide with each other at $x=0$ about the same time (see the left panel of~\cref{fg:num_1d_col_prs}).
This problem admits a smooth solution at least until $T=1.2$ (see the right panel of~\cref{fg:num_1d_col_prs}), when the numerical solution is computed and compared to a reference one to assess the numerical accuracy.
In these plots, little arrows are added above the bumps to denote their moving directions.
The reference solution is computed using an extremely fine grid with $20480$ cells, which is used to compute the numerical errors of the solutions obtained on a sequence of 7 uniform grids with number of cells ranging from $40$ to $2560$, with their $L_1$-norms plotted in logarithmic scales by the solid curves in~\cref{fg:num_1d_col}.
\begin{figure}\centering
  \includegraphics[trim=.6in .0in 1.6in .6in, clip, width=.48\textwidth]{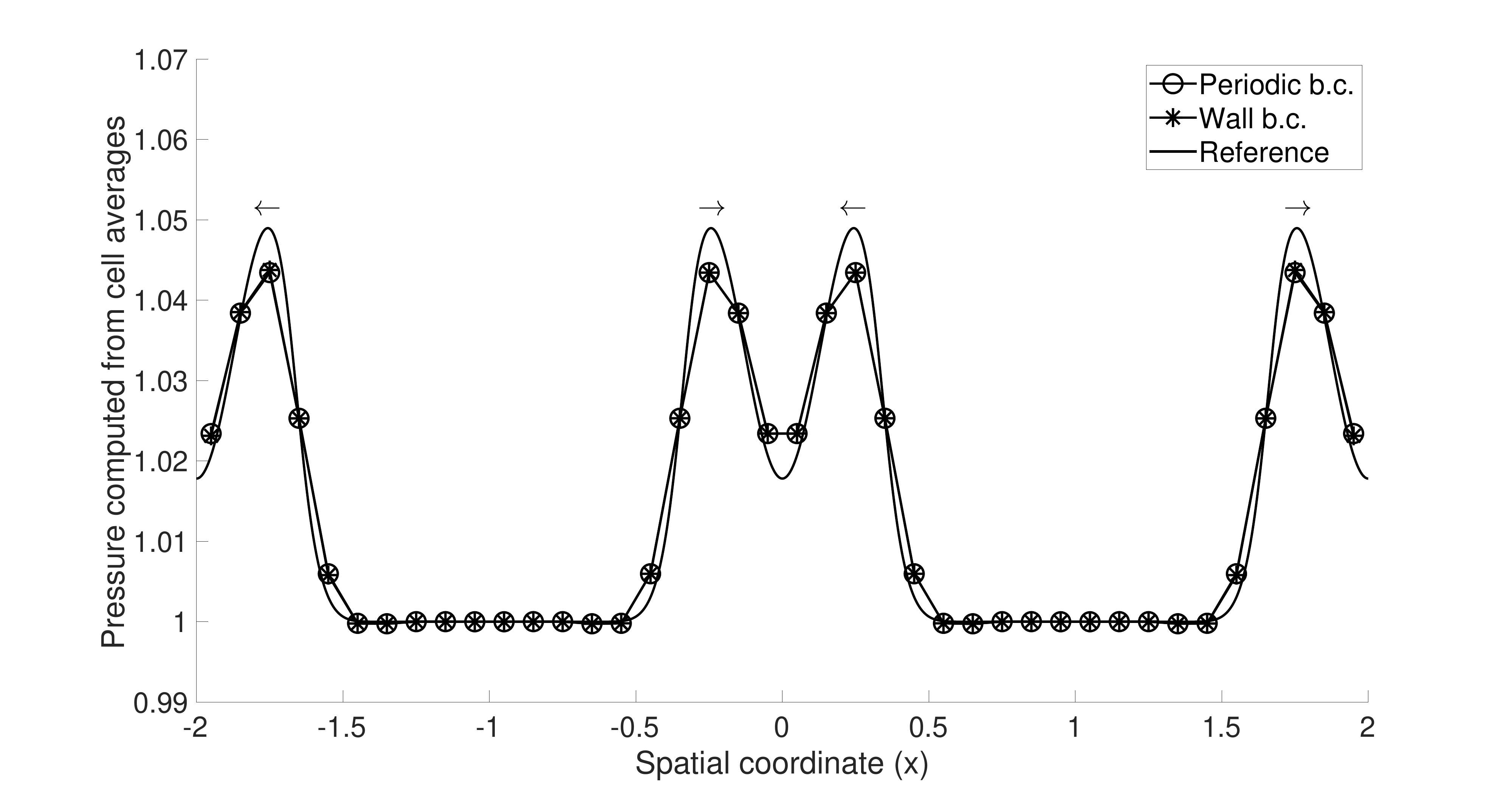} ~~~~
  \includegraphics[trim=.6in .0in 1.6in .6in, clip, width=.48\textwidth]{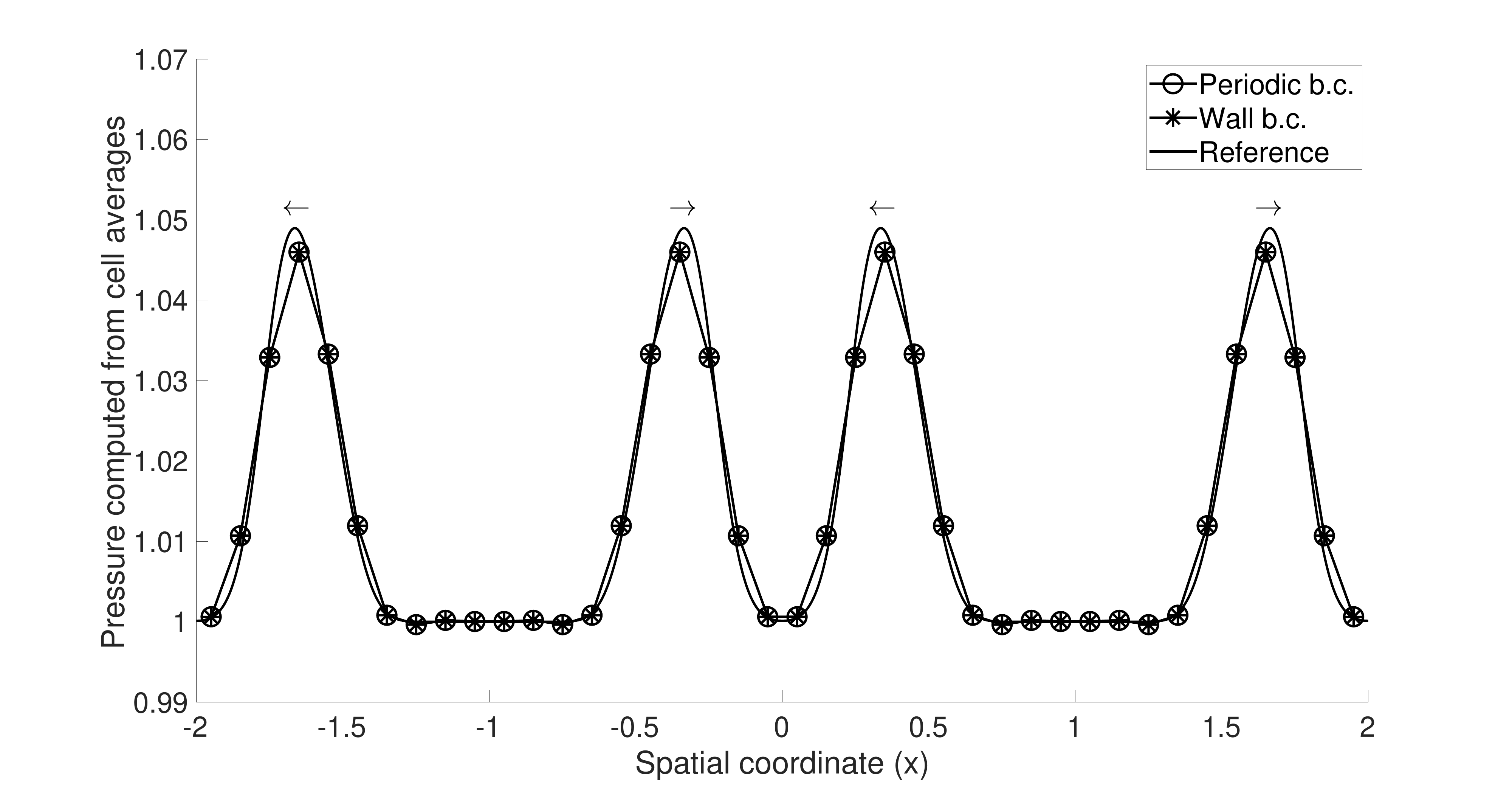}
  \caption{Pressure at $T=0.64$ right before the waves collide (left panel) and at $T=1.20$ (right panel). 
    The computations using periodic boundary conditions and wall boundary conditions are denoted by squares and asteroids, respectively.
    Only cell-averaged solutions are used to make the plots; all computations use a uniform grid of $40$ cells and the entropy viscosity with $z_0=0.04$.}
    \label{fg:num_1d_col_prs}
\end{figure}

\subsubsection{Euler equations: Wave collision with walls}
\label{sec:num_1d_wall}
Due to the symmetry of the previous problem, one can replace the periodic boundary condition with the wall boundary condition and obtain a mathematically equivalent problem.
Using the same sequence of grids, we compute the numerical solutions with wall boundaries and plot them in the same figures as before.
\begin{figure}\centering
  \includegraphics[trim=.6in .0in 1.6in .6in, clip, width=.48\textwidth]{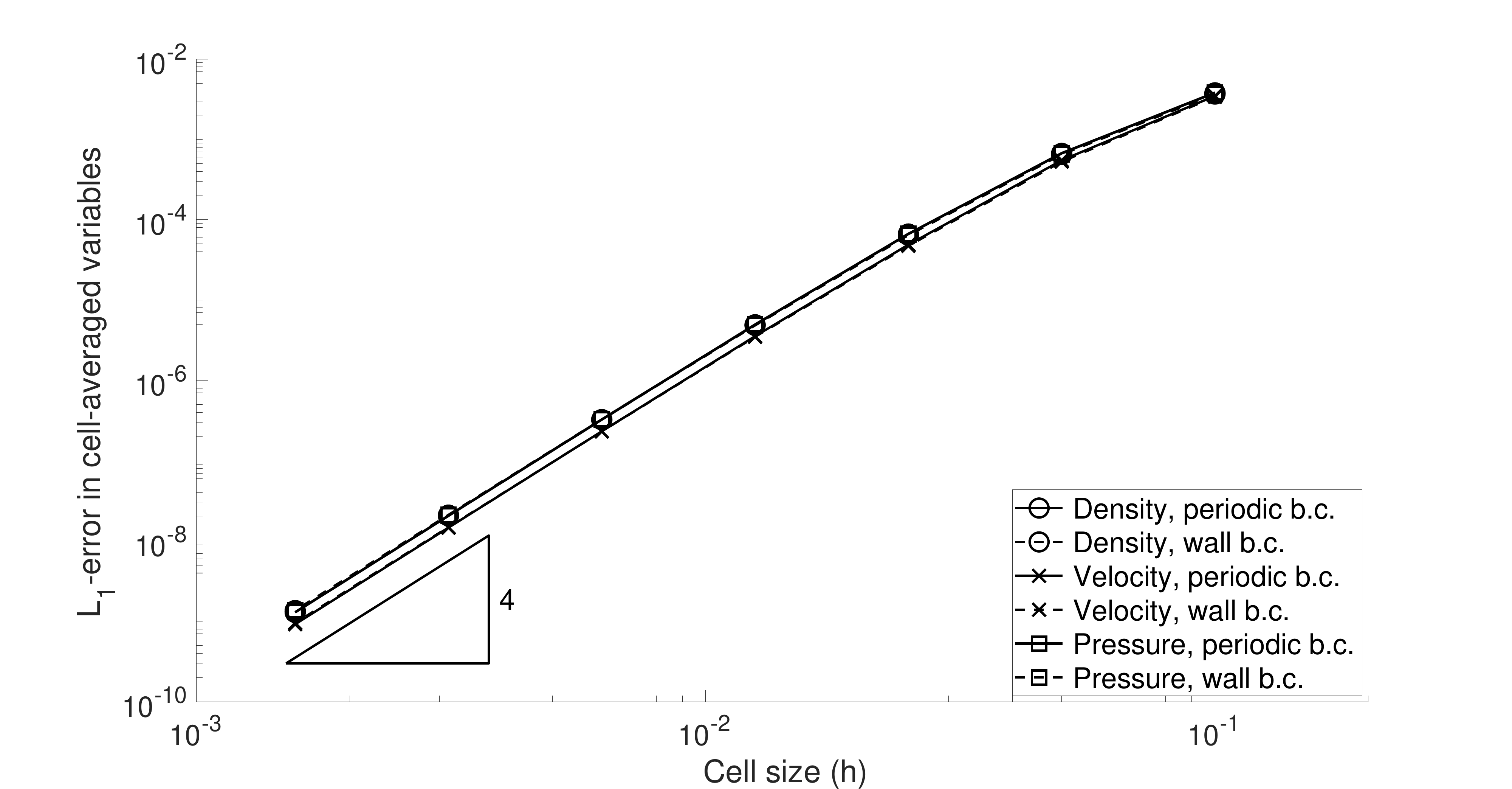} ~~~~
  \includegraphics[trim=.6in .0in 1.6in .6in, clip, width=.48\textwidth]{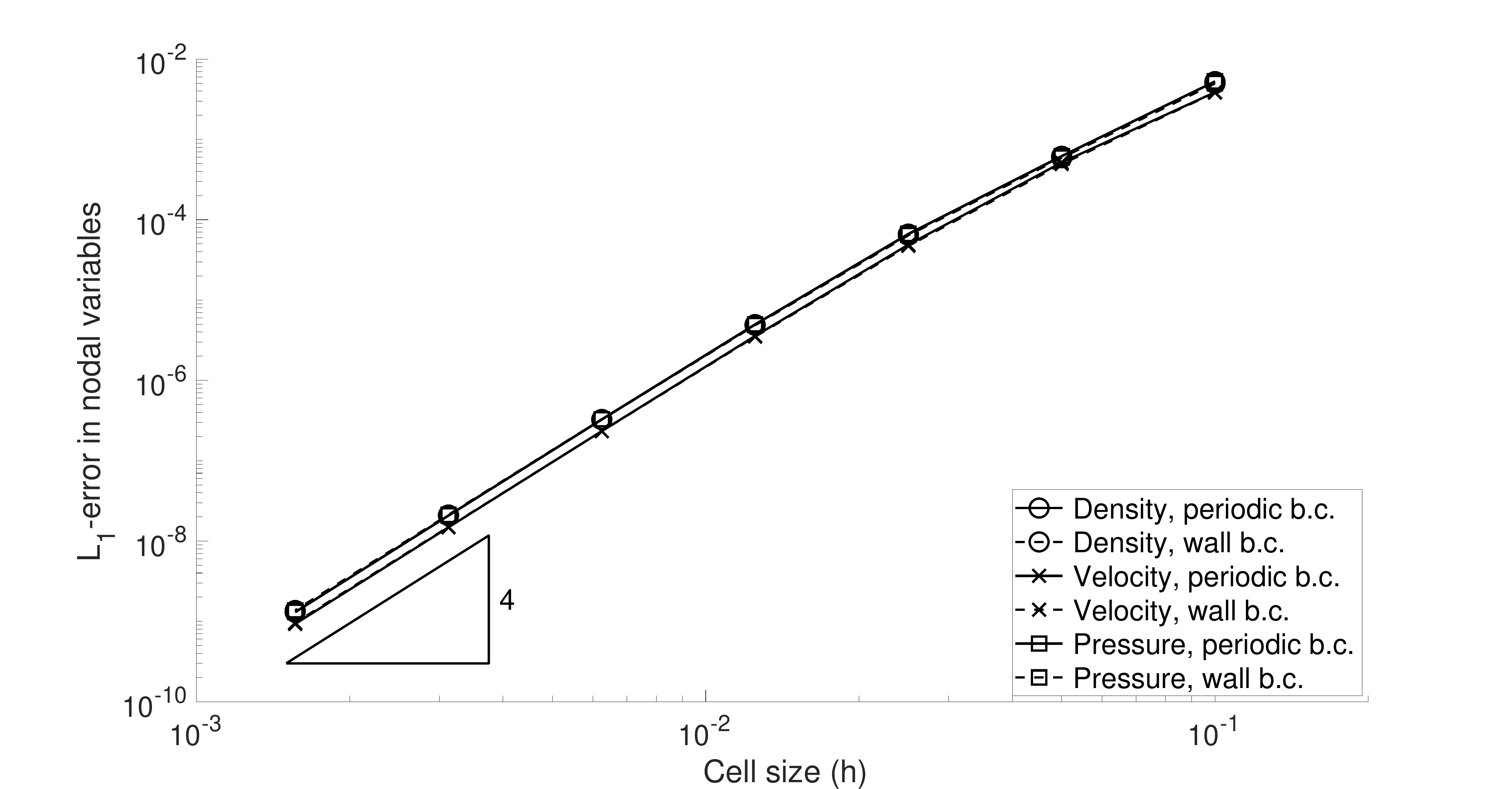}
  \caption{Convergence curves of numerical errors (logarithmic scale) in wave collision problems.
    Solid curves -- periodic boundary conditions (\cref{sec:num_1d_col}); dashed curves -- wall boundary conditions (\cref{sec:num_1d_wall}).
    All computations use the entropy viscosity with $z_0=0.04$.}
  \label{fg:num_1d_col}
\end{figure}
Particularly, in~\cref{fg:num_1d_col} the convergence curves obtained by using the wall boundary conditions are denoted by dashed lines and they're almost on top of the solid ones.
Further comparison can be seen from~\cref{fg:num_1d_col_prs}, where the pressures at $T=0.64$ and $T=1.2$ are plotted and compared to those obtained using periodic boundary condition as well as the reference solution.

\subsubsection{Euler equations: The Sod shock tube problem}
\label{sec:num_1d_sod}
In the last 1D test we assess the performance of the proposed method handling strong discontinuities by solving the benchmark Sod shock tube problem~\cite{GSod:1978a}.
The problem is again governed by the Euler equations with $\gamma=1.4$ on the domain $\Omega = [-2,\,2]$ and initial data:
\begin{equation}\label{eq:num_1d_sod_ic}
  \left\{\begin{array}{lcl}
    \rho(x,0) = 1.0,\quad v(x,0) = 0.0,\quad p(x,0) = 1.0\,, & & x\in[-2,\,0]\;, \\ \vspace{-.1in} \\
    \rho(x,0) = 0.125,\quad v(x,0) = 0.0,\quad p(x,0) = 0.1\,, & & x\in[0,\,2]\;.
  \end{array}\right.
\end{equation}
The problem is solved until $T=0.8$, by which time the waves have not hit the boundaries yet; hence the Dirichlet boundary condition is enforced throughout the computation.
Numerical solutions obtained using a uniform grid of $160$ cells are plotted in~\cref{fg:num_1d_sod}, together with the MUSCL solutions\footnote{The Roe method~\cite{PLRoe:1981a} is chosen to compute the numerical fluxes.} on a $160$-cell grid and a $320$-cell grid as well as the exact solution.
\begin{figure}\centering
  \includegraphics[trim=.6in .0in 1.6in .6in, clip, width=.48\textwidth]{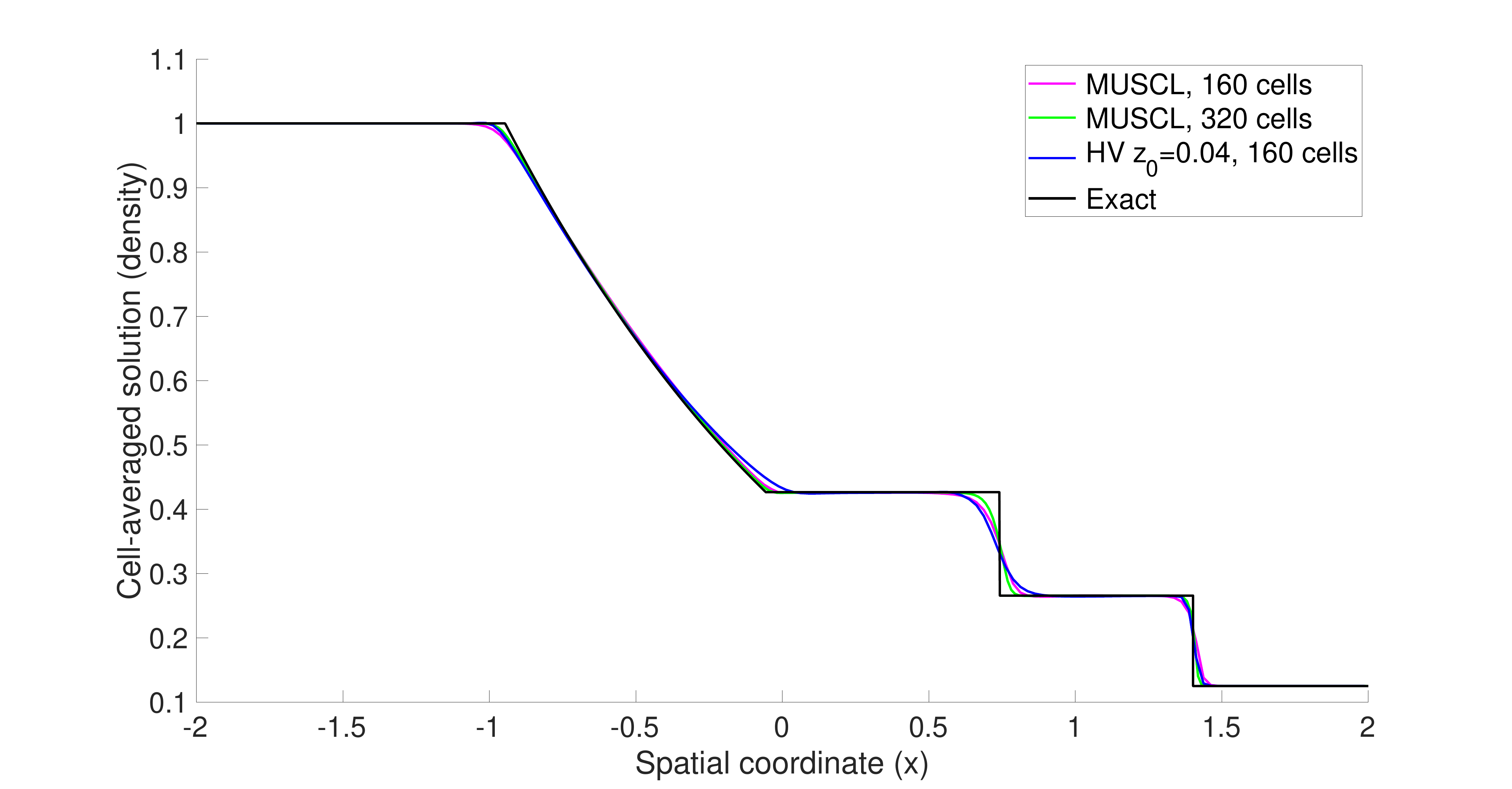} ~~~~
  \includegraphics[trim=.6in .0in 1.6in .6in, clip, width=.48\textwidth]{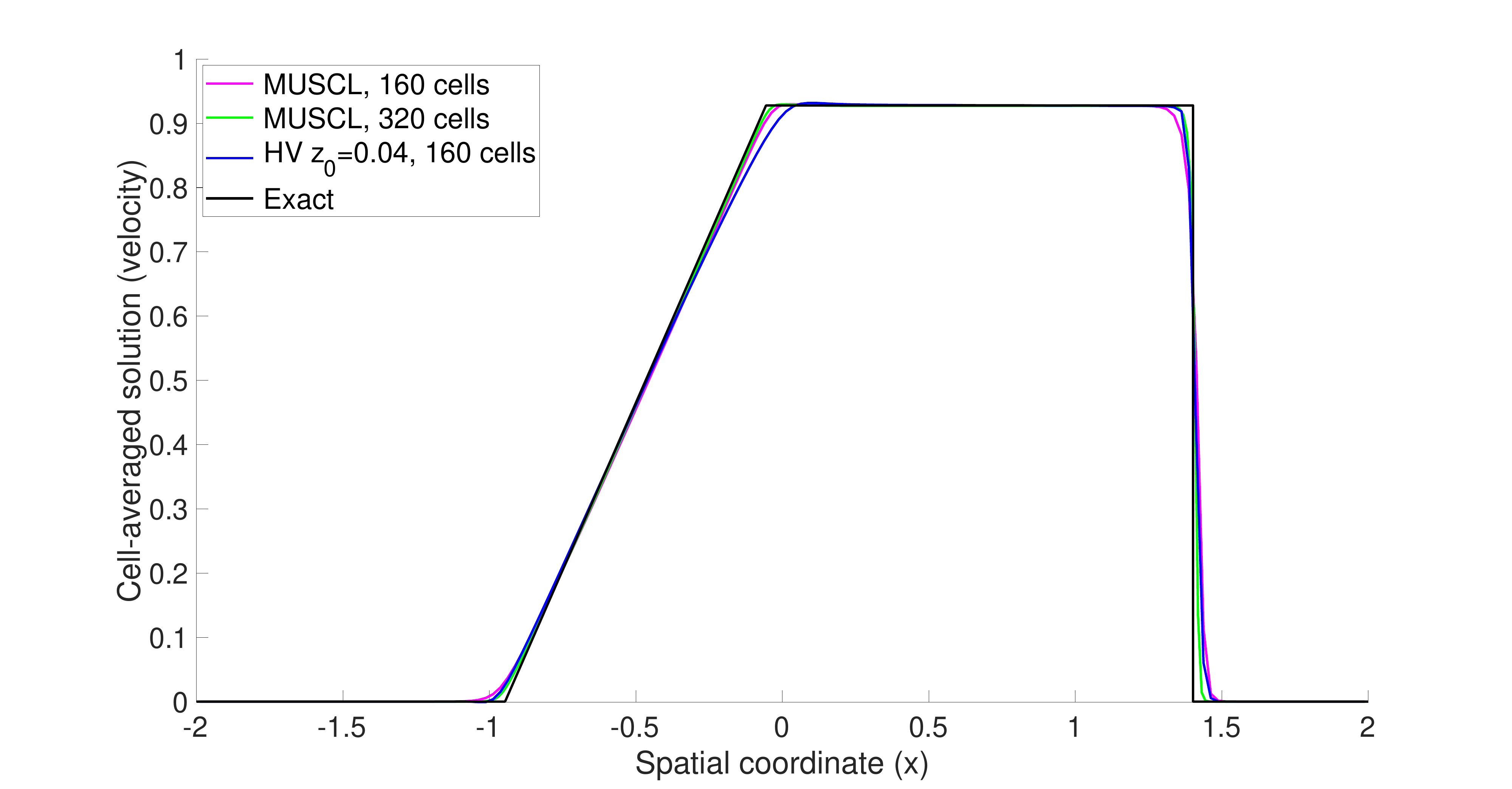} \\
  \includegraphics[trim=.6in .0in 1.6in .6in, clip, width=.48\textwidth]{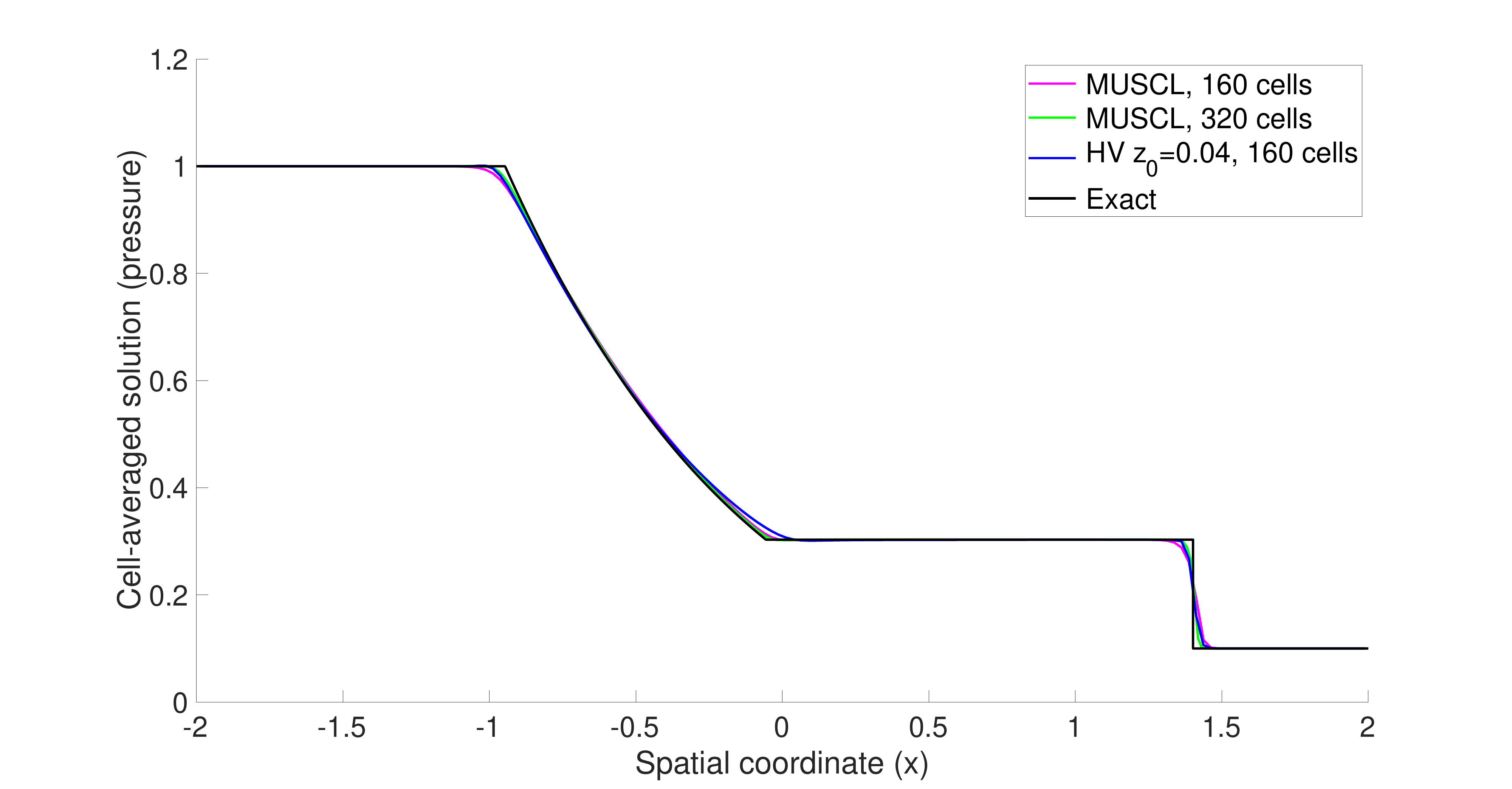} ~~~~
  \includegraphics[trim=.6in .0in 1.6in .6in, clip, width=.48\textwidth]{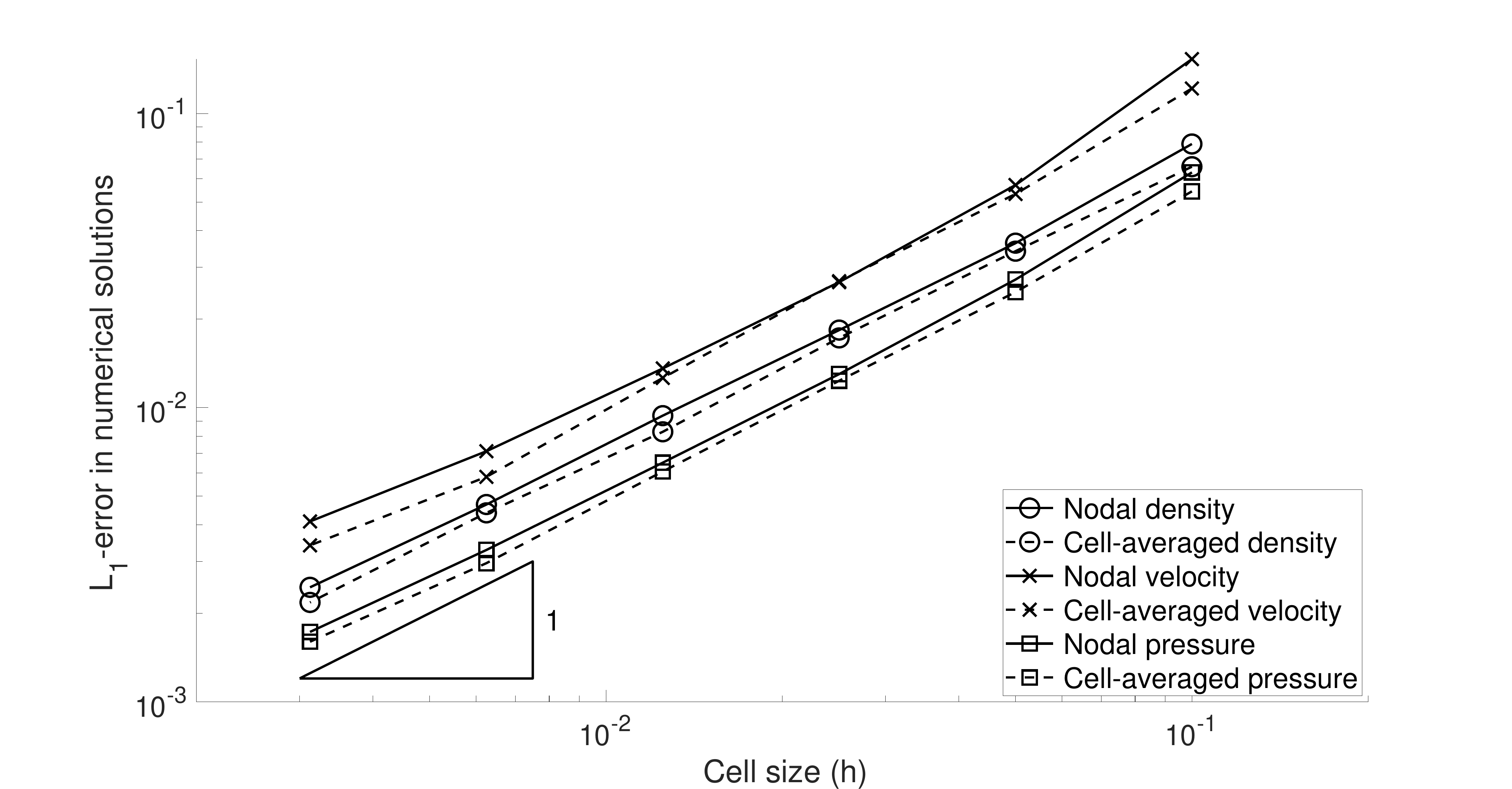} 
  \caption{Numerical solutions in the Sod shock tube test. 
    Density, velocity, and pressure computed from cell-averaged variables are plotted in upper-left panel, upper-right panel, and lower-left panel, respectively; convergence curves are provided in the lower-right panel.}
  \label{fg:num_1d_sod}
\end{figure}
It is observed that the HV solutions show similar resolution as MUSCL with the same number of unknowns at the shock front, whereas the numerical error at a contact discontinuity is largely determined by the mesh resolution.
Finally, using a sequence of 7 uniform grids with number of cells ranging from $40$ to $1280$ we obtain first-order convergence of all variables (see the lower-right panel of~\cref{fg:num_1d_sod}), as expected for discontinuous problems.

Note that we plotted the scaled entropy residual and its processed variants at $T=0.8$ in the $80$-cell and $320$-cell computations in~\cref{fg:rcv_vh_coef}; and it is clear that the location where the baseline von Neumann viscosity is activated agrees well with the shock position.

\subsection{Two-dimensional tests}
\label{sec:num_2d}
In the last section, we consider various benchmark two-dimensional tests.
Particularly, two tests governed by the Euler equations are considered for accuracy assessment, where~\cref{sec:num_2d_vortex} focuses on the interior discretization operator and~\cref{sec:num_2d_tg} concentrates on wall boundary conditions and the handling of source terms.
Solutions with discontinuities are investigated in~\cref{sec:num_2d_kpp} and~\cref{sec:num_2d_sb} for the scalar KPP problem and the 2D Euler equations, respectively.

\subsubsection{The KPP problem}
\label{sec:num_2d_kpp}
The KPP problem~\cite{AKurganov:2007a} is a scalar conservation law given by:
\begin{equation}\label{eq:num_2d_kpp_eqn}
  w_t + (\sin(w))_x + (\cos(w))_y = 0\;,\quad
  (x,y)\in\Omega = [-2,\,2]\times[-2.5,\,1.5]\;,
\end{equation}
with initial condition:
\begin{equation}\label{eq:num_2d_kpp_ic}
  w(x,y,0) = \left\{\begin{array}{lcl}
    \frac{7\pi}{2} & & \textrm{if }\ \sqrt{x^2+y^2}\le1\;, \\ \vspace*{-.15in} \\
    \frac{\pi}{4} & & \textrm{otherwise}\;.
  \end{array}\right.
\end{equation}
Dirichlet condition $w=\pi/4$ is enforced at $\partial\Omega$ for the duration of simulation until $T=1.0$.
To compute the artificial viscosity, we select the usual entropy for scalar equations $s(w)=S(w)=w^2/2$ and use $\bs{v}(w)=[\cos(w),\,-\sin(w)]^T$ and $\overline{\rho}\equiv1$ in the computation of cellular residual and the dimensionless indicator $Z$. 

The KPP equation has a non-convex flux and it is challenging due to a composite wave structure.
In~\cref{fg:num_2d_kpp}, we plot the cell-averaged solution obtained by solving the problem using a uniform $240\times240$ grid alongside the nonentropic solution obtained by the classical MUSCL scheme with superbee limiter~\cite{PLRoe:1986b}; in addition, the activation factor ($\mathcal{M}\circ\mathcal{S}\circ\mathcal{N}(Z_{\phf{i},\phf{j}})$) of the HV computation is also provided and it shows the artificial viscosity is mainly activated in the region of the spiral discontinuity.
\begin{figure}\centering
  \begin{subfigure}[b]{.32\textwidth}\centering
    \includegraphics[trim=5.8in .6in 5.8in .6in, clip, width=\textwidth]{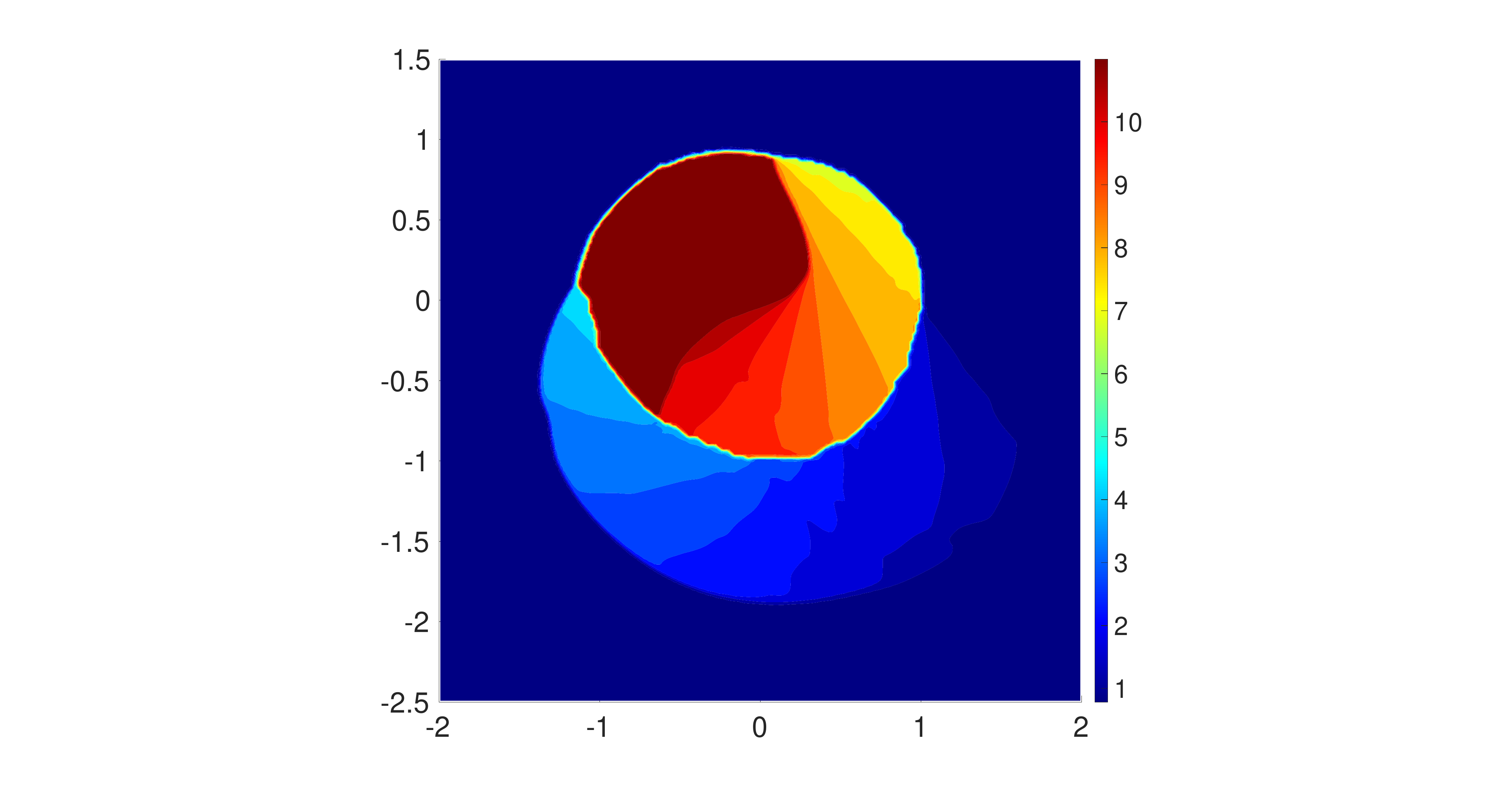}
    \caption{MUSCL solution.}
    \label{fg:num_2d_kpp_muscl}
  \end{subfigure} ~~
  \begin{subfigure}[b]{.64\textwidth}\centering
    \includegraphics[trim=5.8in .6in 5.8in .6in, clip, width=.5\textwidth]{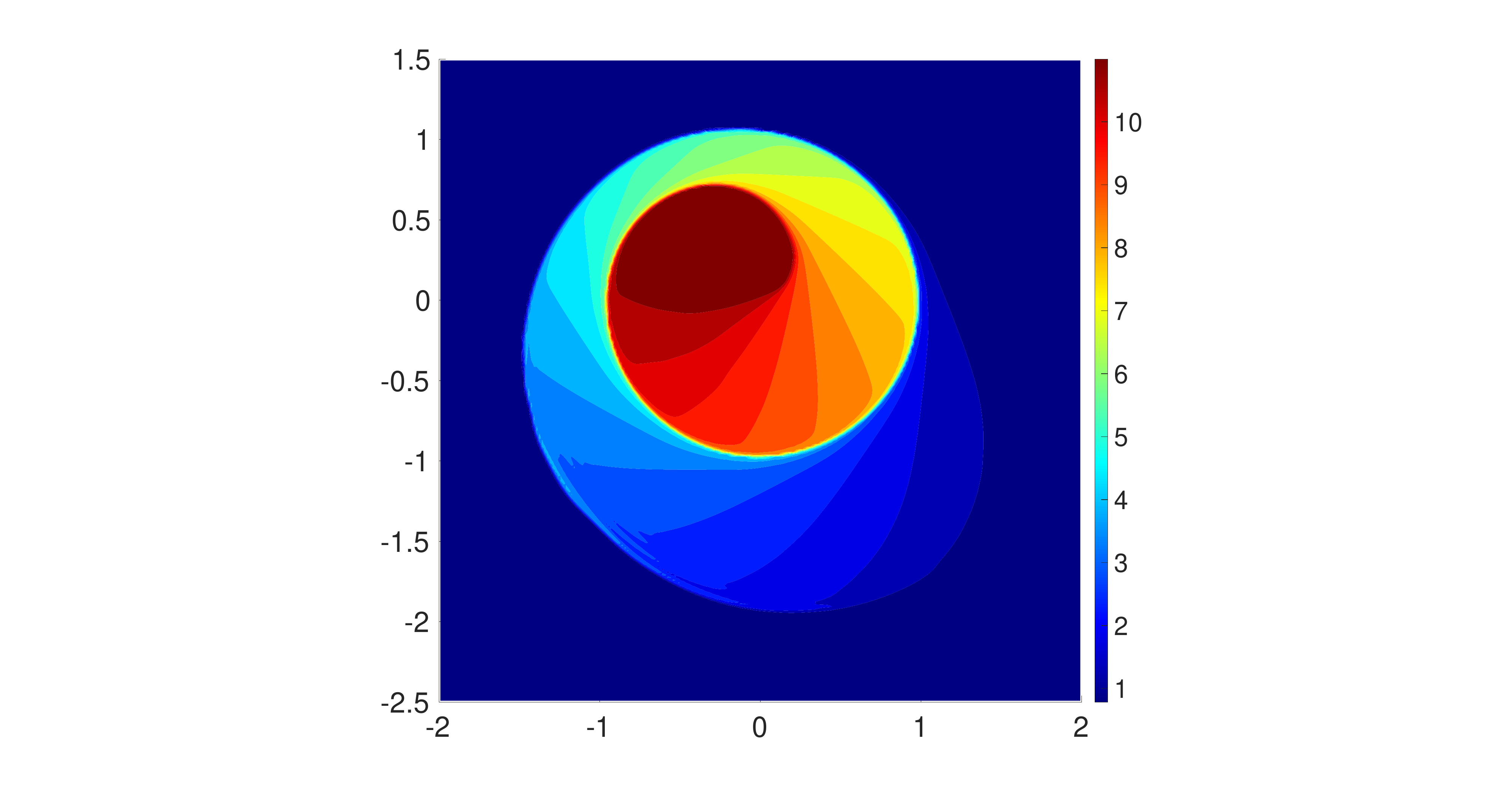}~
    \includegraphics[trim=5.6in .6in 5.8in .6in, clip, width=.5\textwidth]{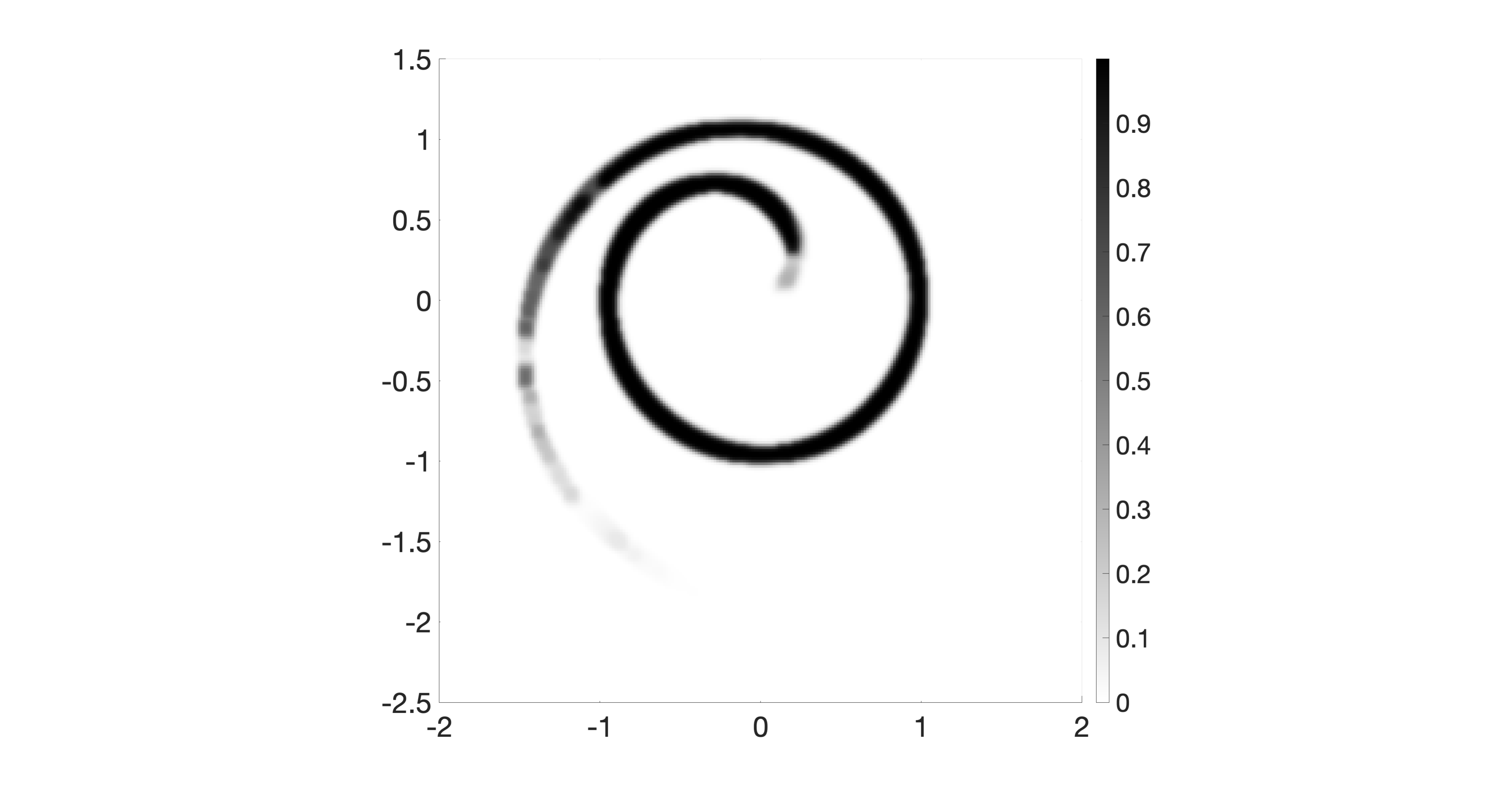}
    \caption{HV solution with $z_0=0.04$.}
    \label{fg:num_2d_kpp_z0d04}
  \end{subfigure} \\
  \begin{subfigure}[b]{.32\textwidth}\centering
    $\quad$
  \end{subfigure}
  \begin{subfigure}[b]{.64\textwidth}\centering
    \includegraphics[trim=5.8in .6in 5.8in .6in, clip, width=.5\textwidth]{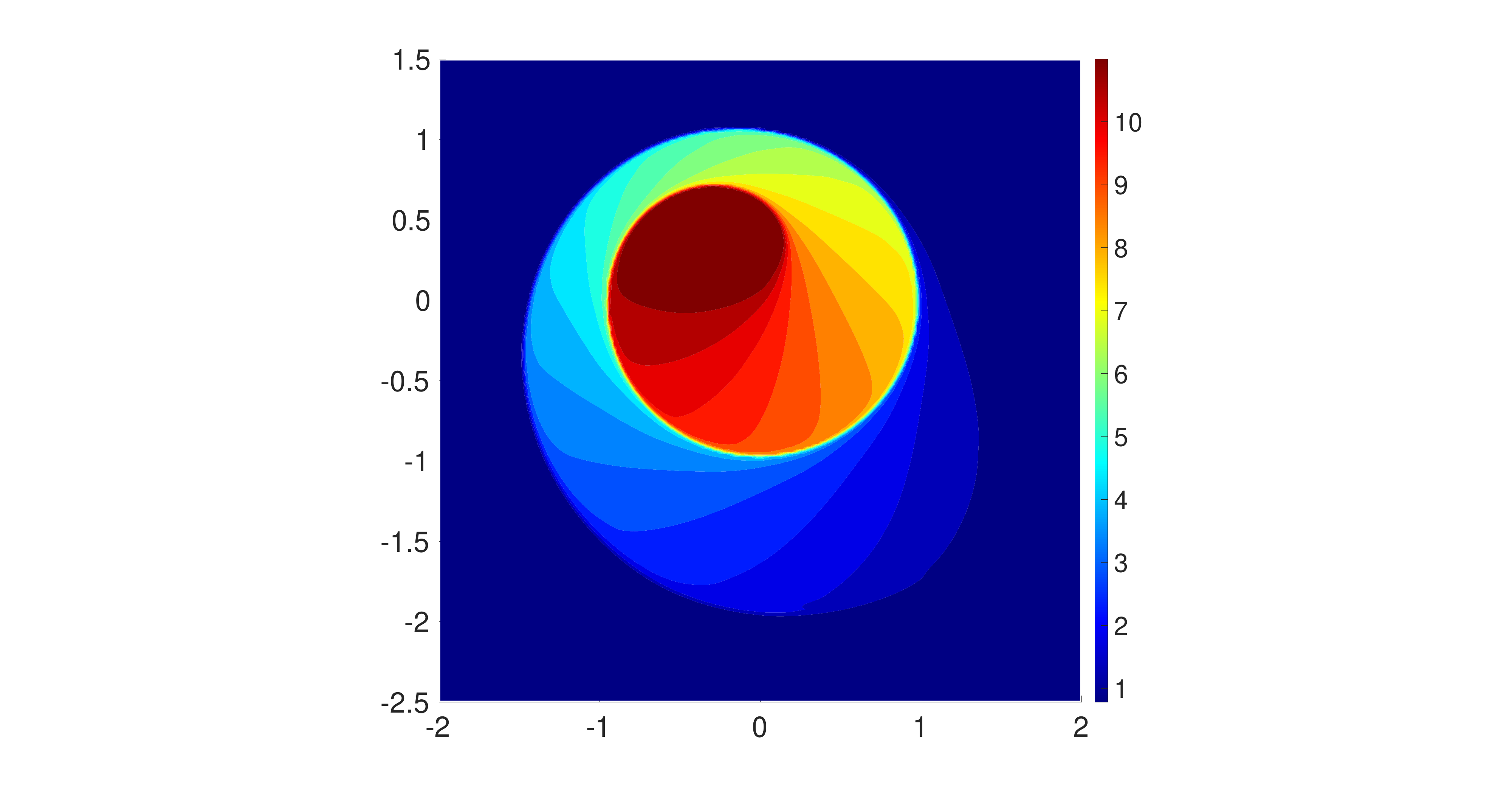}~
    \includegraphics[trim=5.6in .6in 5.8in .6in, clip, width=.5\textwidth]{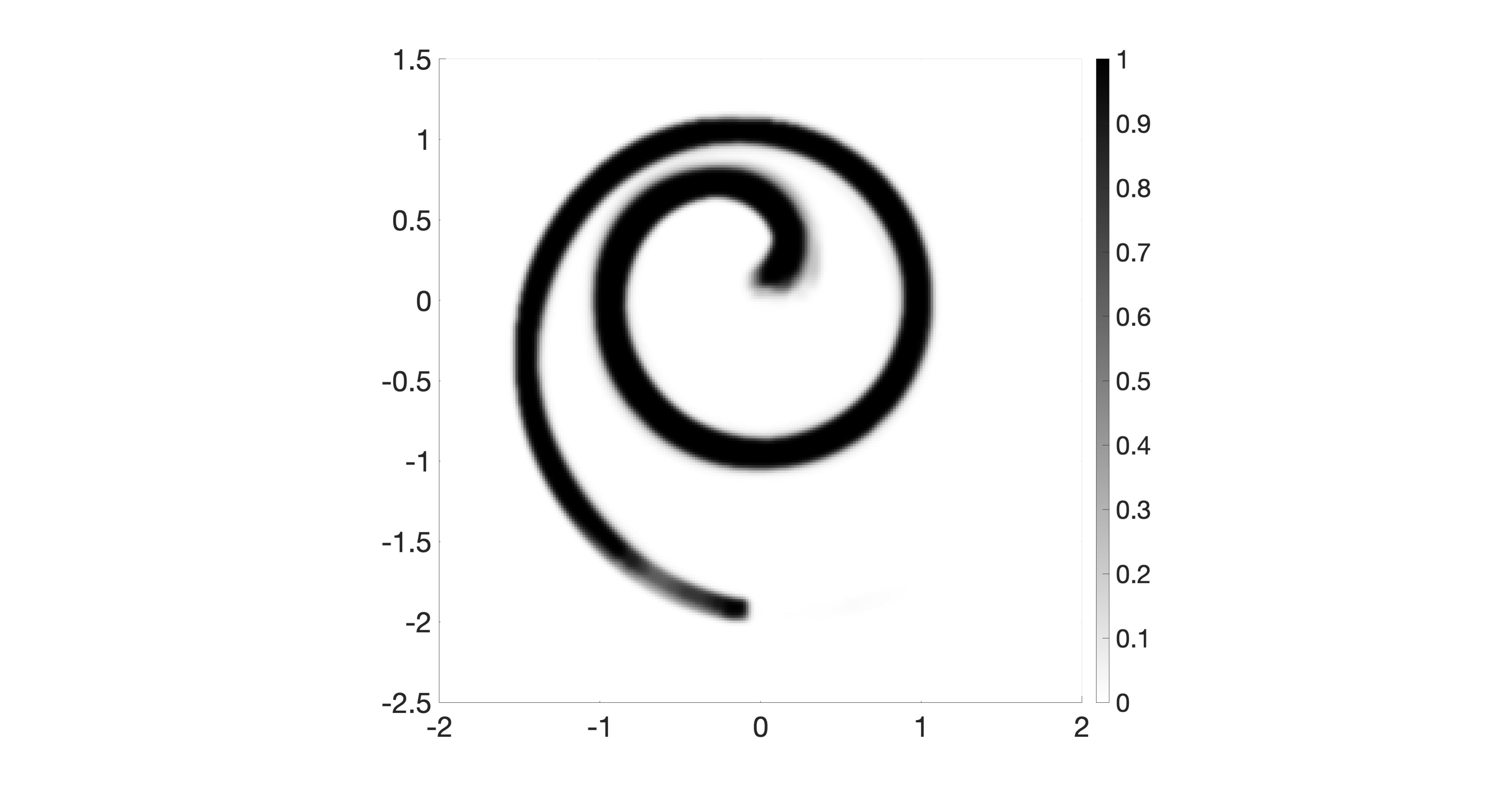}
    \caption{HV solution with $z_0=0.002$.}
    \label{fg:num_2d_kpp_z0d002}
  \end{subfigure}
  \caption{The nonentropic solution of the KPP problem by the MUSCL scheme with superbee limiter (upper row, left panel), the admissible solution and the activation factor of the entropy viscosity of the HV computation with $z_0=0.04$ (upper row middle and right panels), and those of the HV computation with $z_0=0.002$ (lower row).}
  \label{fg:num_2d_kpp}
\end{figure}
Note that although the HV solution computed with $z_0=0.04$ (upper row) is consistent with the entropy one, it still exhibits oscillation along the exterior spiral as the discontinuity gets weaker towards the tail.
We expect to have much better suppression of spurious oscillations if a lower threshold is used, as it is confirmed with the HV solution computed with $z_0=0.002$ (lower row).
The same conclusion can be obtained by comparing the activation factor of the artificial viscosity between these two HV computations.

\subsubsection{Euler equations: Isentropic vortex}
\label{sec:num_2d_vortex}
To assess the accuracy of the method, we consider the isentropic vortex problem~\cite{CWShu:1998a} that is governed by the 2D Euler equation on the square domain $\Omega=[-5,\,5]^2$ with periodic boundary conditions and the initial condition:
\begin{equation}\label{eq:num_2d_vortex_ic}
  \left\{\begin{array}{l}
    \rho(x,y,0) = \left(1-\frac{\epsilon^2}{8}\frac{\gamma-1}{\gamma\pi^2}\exp(1-x^2-y^2)\right)^{1/(\gamma-1)}\;, \\
    v_1(x,y,0) = 1-\frac{\epsilon y}{2\pi}\exp\left(\frac{1}{2}(1-x^2-y^2)\right)\;, \\
    v_2(x,y,0) = 1+\frac{\epsilon x}{2\pi}\exp\left(\frac{1}{2}(1-x^2-y^2)\right)\;, \\
    p(x,y,0) = \left(1-\frac{\epsilon^2}{8}\frac{\gamma-1}{\gamma\pi^2}\exp(1-x^2-y^2)\right)^{\gamma/(\gamma-1)}\;,
  \end{array}\right.
\end{equation}
where $\epsilon=5$ is a parameter that controls the strength of the vortex and the specific heat capacity ratio is $\gamma=1.4$. 
The problem describes a vortex whose center travels with the velocity $(1,\,1)$ across the diagonal of the domain; and the initial condition serves as the reference solution at $T=10.0$.
In~\cref{fg:num_2d_vortex}, we plot the $L_1$-norms of numerical errors in logarithmic scale on a sequence of 6 uniform grids with number of cells ranging from $10\times10$ to $320\times320$ with artificial viscosity activated ($z_0=0.04$).
\begin{figure}\centering
  \includegraphics[trim=.6in .0in 1.6in .6in, clip, width=.48\textwidth]{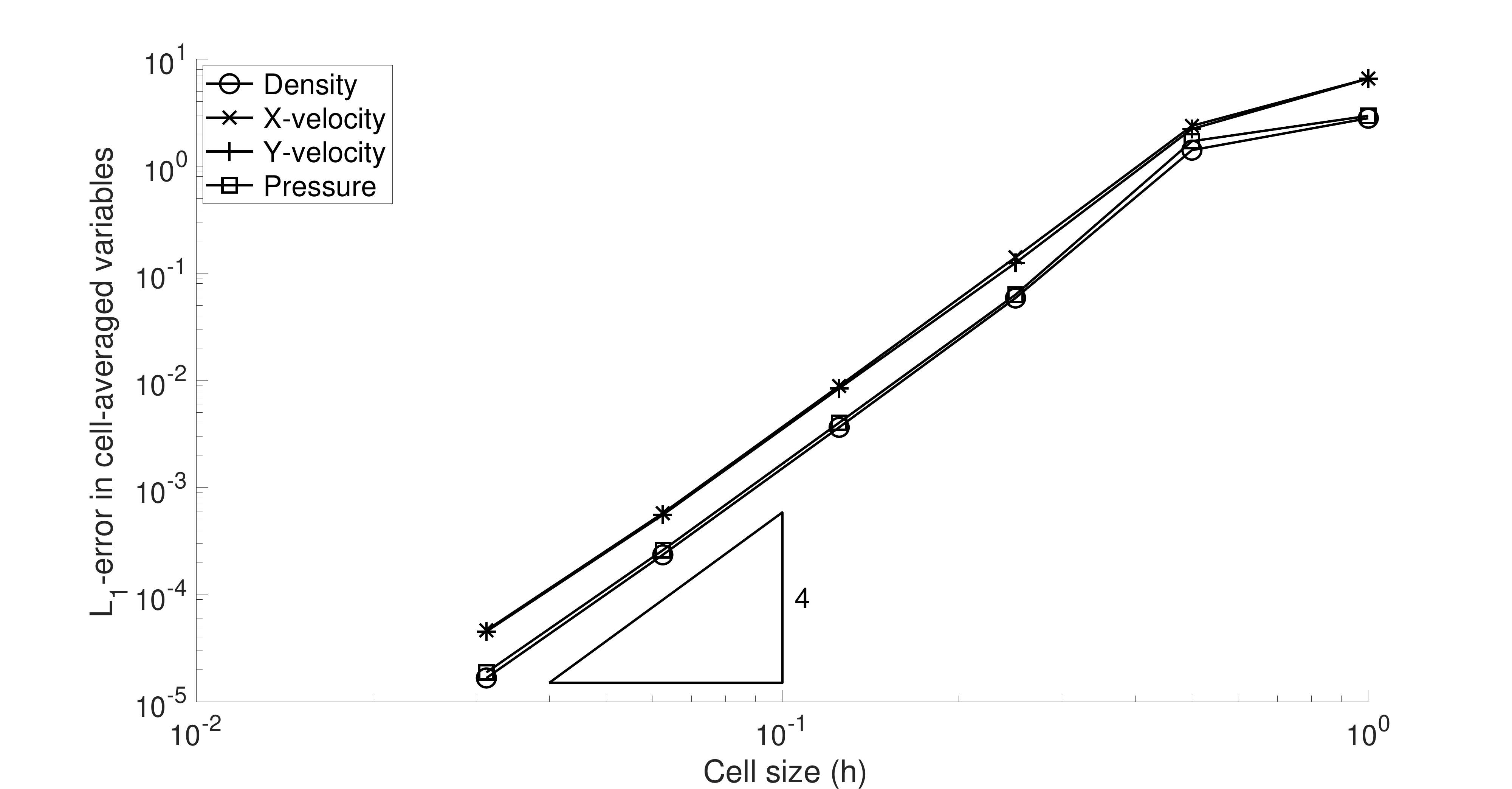} ~~~~
  \includegraphics[trim=.6in .0in 1.6in .6in, clip, width=.48\textwidth]{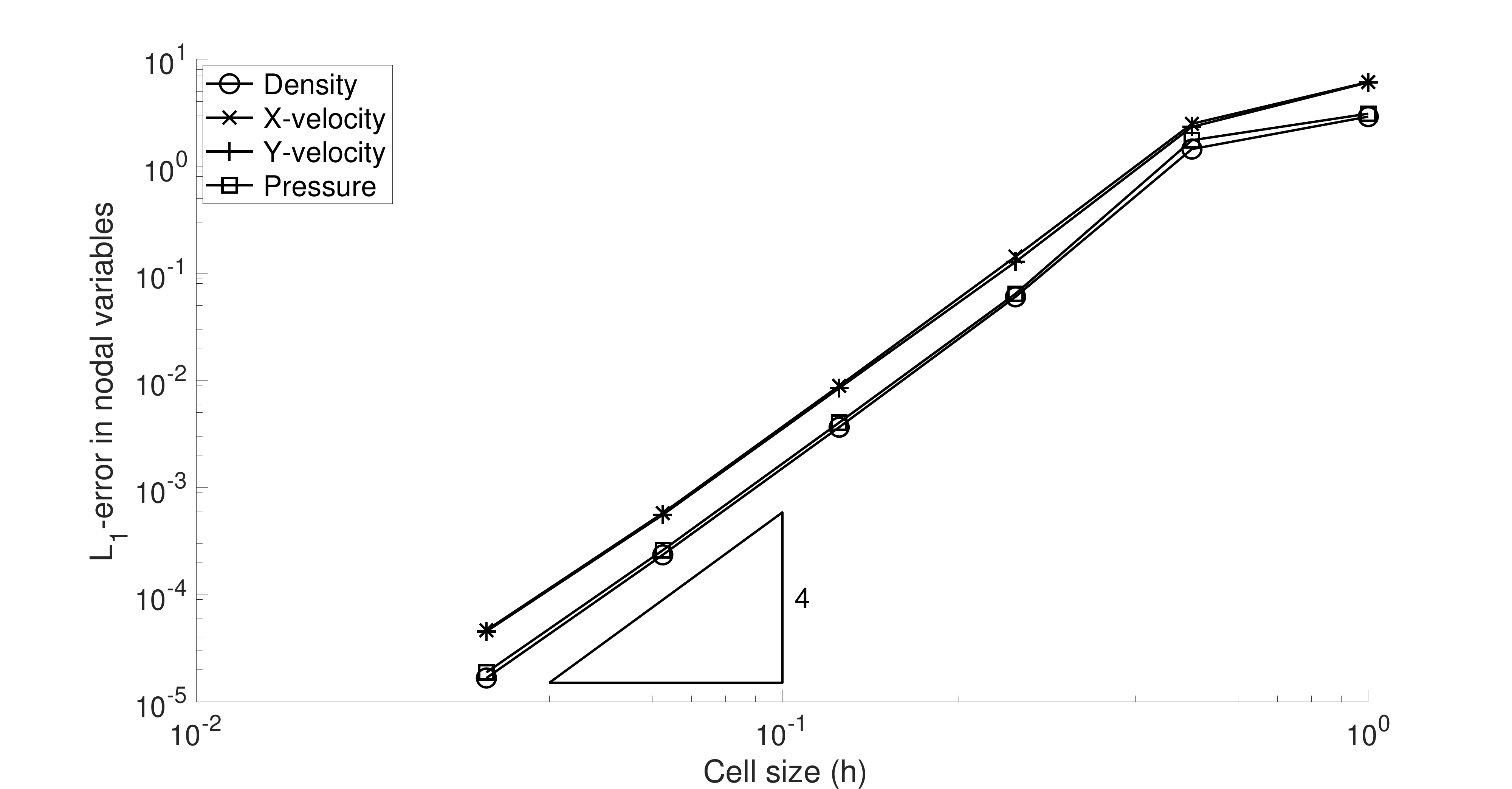}
  \caption{Convergence curves of numerical errors (logarithmic scale) computed in solving the isentropic vortex problem (\cref{sec:num_2d_vortex}).}
  \label{fg:num_2d_vortex}
\end{figure}
We clearly observe the optimal order of convergence once the grid resolution is no coarser than $20\times20$.

\subsubsection{Euler equations: The Taylor-Green vortex}
\label{sec:num_2d_tg}
Next we consider the two-dimensional Taylor-Green vortex problem~\cite{JLGuermond:2016a} that is again governed by the Euler equations with a heat source and specific heat capacity ratio $\gamma=5/3$ on the unit square $\Omega = [0,\,1]^2$:
\begin{align}
  \label{eq:num_2d_tg_eqn}
  \bs{W}_t &+ \nabla\cdot\bs{F}(\bs{W}) + \bs{B} = \bs{0}\;,\quad
  \bs{B} = \left[\begin{array}{c}0 \\ \bs{0} \\ -\rho r\end{array}\right]\;, \\
  \notag
  &r(\bs{x},t) = r(x,y,t) = \frac{3\pi}{8}\left[\cos(3\pi x)\cos(\pi y)-\cos(\pi x)\cos(3\pi y)\right]\;.
\end{align}
With wall boundary conditions on $\partial\Omega$, it admits the smooth and time-independent solution:
\begin{equation}\label{eq:num_2d_tg_sol}
  \rho(\bs{x},t) = 1\,,\ 
  \bs{v}(\bs{x},t) = \left[\begin{array}{c}\sin(\pi x)\cos(\pi y) \\ -\cos(\pi x)\sin(\pi y))\end{array}\right]\,,\ 
  p(\bs{x},t) = \frac{1}{4}\left[\cos(2\pi x)+\cos(2\pi y)\right]+1\;.
\end{equation}
To carry out the computation, the only missing piece is the calculation of cellular entropy residual in appearance of a source term, for which purpose all we need to do is adding to the right hand side of~\cref{eq:ext_ev_res} the term $\pp{s}{\bs{W}}(\overline{\bs{W}}_{\phf{i},\phf{j}})\cdot\bs{B}(\overline{\bs{W}}_{\phf{i},\phf{j}})$.

\begin{figure}\centering
  \includegraphics[trim=.6in .0in 1.6in .6in, clip, width=.48\textwidth]{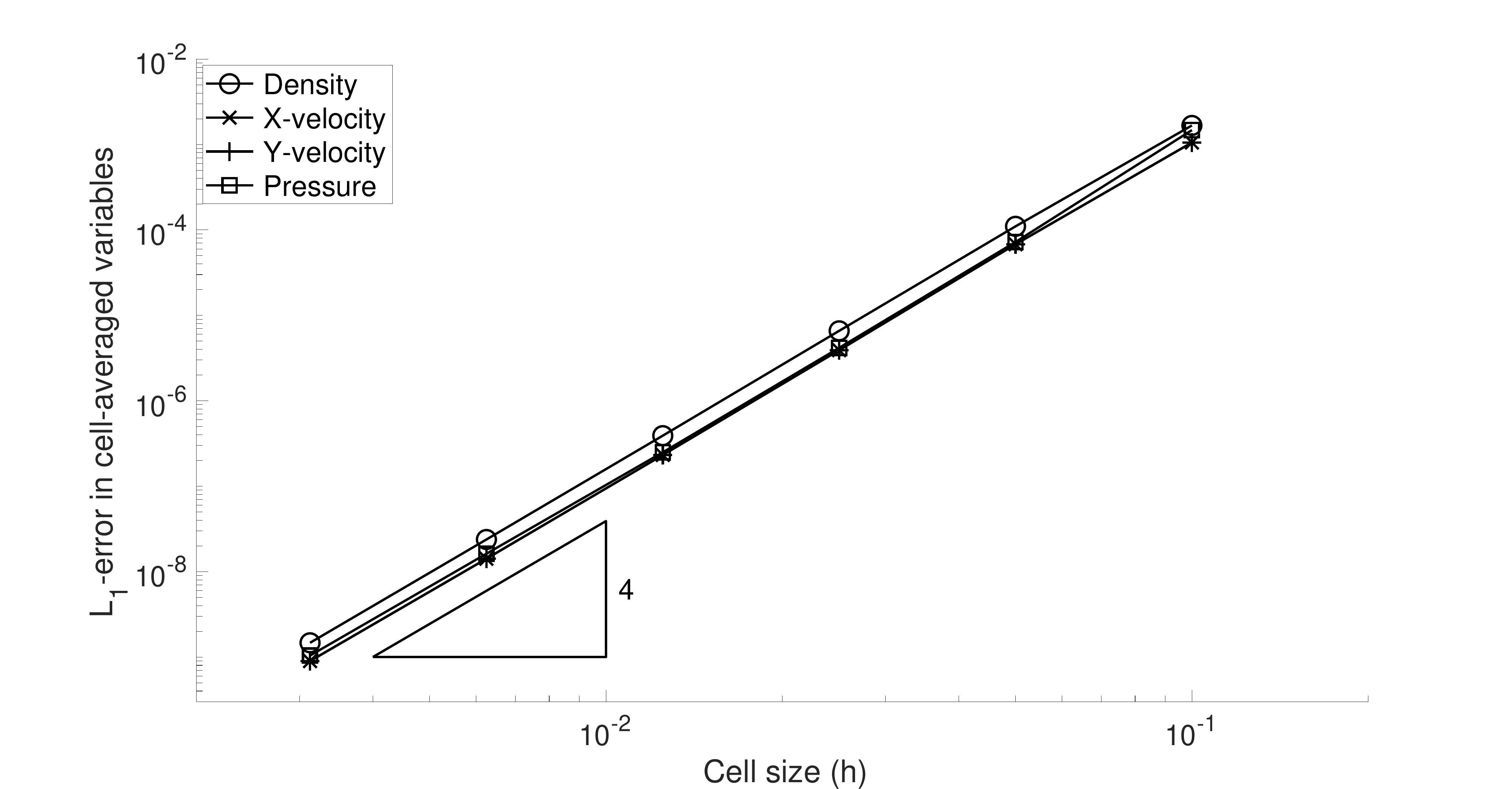} ~~~~
  \includegraphics[trim=.6in .0in 1.6in .6in, clip, width=.48\textwidth]{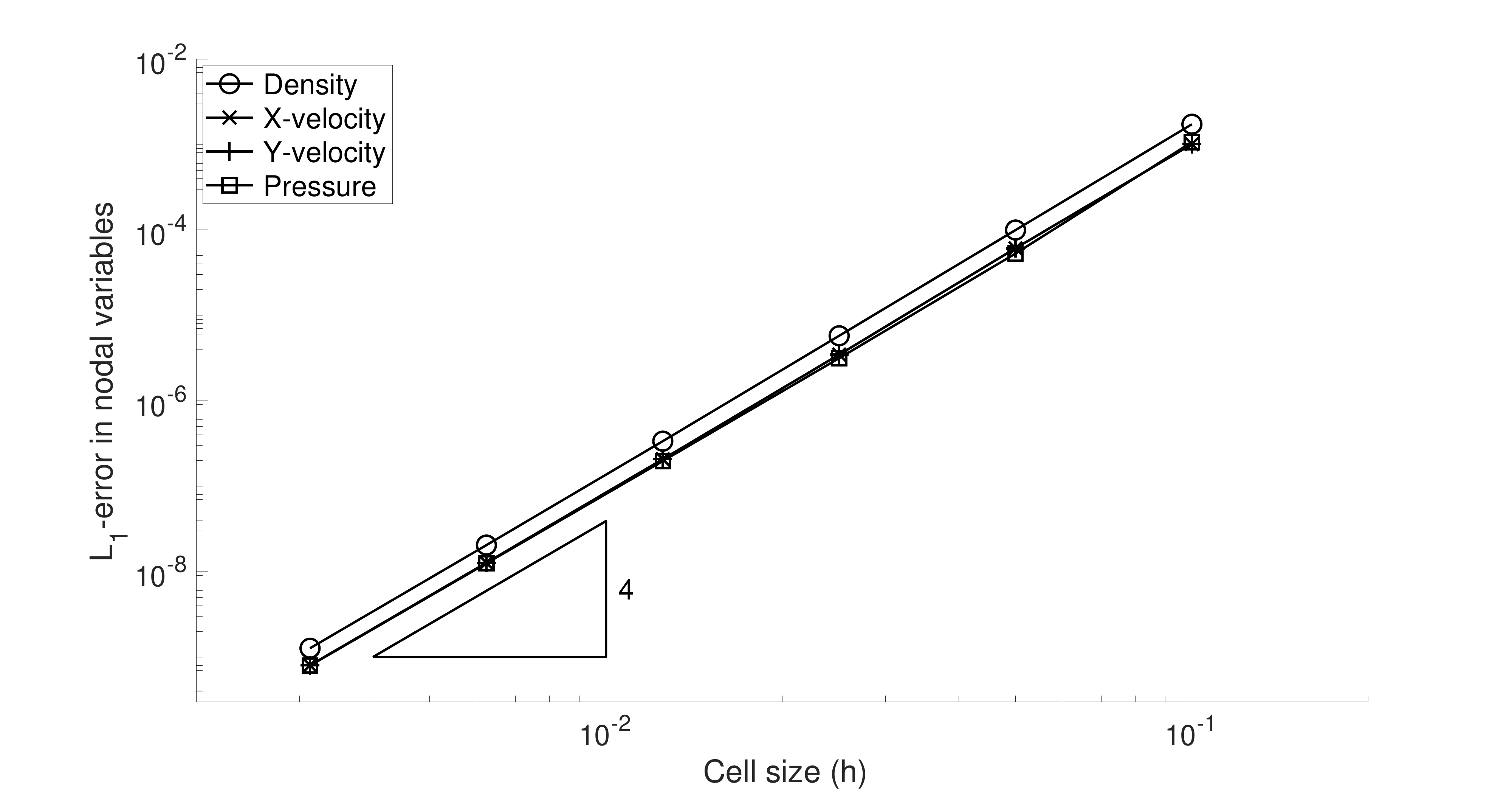}
  \caption{Convergence curves of numerical errors (logarithmic scale) computed in the Taylor-Green vortex problem (\cref{sec:num_2d_tg}).}
  \label{fg:num_2d_tg}
\end{figure}
We assess the numerical performance by solving the problem until $T=0.5$ using a sequence of 6 uniform grids with number of cells ranging from $10\times10$ to $320\times320$ with $z_0=0.04$ in the artificial viscosity computation.
The $L_1$-norms of the numerical errors are plotted in logarithmic scale in~\cref{fg:num_2d_tg}, and we observe the optimal fourth-order convergence in all variables.

\subsubsection{Euler equations: Shock-bubble interaction}
\label{sec:num_2d_sb}
Lastly, we consider a shock-bubble interaction problem motivated by~\cite{MCada:2009a}.
It considers the Euler equations with $\gamma=1.4$ on the rectangular domain $[-0.1,\,1.6]\times[-0.5,\,0.5]$; initially, the problem consists of a vertical shock at $x=0.0$ and a lighter bubble with radius $0.2$ centered at $(0.3,\,0.0)$.
The initial conditions are specified in~\cref{fg:num_2d_sb_setup} along with the boundary conditions at the four edges.
\begin{figure}\centering
  \begin{tikzpicture}[line width=1.6] % x 5 of original coordinates
    \draw [line width=2.4] (-0.5,-2.5) rectangle (8.0,2.5);
    \draw [line width=1.2] (0.0,-2.5) -- (0.0, 2.5);
    \draw [line width=1.2] (1.5,0.0) circle (1.0);
    \draw (0.0,2.5) node [above] {$0.0$};
    \draw (-0.5,-2.5) node [below] {$-0.1$};
    \draw (8.0,-2.5) node [below] {$1.6$};
    \draw (-0.5,-2.5) node [left] {$-0.5$};
    \draw (-0.5, 2.5) node [left] {$0.5$};
    \draw [densely dashed, line width=0.8] (1.5,-1.0) -- (1.5,-2.5) node [below] {$0.3$};
    \draw [densely dashed, line width=0.8] (2.5, 0.0) -- (2.5,-2.5) node [below] {$0.5$};
    \draw (-0.25,0.0) node {\rotatebox{90}{\small$\rho=3.81,\ \bs{v}=(2.58,0.0),\ p=10.0$}};
    \draw (1.5,0.0) node {\small$\begin{array}{l}\rho=0.5\\\bs{v}=(0,0)\\p=1.0\end{array}$};
    \draw (5.25,0.0) node {\small$\begin{array}{l}\rho=1.0\\\bs{v}=(0,0)\\p=1.0\end{array}$};
    \draw (8.25,0.0) node {\rotatebox{90}{\small out-flow}};
    \draw (-0.9,0.0) node {\rotatebox{90}{\small in-flow}};
    \draw (3.75,-2.5) node [below] {\small wall};
    \draw (3.75, 2.5) node [above] {\small wall};
  \end{tikzpicture}
  \caption{Initial setup and boundary conditions for the shock-bubble interaction problem.}
  \label{fg:num_2d_sb_setup}
\end{figure}
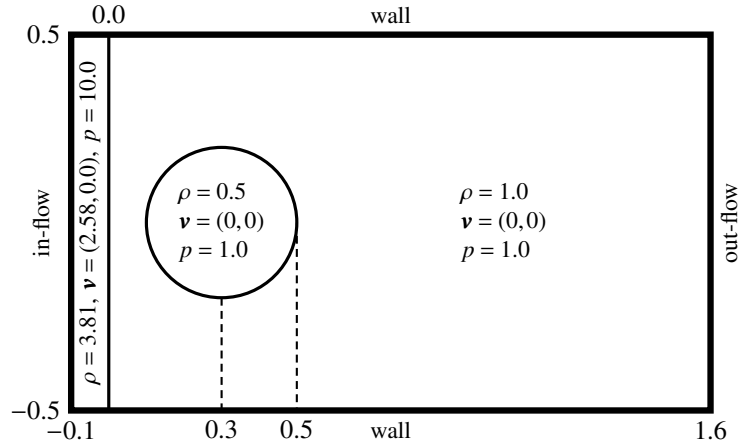
Due to the symmetry of the problem, we only consider the upper half and thus the computational domain is $\Omega = [-0.1,\,1.6]\times[0.0,\,0.5]$.

\begin{figure}\centering
  \begin{subfigure}[b]{\textwidth}\centering
    \includegraphics[trim=1.5in 3in 1.5in 3in, clip, width=.48\textwidth]{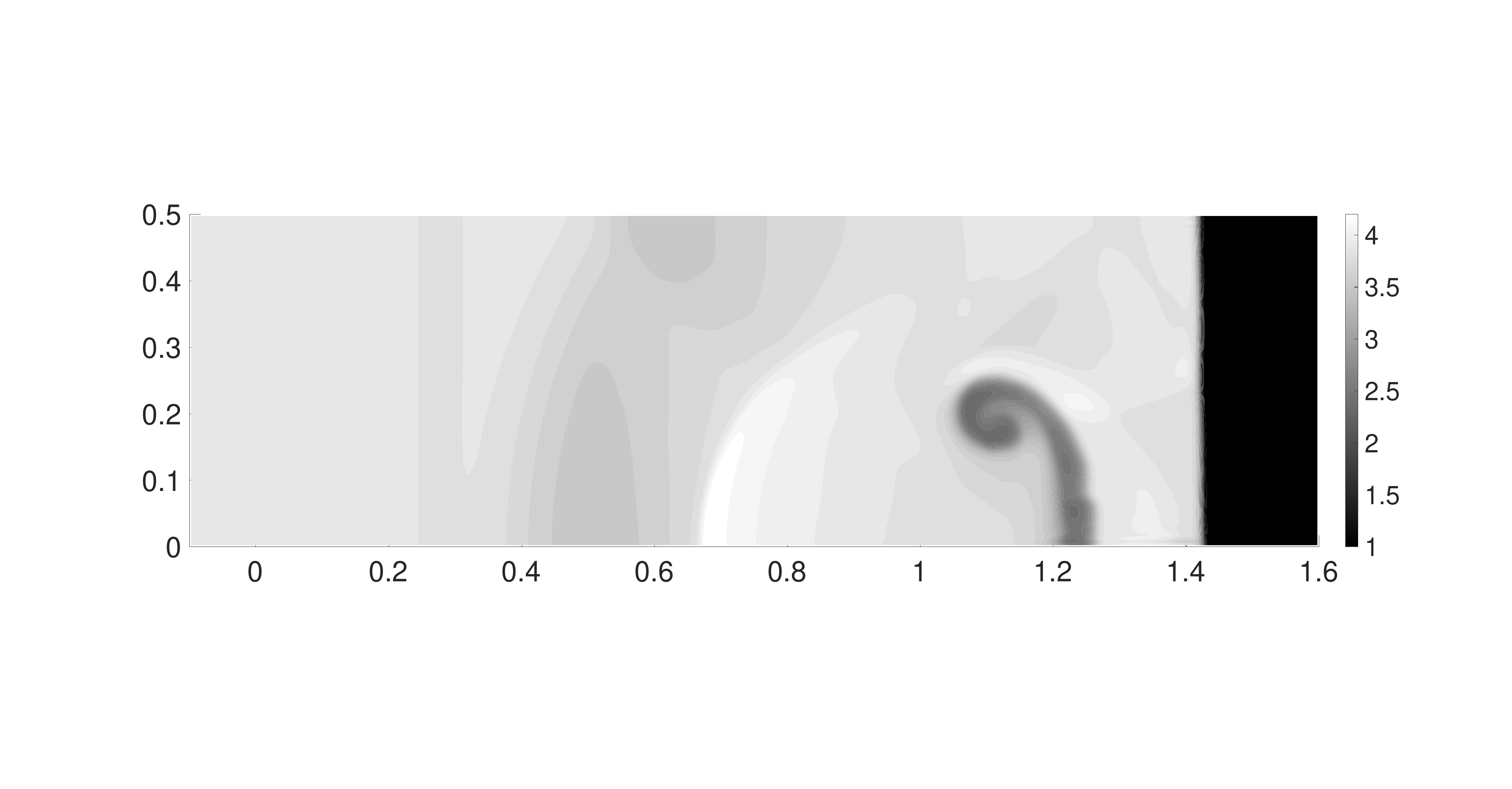}~
    \includegraphics[trim=1.5in 3in 1.5in 3in, clip, width=.48\textwidth]{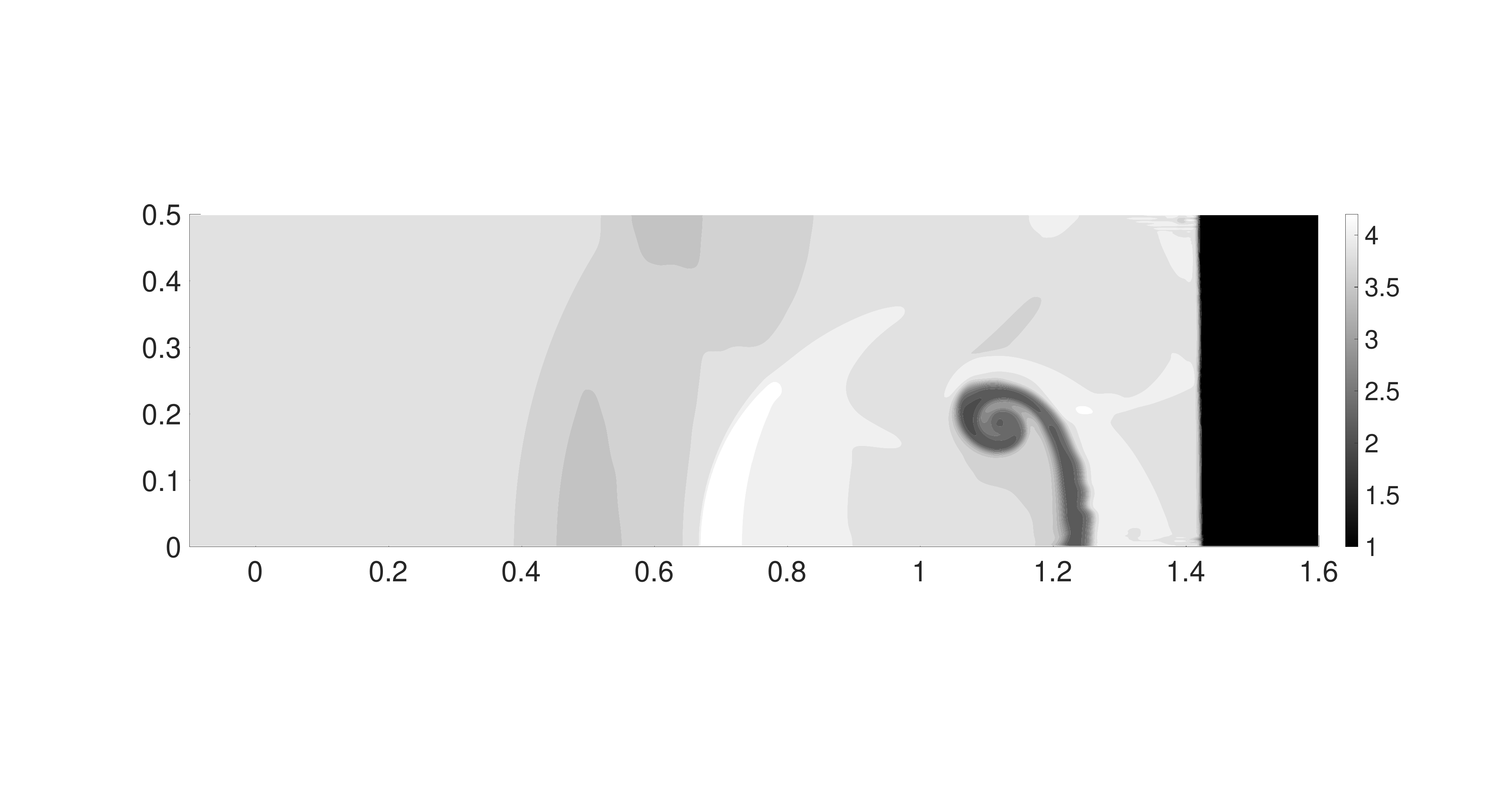}
    \caption{Density computed by MUSCL using: (left) $340\times100$ grid, (right) $680\times200$ grid.}
    \label{fg:num_2d_sb_fvm}
  \end{subfigure}
  \begin{subfigure}[b]{\textwidth}\centering
    \includegraphics[trim=1.5in 3in 1.5in 3in, clip, width=.48\textwidth]{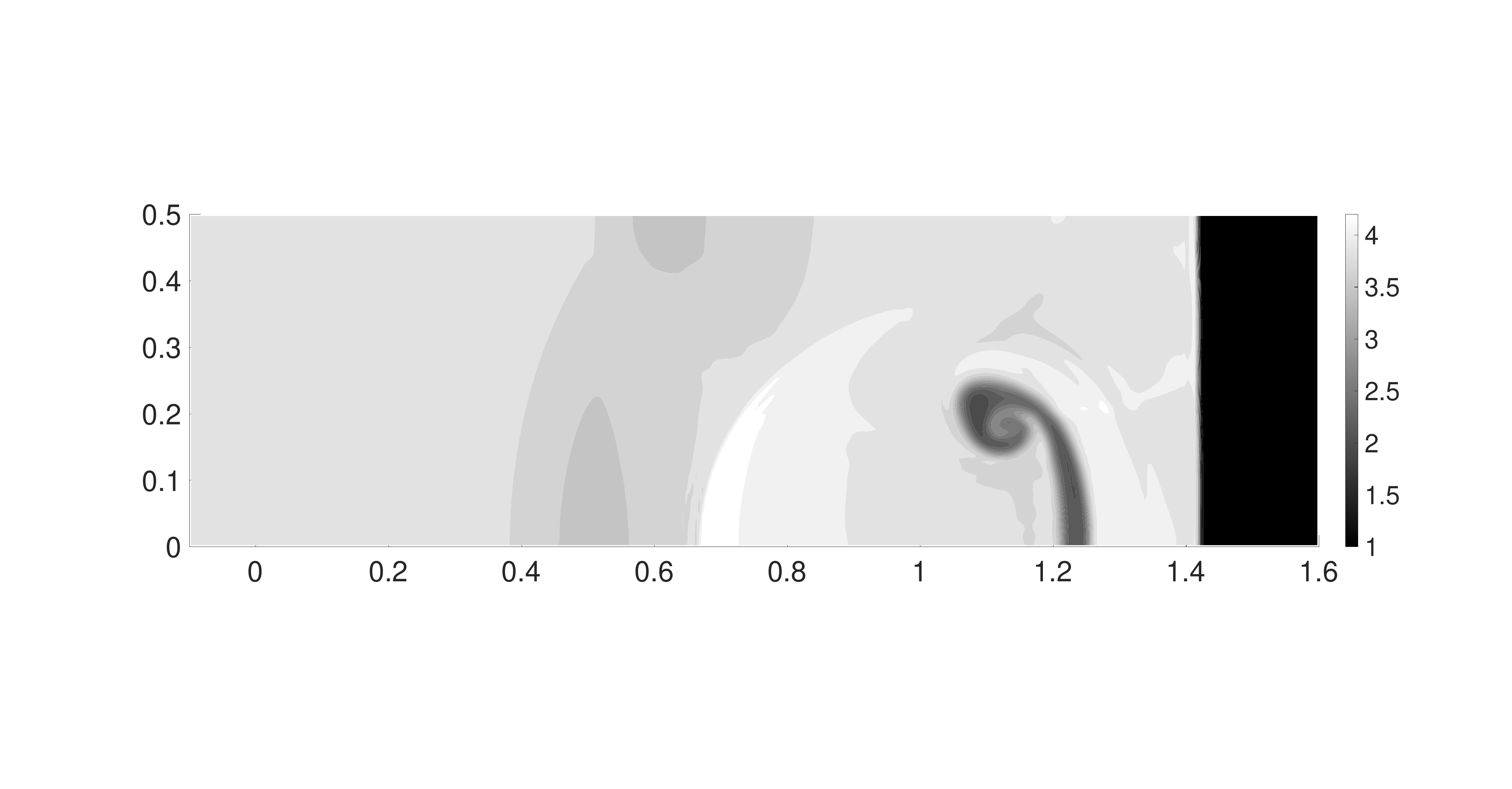}~
    \includegraphics[trim=1.5in 3in 1.5in 3in, clip, width=.48\textwidth]{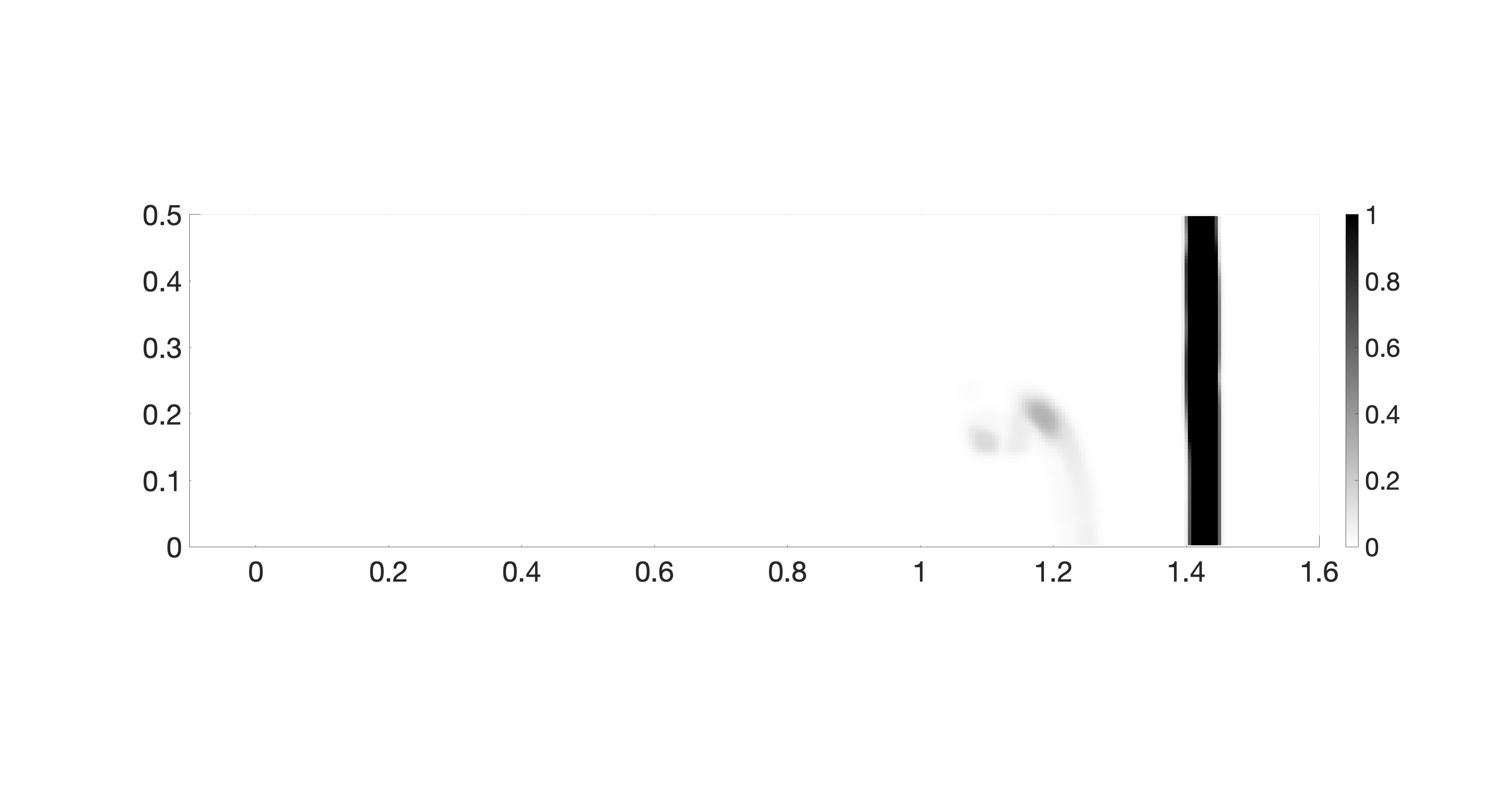}
    \caption{Density (left) and activation factor of viscosity (right) by HV using a $340\times100$ grid.}
    \label{fg:num_2d_sb_hv}
  \end{subfigure}
  \caption{Density contour by the MUSCL scheme with van Albada limiter on two grids (upper row) in comparison with that by the HV scheme on the coarser grid along with the activation factor of the artificial viscosity (lower row).}
  \label{fg:num_2d_sb}
\end{figure}
In the lower row of~\cref{fg:num_2d_sb}, we plot the density at $T=0.4$ computed by the proposed HV method using a $340\times100$ grid along with the activation factor of the artificial viscosity.
For comparison, we plot the density computed by MUSCL with van Albada limiter on the same grid as well as on a finer grid in the upper row of the same figure.
It is clear that HV resolves the vortex due to the bubble much better than MUSCL on the same grid; and MUSCL produces similar solution as the HV using a twice refined grid and twice the number of unknowns comparing to the latter.

\section{Conclusions}
\label{sec:concl}
We presented an explicit discretization method of hybrid type for Euler equations, that is, it seeks approximations to both nodal solutions and cell-averaged solutions and evolves them in time using the method of lines.
While the hybrid-variable (HV) framework is general, the current work focuses on using the smallest stencil (three neighboring cells in 1D or nine cells in 2D) to achieve optimal fourth-order accuracy, utilizing the inherent superconvergence property of the HV methods. 
We proved the convergence of the proposed method for model 1D advection equations given sufficient regularity, and constructed a non-linear artificial viscosity that is based on an entropy residual to capture strong discontinuities that usually appear in the solutions to hyperbolic conservation laws, such as the Euler equations.
Extensive numerical tests are provided to assess the performance of the proposed method, including the handling of various types of boundary conditions, source terms, and the handling of complex discontinuities such as in the challenge KPP problem and shock-bubble interaction.

\section*{Acknowledgements}
The author is supported by the U.S. National Science Foundation (NSF) under Grant DMS-2302080.

\bibliographystyle{plain}      % mathematics and physical sciences
\bibliography{pap}

\begin{thebibliography}{10}

\bibitem{RAbgrall:2019a}
R\'{e}mi Abgrall, Paola Bacigaluppi, and Svetlana Tokareva.
\newblock High-order residual distribution scheme for the time-dependent
  {E}uler equations of fluid dynamics.
\newblock {\em Comput. Methods Appl. Mech. Eng.}, 78(2):274--297, July 2019.

\bibitem{DAppelo:2012a}
Daniel Appel\"{o} and Thomas Hagstrom.
\newblock On advection by {H}ermite methods.
\newblock {\em Pacific Journal of Applied Mathematics}, 4(2):125, April 2012.

\bibitem{JPBoris:1976a}
J.~P. Boris and D.~L. Book.
\newblock Flux-corrected transport. {III}. {M}inimal-error {FCT} algorithms.
\newblock {\em J. Comput. Phys.}, 20(4):397--431, April 1976.

\bibitem{BCockburn:1998a}
Bernado Cockburn and Chi-Wang Shu.
\newblock The {R}unge-{K}utta discontinuous {G}alerkin method of conservation
  laws {V}: {M}ultidimensional systems.
\newblock {\em J. Comput. Phys.}, 141(2):199--224, April 1998.

\bibitem{EFrank:1947a}
Evelyn Frank.
\newblock On the real parts of the zeros of complex polynomials and
  applications to continued fraction expansions of analytic functions.
\newblock {\em T. Am. Math. Soc.}, 62(2):272--283, September 1947.

\bibitem{JGoodrich:2006a}
John Goodrich, Thomas Hagstrom, and Jens Lorenz.
\newblock Hermite methods for hyperbolic initial-boundary value problems.
\newblock {\em Math. Comput.}, 75(254):595--630, 2006.

\bibitem{JLGuermond:2014a}
Jean-Luc Guermond and Bojan Popov.
\newblock Viscous regularization of the {E}uler equations and entropy
  principles.
\newblock {\em SIAM J. Appl. Math.}, 74(2):284--305, 2014.

\bibitem{JLGuermond:2016a}
Jean-Luc Guermond, Bojan Popov, and Vladimir Tomov.
\newblock Entropy-viscosity method for the single material {E}uler equations in
  {L}agrangian frame.
\newblock {\em Comput. Methods Appl. Mech. Eng.}, 300:402--426, March 2016.

\bibitem{EHairer:1993a}
Ernst Hairer, Syvert~P. N{\o}rsett, and Gerhard Wanner.
\newblock {\em Solving Ordinary Differential Equations I: Nonstiff Problems},
  volume~8 of {\em Springer Series in Computational Mathematics}.
\newblock Springer, 2nd edition, 1993.

\bibitem{MMHasan:2023a}
Md~Mahmudul Hasan and Xianyi Zeng.
\newblock A central compact hybrid-variable method with spectral-like
  resolution: {O}ne-dimensional case.
\newblock {\em J. Comput. Appl. Math.}, 421:114894, March 2023.

\bibitem{GSJiang:1996a}
Guang-Shan Jiang and Chi-Wang Shu.
\newblock Efficient implementation of weighted {ENO} schemes.
\newblock {\em J. Comput. Phys.}, 126(1):202--228, June 1996.

\bibitem{CJohnson:1990a}
Claes Johnson, Anders Szepessy, and Peter Hansbo.
\newblock On the convergence of shock-capturing streamline diffusion finite
  element methods for hyperbolic conservation laws.
\newblock {\em Math. Comput.}, 54(189):107--129, January 1990.

\bibitem{AKurganov:2007a}
Alexander Kurganov, Guergana Petrova, and Bojan Popov.
\newblock Adaptive semidiscrete central-upwind schemes for nonconvex hyperbolic
  conservation laws.
\newblock {\em SIAM J. Sci. Comput.}, 29(6):2381--2401, 2007.

\bibitem{JVonNeumann:1950a}
J.~Von Neumann and R.~D. Richtmyer.
\newblock A method for the numerical calculation of hydrodynamic shocks.
\newblock {\em J. Appl. Phys.}, 21(3):232--237, March 1950.

\bibitem{NCNguyen:2012a}
N.~C. Nguyen and J.~Peraire.
\newblock Hybridizable discontinuous {G}alerkin methods for partial
  differential equations in continuum mechanics.
\newblock {\em J. Comput. Phys.}, 231(18):5955--5988, September 2012.

\bibitem{PLRoe:1981a}
P.~L. Roe.
\newblock Approximate {R}iemann solvers, parameter vectors, and difference
  schemes.
\newblock {\em J. Comput. Phys.}, 43(2):357--372, October 1981.

\bibitem{PLRoe:1986b}
P.~L. Roe.
\newblock Characteristic-based schemes for the {E}uler equations.
\newblock {\em Annu. Rev. Fluid Mech.}, 18:337--365, 1986.

\bibitem{VVRusanov:1962a}
V.~V. Rusanov.
\newblock The calculation of interaction of non-steady shock waves with
  obstacles.
\newblock {\em USSR Comput. Math. Math. Phys.}, 1(2):304--320, 1962.

\bibitem{FShakib:1991a}
Farzin Shakib, Thomas J.~R. Hughes, and Zden\u{e}k Johan.
\newblock A new finite element formulation for computational fluid dynamics: X.
  {T}he compressible {E}uler and {N}avier-{S}tokes equations.
\newblock {\em Comput. Methods Appl. Mech. Eng.}, 89(1--3):141--219, August
  1991.

\bibitem{CWShu:1998a}
Chi-Wang Shu.
\newblock Essentially non-oscillatory and weighted essentially non-oscillatory
  schemes for hyperbolic conservation laws.
\newblock In Alfio Quarteroni, editor, {\em Advanced Numerical Approximation of
  Nonlinear Hyperbolic Equations}, volume 1697 of {\em Lecture Notes in
  Mathematics}, pages 325--432. Springer Berlin Heidelberg, 1998.

\bibitem{GSod:1978a}
Gary~A. Sod.
\newblock A survey of several finite difference methods for systems of
  nonlinear hyperbolic conservation laws.
\newblock {\em J. Comput. Phys.}, 27(1):1--31, April 1978.

\bibitem{VStiernstrom:2021a}
Vidar Stiernstr\"{o}m, Lukas Lundgren, Murtazo Nazarov, and Ken Mattsson.
\newblock A residual-based artificial viscosity finite difference method for
  scalar conservation laws.
\newblock {\em J. Comput. Phys.}, 430:110100, April 2021.

\bibitem{GDvanAlbada:1982a}
G.~D. van Albada, B.~van Leer, and Jr. W.~W.~Roberts.
\newblock A comparative study of computational methods in cosmic gas dynamics.
\newblock {\em Astron. Astrophys.}, 108(1):76--84, April 1982.

\bibitem{BvanLeer:1979a}
Bram van Leer.
\newblock Towards the ultimate conservative difference scheme {V}. {A}
  second-order sequel to {G}odunov's method.
\newblock {\em J. Comput. Phys.}, 32(1):101--136, July 1979.

\bibitem{MCada:2009a}
Miroslav \v{C}ada and Manuel Torrilhon.
\newblock Compact third-order limiter functions for finite volume methods.
\newblock {\em J. Comput. Phys.}, 228(11):4118--4145, June 2009.

\bibitem{STZalesak:1979a}
Steven~T. Zalesak.
\newblock Fully multidimensional flux-corrected transport algorithm for fluids.
\newblock {\em J. Comput. Phys.}, 31(3):335--362, June 1979.

\bibitem{XZeng:2014a}
Xianyi Zeng.
\newblock A high-order hybrid finite difference-finite volume approach with
  application to inviscid compressible flow problems: {A} preliminary study.
\newblock {\em Comput. Fluids}, 98:91--110, July 2014.

\bibitem{XZeng:2019a}
Xianyi Zeng.
\newblock Linear hybrid-variable methods for advection equations.
\newblock {\em Adv. Comput. Math.}, 45(2):929--980, 2019.

\end{thebibliography}

%\appendix

\end{document}